\documentclass[12pt]{article}

\usepackage[T1,T2A]{fontenc}
\usepackage[utf8]{inputenc}

\usepackage{latexsym}
\usepackage{mathrsfs}
\usepackage[all]{xy}
\usepackage{bbm}
\usepackage{xcolor}
\usepackage[mathscr]{euscript}
 
\usepackage{amsfonts}\usepackage{amssymb}
\usepackage{amsmath}
\usepackage{amsthm}
\usepackage{hyperref}
\usepackage{subfiles}

\usepackage{tikz}
\usepackage{tikzscale}
\usetikzlibrary{calc, angles, quotes}

\usepackage{blindtext}

\usepackage{comment}

\usepackage{titlesec}
\usepackage{stackengine}
\stackMath

\usepackage{amsmath,amsthm,amssymb,tikz,graphicx,adjustbox,caption,nicefrac}
\usepackage{tabularx}
\usetikzlibrary{quotes,positioning,automata}

\NewDocumentCommand{\vdotss}{ O{0.35em} O{0.8pt} m }{%
  \begin{tikzpicture}[baseline={(0, -0.5ex)}]
    \pgfmathsetmacro{\numdots}{#3 - 1}
    \foreach \i in {0, 1, ..., \numdots}{%
      \fill (0, {-\i * #1}) circle [radius=#2];
    }
    \path (0, 0.3ex) -- (0, {-\numdots * #1 - 0.3ex});
  \end{tikzpicture}%
}

\newcommand\eq[2]{
\begin{equation}
\label{eq:#1}
{#2}
\end{equation}
}
\newcommand{\equ}[1]{\eqref{eq:#1}}

\newcommand{\ignore}[1]{}
\newcommand{\UA}{\operatorname{U\mspace{-3mu}A}}
\newcommand{\conv}{\operatorname{Conv}}
\newcommand{\dist}{\operatorname{dist}}
\newcommand{\GL}{\operatorname{GL}}
\newcommand{\Span}{\operatorname{Span}}
\newcommand{\vp}{{\bf p}}
\newcommand{\vq}{{\bf q}}
\newcommand{\vb}{{\bf b}}

\newcommand{\nz}{\smallsetminus\{0\}}
\newcommand\mr{M_{m,n}}

\newtheorem{theorem}{Theorem}[section]
\newtheorem{lemma}[theorem]{Lemma}

\newtheorem{corollary}[theorem]{Corollary}

\newtheorem{observation}[theorem]{Observation}
\newtheorem{proposition}[theorem]{Proposition}

\newtheorem{thm}{Theorem}

\theoremstyle{definition}
\newtheorem{definition}[theorem]{Definition}
\newtheorem{example}[theorem]{Example}
\theoremstyle{remark}
\newtheorem{remark}[theorem]{Remark}
\usepackage{enumitem}

\newcommand{\R}{{\mathbb{R}}}
\newcommand{\Z}{{\mathbb{Z}}}
\newcommand{\N}{{\mathbb{N}}}

\newcommand{\discrete}{{discrete}}
\newcommand{\locallyfinite}{{locally finite}}

 \newcommand{\tmn}{{\bf T}_{m,n}}

\newcommand{\ba}{{\bf a}}
\newcommand{\q}{{\bf q}}
\newcommand{\p}{{\bf p}}

\newcommand{\bv}{{\bf v}}
\newcommand{\bx}{{\bf x}}
\newcommand{\by}{{\bf y}}
\newcommand{\br}{{\bf r}}
\newcommand{\be}{{\bf e}}
\newcommand{\bfL}{{L}}

\newcommand{\pr}{\pi}

\newcommand{\fR}{E}
\newcommand{\fQ}{{\Theta}}

\newcommand{\alf}{{\underline{\alpha}}}
\newcommand{\betf}{{\underline{\beta}}}

\numberwithin{equation}{section}

\def \le {\leqslant}
\def \ge {\geqslant}
\newcommand {\comm}[1]   {\textcolor{purple}{#1}}

\usepackage{tikz} 
\usetikzlibrary{calc,arrows,decorations.pathreplacing,fadings,3d,positioning}

\makeatletter
\def\namedlabel#1#2{\begingroup
    #2%
    \def\@currentlabel{#2}%
    \phantomsection\label{#1}\endgroup
}
\makeatother

\title{%On Sufficient Conditions for Total Density
Uniform Diophantine approximation with restrictions\\ via {total} density of collections of subspaces}

\author{Leo Hong \\
        { \small Department of Mathematics, UNCC, Charlotte, NC 28223, USA} \\
        {\small \texttt{lhong6@charlotte.edu}} \\
        Dmitry Kleinbock \\
        { \small Department of Mathematics, Brandeis University, Waltham, MA 02453, USA} \\
        {\small \texttt{kleinboc@brandeis.edu} }\\
  Vasiliy Neckrasov \\
  {\small Department of Mathematics, Brandeis University, Waltham, MA 02453, USA} \\
  {\small \texttt{vneckrasov@brandeis.edu}}
}
\date{}

\begin{document}

\maketitle

\begin{abstract} In 1926 Khintchine introduced a topological argument proving the existence of uncountably many nontrivial singular linear forms of $n \geq 2$ variables. Throughout the years, this argument {has been} extensively modified and generalized. Most recently, Kleinbock {et al.\ (2025)} introduced a general framework of Diophantine systems  and %it was 
{showed} that a certain topological property called total density implies a far-reaching generalization of Khintchine's result. We describe a way to establish total density for a variety of Diophantine systems, and thus prove that the sets of singular objects are uncountable and dense in a wide range of set-ups in Diophantine approximation. As a special case, we establish {such a} result for inhomogeneous approximation, {proving {the} existence of uncountably many singular systems of affine forms with a fixed translation part. One can also consider} approximation with  prime denominators, or more generally, approximation under some strong restrictions on numerators and   denominators.
\end{abstract}

\section{Introduction}\label{intro}

{
\it Throughout the text we mark our original theorems by numbers ({such as} Theorem 1.6, Theorem 2.1) and previously known ones by letters ({such as} Theorem A, Theorem B).}

\subsection{{The Khintchine--Jarnik Theorem and total density}}\label{beginning}

%!
{Throughout the paper we denote the supremum norm on $\R^k$, $k\in\N$, by $| \cdot |$, and  the Euclidean norm by $\|\cdot\|$.  $M_{m,n}$
% = M_{m,n}(\mathbb{R})$ 
will denote} the set of $m \times n$ real matrices. This paper is about {\it uniform} Diophantine approximation; that is, for a matrix $A \in M_{m,n}$ and a function $\psi$ we are interested in whether or not the system of inequalities
\begin{equation}\label{mainsystem}
%{{|A\vq   - \vp |
%\le  \psi(t)}
%  \ \ \ \mathrm{and}  
%\ \ |\vq|
%\le  t}
\begin{cases}
     | A{\q}-{\p}|\le \psi(t)\\ 
     \qquad\ \ | {\q}|\le t
\end{cases}
\end{equation}
has solutions $ {(\q, \p) \in (\mathbb{Z}^n\nz) \times \mathbb{Z}^m}$ for all {{sufficiently large} $t > 0$}. The question motivates the following definition:

\begin{definition}\label{UA-matrices}
    For a %non-increasing
     function $\psi:\R_{> 0}\to \R_{> 0}$ and a matrix $A\in M_{m,n}%(\mathbb{R})
    $, we say {that} $A$ is {\underbar{$\psi$-uniformly} \underbar{approximable}, or  \underbar{$\psi$-uniform} for brevity}, if for all {sufficiently large} $t$ there exist ${\bf q}\in\mathbb{Z}^n\nz$ and ${\bf p}\in\mathbb{Z}^m$ such that \eqref{mainsystem} holds. We denote the set of $\psi$-uniform $m\times n$ matrices by $\UA_{m,n}(\psi)$.
\end{definition}

We note that a similar notion is also sometimes called $\psi$-Dirichlet in the literature; see \cite{KW18} for the case $m=n=1$ and \cite{KW19} for the general case\footnote{{There and in the later papers \cite{KK, KSY,SY} instead of \eqref{mainsystem} a modification %\begin{equation*}\label{mainsystem}
$\begin{cases}
     | A{\q}-{\p}|^m < \psi(t)\\ 
     \qquad\ \ | {\q}|^n < t
\end{cases}$
%"${{|A\vq   - \vp |^m
%<  \psi(t)}}$ and  
%$|\vq|^n
%<  t$"
%\begin{cases}
%     | A{\q}-{\p}|\le \psi(t)\\ 
%     | {\q}|\le t
%\end{cases}
%\end{equation*} 
is used.}}. {We have decided to follow the notation and terminology of \cite{KMWW}, on which a substantial portion of our approach is based.}

%Throughout the paper we use a convenient notation: $$\psi_a(t) = %\frac{1}
%{t^{-a}}.$$

The classical Dirichlet's theorem states that every $m \times n$ matrix is $\psi_{{{n}/{m}}}$-uniform, {where} $${\psi_{{{n}/{m}}}(t) = %\frac{1}
{t^{-n/m}};}$$ this result is the best possible
{in the sense that the set $\UA_{m,n}( c \psi_{{{n}/{m}}})$  has Lebesgue measure zero for any $c<1$}
%,  that is, one can not replace $\psi_{\frac{n}{m}}$ with a quicker decreasing function 
(see \cite{DS1, DS2}). {Matrices that are $c \psi_{{{n}/{m}}}$-uniform for all $c>0$ are called  \underbar{singular}.} 
%Much less is known about the sets of $\psi$-uniform matrices if $\psi$ is different from $\psi_{\frac{n}{m}}
%\comm{We never seem to use this notation afterwards, or do we? -- we use it later to talk about the case $n=1$.}

It is clear that some matrices are $\psi$-uniform for trivial reasons, no matter what $\psi$ is: this happens when $ A{\bf q}\in \mathbb{Z}^m$ for some non-zero $ \q \in  \mathbb{Z}^n$; that is, the columns of $A$ are linearly dependent {over $\mathbb{Q}$ together with rational $m$-vectors}. Following \cite{MN25}, we will call such matrices $A$  \underbar{trivially singular}. 
A weaker and related property is the one of a matrix being not totally irrational: we call a matrix  \underbar{totally irrational} if it does not belong to any rational affine subspace of $M_{m,n}$; it is clear that any trivially singular matrix is not totally irrational, and if $m = 1$, these two notions coincide.

A natural question arises from these definitions: given a {positive} function $\psi$,
% (decreasing faster than $\psi_{\frac{n}{m}}$), 
do there {always} exist $\psi$-uniform $m \times n$ matrices {that} are not trivially singular, or even totally irrational? In the simplest case $m=n=1$ the answer is negative:  namely, it is due to Khintchine \cite{H1} that the only singular numbers are the {rational numbers}.
However, if $\max(m,n) > 1$, {the situation is different: Khintchine \cite{H1}, and then  Jarn\'{\i}k \cite{J} in bigger generality, showed
that  there exist
uncountably many totally
irrational singular matrices. More generally, the following is true:}
% which are  {\sl },not trivially singular and even totally irrational $\psi$-uniform matrices exist for a variety of functions $\psi$. In \cite{J}, Jarn\'{i}k proved that the sets of $\psi$-uniform matrices are big:
\begin{thm}\label{Jarnik}
    Let $m,n\in\mathbb{N}$.% where $n>1$.

    \begin{itemize}
    \item[\rm (1)] If $n > 1$, then for any non-increasing function $\psi:\,\,\, \R_{> 0}\to\R_{> 0}$ the set of totally irrational $\psi$-uniform $m\times n$ matrices is uncountable and dense in $M_{m,n}$;
 \item[\rm (2)]  If  $n = 1$ and one assumes in addition that $\lim\limits_{t \rightarrow \infty} t\psi(t) = \infty$, then for any $m > 1$ the set of totally irrational $\psi$-uniform $m\times 1$ matrices (column vectors) is uncountable and dense in ${M_{m,1}}$.    
\end{itemize}
\end{thm}

{In recent years there have been several extensions and modifications of the Khintchine--Jarn\'{\i}k construction. Most recently, in \cite{KMWW} the following general definition was introduced. 
\begin{definition}\label{diophsys}
    A  \underbar{Diophantine system}  is
 a triple $\mathcal{X} = \left( X, \mathcal{D}, \mathcal{H}\right)$,
  where  $X$ is a topological space, 
  $\mathcal{D}= \{d_s : s\in  I\}$ is a sequence of continuous functions on $X$ taking nonnegative values ({\sl generalized distance functions}), and     $\mathcal{H}= \{h_s: s\in  I\}$ is a sequence of positive real numbers; we will say that $h_s$ is the {\sl height} associated with %resonant sets, so that $h_\alpha$ is the height of $R_\alpha$.
    $d_s$.
   Here $I$ is a fixed countable set
  %\footnote{{For our exposition it will be convenient to have a freedom of choice of  $\I$ instead of always using $\I = \N$.}}
  of indices.   
  
 Given a non-increasing function $\psi:\R_{> 0}\to \R_{> 0}$, a point $x\in X$
     is called  \underbar{$(\mathcal{X},\psi)$-uniform}, with the notation $x\in\UA_{\mathcal{X}}(\psi)$,  if %there exists a positive
                                %$\vre < 1$ such that  
for every sufficiently
large $t$  one can find $s\in I$ such that 
%a nonzero integer 
%\eq{diverygeneral}
$${
d_s (x)
\le \psi(t)
  \ \ \ \mathrm{and}  
\ \ h_s
\le  t
.}$$
\end{definition}
The set-up of Theorem \ref{Jarnik} %discussed in \S\ref{beginning} 
clearly corresponds to 
\eq{mainexample}{ X = \mr,\  I =  \Z^m\times (\Z^n\nz), \ %R_{\vp,\vq} = \{A\in\mr: A\vq = \vp\},\\ &
h_{\vp,\vq} = |\vq|\, \text{ and } \ d_{\vp,\vq}(A) = |A\vq - \vp|{.}%\quad\end{aligned}
}
It immediately follows from the above definition that 
points $x\in X$ such that $d_s(x) = 0$ for some $s\in I$ are $(\mathcal{X},\psi)$-uniform  for any positive function $\psi$. Such points will again be referred to as {trivially singular}. In order to prove existence of $(\mathcal{X},\psi)$-uniform points that are not trivially singular one needs to pay special attention to the geometry of the  zero sets $d_s^{-1}(\{0\})$; those are usually called  \underbar{resonant sets}. In \cite{KMWW} this has been done using  a certain convenient language, which we present below in a slightly simplified way. %\pagebreak
\begin{definition}\label{collections} Let  a   collection $\mathcal{L}$ %and $\mathcal{R}$ 
     of proper closed subsets of %a %locally compact 
%metric space 
$X$ be given. 
\begin{itemize}
 \item 
  Say that    $\mathcal{L}$   is 
 \underbar{dense}   
%  relative to $\mathcal{R}$} %(to be abbreviated by TD)]
 if 	the union  %\eq{vd1}
 ${\bigcup_{L\in \mathcal{L}} L\text{ is dense in }X}$.
 \item 
  Say that    $\mathcal{L}$   is 
 \underbar{totally
%very 
dense} (in $X$) 
%  relative to $\mathcal{R}$} %(to be abbreviated by TD)]
 if it is dense,
and in addition  for any open  $W
 \subset                           X$ and  any $L\in \mathcal{L}$ with $L  \cap W \neq
                          \varnothing$, the closure of the union $$\bigcup_{L'\in \mathcal{L}\,:\,L'\cap L \cap W  \ne
                          \varnothing}L'$$ has a nonempty interior.
                          %\qquad\\\ \text{ 
% in a proper affine subspace of $Y$ (resp.,  
%is dense in }Y)&.
\item
Say that  $\mathcal{L}$ \underbar{is aligned with}  a Diophantine system $\mathcal{X} =
     \left( X, \mathcal{D}, \mathcal{H}\right)$   if for any $L\in \mathcal{L}$ there exists $s\in I$ such that  $d_{s}|_{L}\equiv 0$; in other words, if $L$ is contained in one of the resonant sets $d_s^{-1}(\{0\})$. Note that this property is independent of the {choice of heights and generalized distance functions} and depends only on the resonant sets.
\item Say that  %$\mathcal{R}$ (or {
a collection $\mathcal{R}$ of subsets of $X$ % $\mathcal{R}$ 
\underbar{is respected by} $\mathcal{L}$
if whenever $L\in \mathcal{L}$ and $R\in \mathcal{R}$ are 
such that  $L\cap R$ has nonempty interior in $L$,
% (in the topology induced from $X$),
 it follows that $L\subset R$.  %\comm{Maybe a picture illustrating disrespectful behavior?} 
%\textsl{respects}  if it respects every $R\in \mathcal{R}$.  
\end{itemize}
\end{definition}
The following is a simplified version of \cite[Theorem  1.5]{KMWW}.
   \begin{thm} \label{kmww}
	Let $X$ be a locally compact %complete 
	metric space and $\mathcal{X} =
     \left( X, \mathcal{D}, \mathcal{H}\right)$ a Diophantine system.
        %, and let $S \subseteq Y$ be a nonempty locally
        %closed set. % with no isolated points. %\comm{Actually I (Dima)
                                %am not sure why we need $\R^d$, maybe
                                %an arbitrary metric space will do.}  
                                %on $\varphi_k$?}.  
	Let $\mathcal{L}$  be a countable 
        collections of closed connected  subsets of $X$ %with empty interior 
       that is
            %\begin{itemize}
        %\item 
        totally dense and aligned with $\mathcal{X}$, and        %\item 
        let $\mathcal{R}$ be a countable collection of closed nowhere dense subsets of $X$ respected by $\mathcal{L}$.
       	Then  the set 
	%\eq{specialset}
	${%Y_{\mathrm{sing}} \df
	%\left\{x\in Y :
        %  \forall\,k\in\N, \ \varphi_k(x)\text{ is $(\mathcal{X}_k,f_k)$-uniform }\right\}
         \displaystyle\UA_{\mathcal{X}}(\psi)
        \smallsetminus  \bigcup_{{R\in\mathcal{R}}} R
	}$
	%there exist 
	is %uncountable and 
	 uncountable and dense. 
		\end{thm}
In the classical Diophantine set-up \equ{mainexample} the relevant collection $\mathcal{L}$ consists precisely of resonant sets:
\eq{ell}{%\begin{aligned}
 \mathcal{L} = \{L_{\vp,\vq} : \vq \in\Z^n\nz, \ \vp\in\Z^m\},
 \text{  
% for $\vp\in\Z^m$ and $\vq \in\Z^n\nz$ we will let  $L_{\vp,\vq}$ be the affine subspace of $\mr$ given by
where }L_{\vp,\vq} := \{A\in\mr: A\vq = \vp\}.
%\end{aligned}
}
{These are rational affine subspaces of $\mr$ of codimension $m$; in fact one can  easily see that if $\p = (p_1, \ldots, p_m)$, then
\begin{equation}\label{direct product}
\bfL_{\p,\q} = \bfL_{p_1, \q} \times \ldots \times \bfL_{p_m, \q},
\end{equation}
where each $\bfL_{p_i, \q}$ is a rational affine hyperplane in the space of $i$th rows of matrices in $\mr$.} 
The total density {of  $\mathcal{L}$ as in \equ{ell} under the assumption   $n>1$} was proved in \cite[Proposition 3.1]{KMWW} and is straightforward to verify when $m=1$. Also it is not hard to see that any closed real analytic submanifold   of $\mr$  is respected by %$\mathcal{L}$ as in \equ{ell} 
{any collection of affine subspaces} \cite[Lemma 3.2]{KMWW}. Thus Theorem \ref{Jarnik}(1) becomes a special case of  Theorem \ref{kmww}. In fact  a slightly stronger version follows, with the set of totally irrational $\psi$-uniform matrices replaced by $$  \UA_{m,n}(\psi)
        \smallsetminus  \bigcup_{{R\in\mathcal{R}}} R,$$ where $\mathcal{R}$ is a countable collection of  proper {closed} analytic submanifolds     of $\mr$. We remark that an additional argument using transference {derives a similar strengthening of Theorem \ref{Jarnik}(2) from Theorem \ref{kmww}}, see \cite[Theorem 7.2]{KMWW}.} %\comm{I think here we should have some more examples of totally dense collections. Maybe mention rational lines in $\R^2$ parallel to coordinate axes?}

\subsection{Uniform approximation with restrictions}\label{restr} {The paper \cite{KMWW} outlines many diverse applications of the general set-up of Definition \ref{diophsys}, see e.g.\ \cite[\S\S 1.4--1.9]{KMWW}. One of them is the following:} instead of requiring the solutions $\q, \p$ of the system \eqref{mainsystem} to be {arbitrary} integer vectors, one can assume that they belong to given subsets of $\mathbb{R}^n\nz$ and $\mathbb{R}^m$ {respectively} and study solvability of \eqref{mainsystem} for all {sufficiently large $t$}. The definitions of $\psi$-uniform matrices %and $\psi$-Dirichlet spectrum 
can be naturally generalized for this set-up.

\begin{definition}\label{UAPQ_def}
    Let {$P \subseteq \mathbb{R}^m$ and $Q \subseteq \mathbb{R}^n\nz$}. For a non-increasing function $\psi: \,\,\,\R_{> 0}\to \R_{> 0}$
    % and a matrix $A\in M_{m,n}(\mathbb{R})$, 
    we say that $A\in M_{m,n}$ is $\psi$-uniform with respect to ${P,Q}$, with the notation $A\in \UA_{P, Q}(\psi)$, if for all %large enough 
    {sufficiently large $t$ there exist} {${\bf p}\in P$ and ${\bf q}\in Q$}   such that \eqref{mainsystem} holds.
\end{definition}

{In order to apply Theorem \ref{kmww} to this new set-up, one is led to considering the collection
%\eq{ellpq}{%\begin{aligned}
 $$\mathcal{L}_{P,Q} := \{L_{\vp,\vq} : \vp\in P,\ \vq \in Q \}%\end{aligned}
 $$
%}
instead of \equ{ell}. Clearly, the smaller are $P$ and $Q$, the harder it may be to prove the total density of $\mathcal{L}_{P,Q}$ and to exhibit elements of $\UA_{P, Q}(\psi)$. For motivation let us consider the following example: take \eq{niceexample}
{n=2, \ m=1,\  P=\Z\ \text{ and }\ Q\ = \big(\Z(1,0) \cup \Z(0,1)\big)\smallsetminus\{(0,0)\}.}
Then   the collection $\mathcal{L}_{P,Q}$ consists of rational lines parallel to the coordinate axes in $\R^2$ and can be easily checked to be totally dense. Thus it follows from Theorem \ref{kmww} that for any positive non-increasing $\psi$ the set
$$
\UA_{P, Q} = \left\{
 (x,y)\in\R^2\left|
 \begin{aligned}\text{ for all  %large enough
 {sufficiently large }}t > 0 \text{ one of the systems }\\
\  \ \begin{cases}
     |qx   - p |\le \psi(t)\\ 
      \ \  0 < q \le t
\end{cases}
\text{ and  \ \ \ }\begin{cases}
     |qy   - p |\le \psi(t)\\ 
     \ \  0 < q \le t
\end{cases}\\ \text{ has a non-trivial integer solution.}\qquad\ \ 
\end{aligned}\right.\right\}
%\begin{cases}
%     | A{\q}-{\p}|\le \psi(t)\\ 
%     | {\q}|\le t
%\end{cases}    
$$
is uncountable and dense.
Also in general one does not %have
{need} to assume that $P \subseteq \mathbb{Z}^m$ and $Q \subseteq \mathbb{Z}^n$. A special case of  our general set-up %is that of  
is  {inhomogeneous} Diophantine approximation, which we discuss in \S\ref{inhomappr} below.}

\smallskip

One of %our results
{the goals of this paper is to prove} the following theorem, stated without proof in \cite[\S8]{KMWW}.
 \begin{theorem} \label{8.1} The collection  $\mathcal{L}_{P,Q}$ is totally dense in $\mr$, and hence 
\eq{conclusion1}{
\begin{aligned} \UA_{P,Q}(\psi)
        \smallsetminus  \bigcup_{{R\in\mathcal{R}}} R\text{ is   uncountable and dense for any %positive
        non-increasing }
     %function 
\psi: \R_{>0} \to \R_{>0}\\ \text{and any  countable collection $\mathcal{R}$ of proper analytic submanifolds of }\mr,\qquad
\end{aligned}}
     if
 \begin{itemize}
\item[\rm(a)]  $Q = \Z^n\nz$ and  $P$ is a subgroup of $\Z^m$ of rank $>m-n+1$;
\item[\rm(b)]  $Q = Q_1\times\cdots\times Q_n\subseteq \Z\times\cdots\times\Z$, 
where
 at least two of the sets $Q_i$ are infinite, 
and $P$ is of bounded Hausdorff distance from $\R^m$.
\end{itemize}
%It is clear that this collection is very dense whenever $n > 1$ \comm{(does it require an explanation?)}
%Consequently, for any %positive
 %       non-increasing 
     %function 
%$\psi: \R_{>0} \to \R_{>0}$ the set of %totally irrational
%     $m \times n$ matrices	 that are  $\psi$-uniform with respect to $P,Q$ and not contained in a given countable family of proper analytic submanifolds of $\mr$ is
%     uncountable and dense. 
 \end{theorem}
% Note that both (a) and (b) implicitly assume that $n > 1$, and the set-up of (b) includes {inhomogeneous Diophantine approximation as a special case, see \S\ref{inhomappr}}.
 
 %\smallskip
 
 {%One of the goals of the present paper is to %provide a proof of
{Below we state and prove} several extensions of Theorem~\ref{8.1}. One new observation that we make in this paper is as follows: in order to establish \equ{conclusion1} one does not have to work with the whole collection $\mathcal{L}_{P,Q}$; instead it suffices to find a totally dense subcollection. Indeed, if $\mathcal{L}\subseteq\mathcal{L}_{P,Q}$ is totally dense, one can apply Theorem \ref{kmww} to $\mathcal{L}$ and \equ{conclusion1} will follow. One of the reasons why this observation is valuable is 
the existence of collections of subspaces that are not   totally dense but contain  totally dense subcollections. See {Proposition \ref{subcoll}} for an example.} 
%!
{It will be convenient to %abbreviate 
{refer to} this property %and
%Another way of saying this is the following: let us 
%say 
{by saying} that a pair $(P, Q)$ \underline{has property (TDS)} if $\mathcal{L}_{P,Q}$ contains a  totally dense subcollection. Obviously this property is stable with respect to taking {supersets}, which, in view of the aforementioned example, cannot be said about the (less {well-behaved}) property of  $\mathcal{L}_{P,Q}$ being totally dense.}
%\comm{Here we should refer the reader to an example in some later section.}

{In order to state %two \comm{( I wrote "two" because was thinking of adding another theorem here, let's discuss it)} 
 a notable special case of our more general results,
  %In order to understand how small can $P$ and $Q$ be for that to happen,} w
  we need some additional definitions and notation.} For {$P\subseteq \R^m$, an $m$-dimensional unit vector $\phi \in \mathbb{S}^{m-1}$} and $\varepsilon > 0$, define  the set
%\eq{pepsilon}{
$$
%Q_{\varepsilon}^{\theta}: = \left\{ \q \in Q: \,\,\, \left\| \frac{\q}{\| \q \|} - \theta \right\| < \varepsilon \right\} \,\,\,\,\,\,\,\,\,\,\,\, \text{and} \,\,\,\,\,\,\,\,\,\,\,\, 
P_{\varepsilon}^{\phi}: = \left\{ \p \in P: \,\,\, \left\| \frac{\p}{\| \p \|} - \phi\right\| < \varepsilon \right\} 
$$
of elements of %$Q$ and 
$P$ belonging to the $\varepsilon$-cone in direction $\phi$. 
We use the notation
%\eq{pepsilonnorms}{
$$
%|Q|_{\varepsilon}^{\theta}: = \{ \| \q \|: \,\,\, \q \in  Q_{\varepsilon}^{\theta} \} \,\,\,\,\,\,\,\, \text{and} \,\,\,\,\,\,\,\, 
|P|_{\varepsilon}^{\phi}: = \big\{ \| \p \|: \,\,\, \p \in  P_{\varepsilon}^{\phi} \big\}
$$
for the sets of corresponding norms. 
Say {that $P$ is \underbar{logarithmically dense} along $\phi$ if  for %any $\phi \in \mathbb{S}^{m-1}\,\cap\, {H}$ and 
any $\varepsilon > 0$ the set $|P|_{\varepsilon}^{\phi}$ 
contains an
unbounded sequence $\{ s_k \}_{k=1}^{\infty}$ %of positive real numbers 
with $\lim\limits_{k \rightarrow \infty} \frac{s_{k+1}}{s_k}=1$; equivalently (a characterization that motivates the terminology), if the gaps between logarithms of elements of $|P|_{\varepsilon}^{\phi}$ tend to $0$ at infinity. See Lemma \ref{LDequiv} for several equivalent restatements. We will say that $P$ is  \underbar{logarithmically dense in a cone} ${H}\subseteq \R^m$ if it is  logarithmically dense along  any $\phi \in \mathbb{S}^{m-1}\,\cap\, {H}$.}

%{ We can naturally extend the definition of logarithmic density if we take half-spaces or, more generally, any cones in place of ${H}$.}
%For a set $P \subseteq \mathbb{R}^m$, we say that $Q$ is {\it exponentially dense 
%in direction $\theta \in S^{m-1}$} 
%Then say that $P\subseteq \R^m$ is  \underbar{logarithmically dense in direction} $\phi \in \mathbb{S}^{m-1}$ if for any $\varepsilon > 0$ the set $|P|_{\varepsilon}^{\phi}$ is logarithmically  dense. We say that $P$ is \underbar{logarithmically dense} if is  is logarithmically dense in any direction.} 
{Clearly being  of  bounded Hausdorff distance from $\R^m$ implies logarithmic  density in $\R^m$. But the latter condition is much less restrictive. For example {the sets}  $\mathbb P$ (prime numbers) or $\{\pm k^l : k\in\N\}$ for any fixed $l>0$ are logarithmically dense in $\R$. Also a Cartesian product of two logarithmically dense sets is logarithmically dense, see Proposition \ref{prod_of_log_dense}.}
\smallskip

Let us also introduce the spherical projection  \eq{defpi}{\pi:\R^n\nz \to \mathbb{S}^{n-1}, \quad\q\mapsto \q/\|\q\|.} By  $B_C(\bx)$ we will denote the Euclidean ball of radius $C$ centered in $\bx\in\R^n$; this convention will be extended to other metric spaces. 
Throughout the paper we will often %refer to 
{work with} spheres and halfspheres. {By a \underbar{$k$-subsphere} of $\mathbb{S}^{d-1}$ we will %always refer to 
mean the intersection of a $(k+1)$-dimensional} linear subspace of $\R^d$ with $\mathbb{S}^{d-1}$. Likewise, {by a \underbar{$k$-halfsphere} of $\mathbb{S}^{d-1}$ we will %always refer to 
mean 
%always refer to 
the intersection of a $(k+1)$-dimensional} linear half-space of $\R^d$ with $\mathbb{S}^{d-1}$. 

\smallskip

The following theorem
%, a notable special case of our more general results, 
{incorporates the set-ups of both parts of Theorem~\ref{8.1}:}

%\textcolor{orange}{\begin{definition}
%    Let $0\le k\le n$ be integers. Then a $k$-dimensional halfsphere $T$ in $\mathbb{S}^{n}$ is a subspace $T\subseteq \mathbb{S}^n$ satisfying that $T=\mathbb{S}^n\cap L$ for some $k+1$ dimensional half-space $L$ in $\mathbb{R}^{n+1}$ whose boundary contains the origin.
%\end{definition}}

\begin{theorem}\label{not_on_line}
    Let $m,n\in\mathbb{N}$ 
    %where $n>1$. 
    and  $0 \leq k \leq \min(n-2, m)$. Suppose $\varnothing \ne P \subseteq \mathbb{R}^m$ is {logarithmically dense in an $(m-k)$-dimensional subspace ${H}$ of $\R^m$, and $Q\subseteq \R^n\nz$ %is an unbounded set with an additional condition: $Q$ 
    %contains two non-proportional elements with arbitrarily large norms. 
    satisfies the following property:
   \eq{twohalfspheres}{
    \begin{aligned}\forall\, C>0  \text{ there exist two $k$-halfspheres }&\Phi_1,\Phi_2\subseteq \mathbb{S}^{n-1}\text{ not contained in a single $k$-sphere}\\
    \text{ such that }&\Phi_1\cup\Phi_2 \subseteq \overline{\pi\big(Q\smallsetminus B_C(0)\big)}.
    \end{aligned}
   }
    Then 
    %!
    %the collection  $\mathcal{L}_{P,Q}$ contains a  totally dense subcollection
    {$(P,Q)$  has property {\rm (TDS)}}, hence \equ{conclusion1} holds.}
    %. Consequently,   the set  $$\UA_{P,Q}(\psi)
   %     \smallsetminus  \bigcup_{{R\in\mathcal{R}}} R$$ is   uncountable and dense for any %positive
 %       non-increasing 
     %function 
%$\psi: \R_{>0} \to \R_{>0}$ and any  countable collection $\mathcal{R}$ of proper analytic submanifolds of $\mr$.}
 %    Statement \ref{central} (and, more generally, Statement \ref{central2} for any Diophantine system $\mathcal{X}$) holds.   
\end{theorem}

%But %the latter 
%are much less restrictive. 

%\comm{This formulation makes no sense in case $k = n-1$. In more general results we can just take two supspaces in $P$-coordinate rather than $Q$, but here we can't. We can just restrict $k \leq n-2$, but maybe it can be addressed in a better way?}

{Note} that one cannot relax the condition of {the} logarithmic density of $P$ without putting   additional constraints on $Q$ and still guarantee 
%!
{(TDS)}; 
%the existence of a totally dense subcollection {of $\mathcal{L}_{P,Q}$}; 
see \S \ref{logdense_optimal_section} for details and discussion.

As an illustration, let us see what this theorem  is saying in the case $k=0$. Then ${H} = \R^m$,   $0$-dimensional subspheres are simply {pairs of} directions in $\mathbb{S}^{n-1}$, and 
%it is easy to see that    
\equ{twohalfspheres} is equivalent to 
\eq{nonprop}{Q \text{
    contains two non-proportional elements with arbitrarily large norms.}}
    Obviously \equ{nonprop} is satisfied 
    when $Q$ is as in Theorem \ref{8.1}(b);
    % and $P$ of  bounded Hausdorff distance from $\R^m$ is clearly logarithmically dense in $\R^m$. 
    hence the "$k=0$" case of Theorem \ref{not_on_line} implies Theorem \ref{8.1}(b) in a slightly weaker form, with  total density replaced by %existence of a totally dense subcollection
    %!
    {(TDS)}.   %But the assumptions of 
    However in Theorem \ref{not_on_line} $Q$ does not have to have the product structure, e.g.\ it can be the set of integer points on an unbounded algebraic curve such as $\{(x,x^2, \ldots,x^n) \}$; and as was mentioned before, logarithmic  density in $\R^m$ is much weaker than bounded Hausdorff distance from $\R^m$. % implies . But the latter condition is much less restrictive
%    are much less restrictive.
%For example they will be satisfied, {still with $k=0$,} if %one takes $P$ consisting of vectors with prime coordinates, 
%or vectors satisfying some congruence conditions, 
%and lets $Q$ be the set of integer points on an unbounded algebraic curve such as $\{(x,x^2, \ldots,x^n) \}$.  
The assumptions will be weakened even further in our main results, Theorems  %Theorem 
\ref{any_k_theorem} and \ref{ch_theorem}. {See Examples \ref{ch_examples}, \ref{ex_logdenseP} and \ref{thelastexample} for several concrete special cases in which the aforementioned two theorems can be applied.}
%\comm{Need to refer to a discussion somewhere later in the paper.} 
%\comm{I believe we should find a good way of organizing examples; now some of them are just in the text, some are treated as examples to particular theorems, and some are formulated as a separate section. There are also intersections.}%{Another remark to make is that the conditions on $P,Q$ in  Theorem \ref{not_on_line}, unlike those of Theorem \ref{8.1}, are invariant by scaling. This makes sense since one can easily observe that $$L_{c\vp,d\vq} = \frac cd L_{\vp,\vq},$$ hence  $\mathcal{L}_{P,Q}$ is totally dense, or contains a  totally dense subcollection, if and only if the same is true for $\mathcal{L}_{cP,dQ}$ for any $c,d > 0$.} 

\begin{remark}  \rm   Note that if one adds some additional restrictions guaranteeing that $Q$ and $P$ are not "too dense", e.g.\  assumes that $P \subseteq \mathbb{Z}^m$ and $Q \subseteq \mathbb{Z}^n \nz$, then one can show that  condition \equ{nonprop}
%{given in} Theorem \ref{not_on_line}
%cannot be weakened.
is optimal.
%\comm{Please write down your propositions somewhere later and refer to them here.}
 %   Moreover, one can show that in this case the condition on $Q$ cannot be improved if we want an analog of Theorem \ref{Jarnik}(1) to hold. 
    %\comm{Not clear: are you talking about \equ{twohalfspheres} or \equ{nonprop}?} 
    Indeed: suppose all except finitely many elements $\q \in Q$ are proportional, that is, are of form $k \q_0$ for some fixed $\q_0 \in \mathbb{Z}^n\nz$ and $k \in \mathbb{Z}$. Let ${\ba} = A \q_0 \in \mathbb{R}^m$. In order for \eqref{mainsystem} to {be solvable in $\vq\in Q$ and $\vp\in\Z^m$}, we need the system 
    $$
    {{|{k\ba}   - \vp |
\le  \psi(t)}
  \ \ \ \mathrm{and}  
\ \ {k
\le  t/|\vq_0|}}
    $$
%    equivalently, 
%    $$
%    {{|A' k   - \vp |
%\le  \psi(|\q_0| t')}
%  \ \ \ \mathrm{and}  
%\ \ k
%\le  t'},
%    $$
to have a solution $k \in \mathbb{Z}, \, \vp \in %P \subseteq 
\mathbb{Z}^m$ for any 
%large enough 
{sufficiently large} $t$. %(respectively, $t'$) . 
If $\lim\limits_{t \rightarrow \infty} t \psi(|\q_0| t) = 0$, {this} implies that every coordinate of the vector ${\ba}$ is singular, hence rational. %In this case  
{Consequently $A (k\q_0) \in \mathbb{Z}^m $ for some $k \in \mathbb{Z}$, thus} $A$ has to be trivially singular (in particular, not totally irrational). This argument shows that 
%in the absence of 
\equ{nonprop} is necessary for \equ{conclusion1}, and hence for 
%!
{$(P,Q)$  having property (TDS)}.
%the existence of a totally dense subcollection inside $\mathcal{L}_{P,Q}$.

We also remark that  %restrictions on $Q$ similar to 
%the 
%restriciton 
\equ{nonprop} %on $Q$
has already been used in a similar context: namely, it is equivalent to the negation of condition (B) from  \cite{KM03}, and it was shown there 
(see the last paragraph of the paper) %\comm{(to be honest, I looked at \cite{KM03} and did not underdstand much, hope we'll talk about it and maye edit this remark.)}
that in case $m = 1, \, P = \mathbb{Z}$ and $Q \subseteq \mathbb{Z}^n$ it is enough for $Q$ to %have two arbitrarily large non-proportional elements 
satisfy \equ{nonprop} in order for the set $\UA_{P, Q}(\psi)$ to contain linear forms which are not trivially singular. %\comm{Need to check what exactly they say about these sets - it's more than nonemptiness. YES –– PLEASE CHECK.}
\end{remark}

%\smallskip

%It is also easy to see 
{Note} that the assumptions of Theorem \ref{not_on_line} hold in the case of Theorem \ref{8.1}(a) as well. {Indeed, take $Q = \Z^n\nz$ and a  subgroup $P$ of $\Z^m$ of rank $l\ge \max(m-n+2,0)$, and let ${H}$ be the $\mathbb{R}$-span of $P$. Put   $k := m-l $.
%= \dim \Pi$.
%$k := m-l$. The $\mathbb{R}$-span of $P$ is an $(m-k)$-dimensional linear subspace of $\mathbb{R}^m$. 
%, thus it contains a  subspace $\Pi\subseteq\mathbb{R}^m$ of dimension $\max(m-n+2, 0) = m-k$. 
%let $k = \min(n-2, m)$. In this theorem $n>1$, hence $k \ge0$. T
If $l = m-k = 0$, then  the logarithmic density of $P$ in ${H}$ is a trivial condition which holds for any nonempty set $P$. Otherwise $P$ is a lattice in ${H}$, which clearly implies the logarithmic density of $P$ in ${H}$.
%, and we only need to verify nonemptiness. 
%If $l > 0$, fix $\phi \in \mathbb{S}^{m-1} \cap \Pi$ and $\varepsilon > 0$. Since $P$ is a subgroup of $\mathbb{Z}^m$, there exists a constant $C>0$ such that every point of $\mathbb{R}$-span of $P$ has at most distance $C$ to $P$. In particular, for every $x \in \mathbb{R}_+$ there exists $\p \in P$ such that $|x \phi - \p| \leq C$. It remains to notice that for all $x$ large enough such points $\p$ have to belong to $P_{\varepsilon}^{\phi}$. Thus $P$ is logarithmically dense in $\Pi$.}
Further, for any $C>0$ we can see that $\overline{\pi\big(\Z^n\smallsetminus B_C(0)\big)} = \mathbb{S}^{n-1}$   %for any $C>0$, this set 
  contains any number of $k$-dimensional halfspheres or even spheres since $k \le m - (m-n+2) = n-2$.
Hence \equ{twohalfspheres} holds, and a weaker 
%!
{(TDS)} form of Theorem \ref{8.1}(b) also follows from Theorem \ref{not_on_line}.}
{Yet, the more restrictive conditions of Theorem \ref{8.1} make it possible to strengthen the conclusion to the total density of the collection itself, as was announced in \cite{KMWW}. See \S\ref{pf8.1} for the proof.}

%\comm{Here  we also should comment that the more restrictive conditions of the theorem allow to strengthen the conclusion to the total density of the collection itself, again referring the reader to \S\ref{proof8.1}.} %On the other hand, we can relax the set-up of Theorem \ref{8.1}(a) even further, as the next example shows.
%Hence Theorem \ref{not_on_line} implies Theorem \ref{8.1} in a slightly weaker form, with the total density replaced by existence of a totally dense subcollection.  

\subsection{Inhomogeneous Diophantine approximation}\label{inhomappr}
%Let $\vb \in \mathbb{R}^m$ be a vector; 
In the inhomogeneous set-up one replaces linear forms $\q\mapsto A\q $ with affine forms $\q\mapsto A\q  - \vb $, where $\vb \in \mathbb{R}^m$, and studies the same questions as above.
%, we obtain the set-up of  Diophantine approximation. 
More precisely,  the system \eqref{mainsystem} is replaced with \begin{equation}\label{digeneralinh}%\tag{2.5}
%{|A\vq + \vb - \vp |
%\le  \psi(t)}
%  \ \ \ \mathrm{and}  
%\ \ |\vq|
%\le  t
%.
\begin{cases}
     | A{\q} - \vb-{\p}|\le \psi(t)\\ 
     \qquad\qquad  \ \ | {\q}|\le t
\end{cases}\end{equation}  
and one can choose $P \subseteq \mathbb{R}^m$, $Q\subseteq \R^n\nz$ and  define the set
 $$
    \widehat{\UA}_{P, Q}(\psi): = \big\{ (A, \vb) \in M_{m,n} \times \mathbb{R}^m: \,\, %\text{\eqref{inhsystem}
    \eqref{digeneralinh}\text{  has a solution  $(\p, \q) \in P \times Q$ for all  
    %large enough 
    {sufficiently large} }   t   \big\}.
    $$
    Denote $\widehat{\UA}_{\mathbb{Z}^m, \mathbb{Z}^n\nz}(\psi)$ %\comm{(we need to decide whether or not to include $\q=0$)} 
    simply by $\widehat{\UA}_{m,n}(\psi)$, and  say that $(A, \vb) $  is   \underbar{trivially singular} if   $ A{\bf q}-\vb\in \mathbb{Z}^m$ for some  $ \q \in  \mathbb{Z}^n\nz$.
  It was proved by Khintchine in 1947  \cite{Khi} that  for any $m,n\in\N$ and any non-increasing function $\psi: \R_{> 0}\to\R_{> 0}$ the set  of pairs $(A, \vb) \in\widehat{\UA}_{m,n}(\psi)$ that are not {trivially singular is {uncountable and dense}  (see also \cite[\S 3.3]{MN25} for a discussion).}

%the system 
%\begin{equation}\label{inhsystem}\tag{2.4}
%\begin{cases}
%     | A{\q} - \vb -{\p}|\le \psi(t),\\ 
%     | {\q}|\le t.
%\end{cases}
%\end{equation}
%We will discuss two possible set-ups: 

\smallskip
Using Theorem \ref{not_on_line} we can generalize this as follows: 

\begin{corollary}\label{inhom_pairs}
    For any $m,n\in\mathbb{N}$, suppose that $Q\subseteq \R^n\nz$ is unbounded and $P \subseteq \mathbb{R}^m$ is logarithmically dense in  $\R^m$.  
    Then  the set $$\widehat{\UA}_{P, Q}(\psi)
        \smallsetminus  \bigcup_{{R\in\mathcal{R}}} R$$ is   uncountable and dense for any %positive
        non-increasing  
     %function 
$\psi: \R_{>0} \to \R_{>0}$ and any  countable collection $\mathcal{R}$ of proper analytic submanifolds of $\mr\times \mathbb{R}^m$.
\end{corollary}

\begin{proof} Clearly   a pair $(A, \vb)$ is in  $ \widehat{\UA}_{P, Q}(\psi)$ if and only if   {the} augmented  matrix 
$(A\,|\,\vb) \in M_{m,n+1}$ belongs to $ {\UA}_{P, Q'}(\psi)$, where 
$Q': = Q \times \{ -1 \} \subseteq \mathbb{R}^{n+1}$. The unboundedness of $Q$ implies that $Q'$ satisfies   \equ{nonprop}, thus Theorem \ref{not_on_line}  applies and the conclusion follows.
    \end{proof}

One can also fix a vector $\vb \in \mathbb{R}^m\smallsetminus \Z^m$ and study the set of matrices $A$ such that the pair $(A, \vb)$ is $\psi$-uniform; that is,  %consider 
{the set} 
$$\begin{aligned}
    \widehat{\UA}_{P, Q}^\vb(\psi) &: = \big\{ A \in M_{m,n}: (A, \vb) \in \widehat{\UA}_{P, Q}(\psi)  \big\}\\ 
    &\ = \big\{ A \in M_{m,n}:  \eqref{digeneralinh}\text{ has a solution   $(\p, \q) \in P \times Q$ for all  
    %large enough
    {sufficiently large} }   t  \big\}. \end{aligned}   $$
    Here, in order to apply our technique, one needs to assume\footnote{However see \cite{Schl} for some results in the case $n=1$ for inhomogeneously singular vectors, that is, elements of $\bigcap_{c>0}\widehat{\UA}_{m,1}^\vb(c\psi_{1/m})$ for fixed $\vb \in \mathbb{R}^m$.} $n>1$. 
    To the best of the authors' knowledge there have been no results in the literature concerning {the} existence of non-trivial elements of $\widehat{\UA}_{m,n}^\vb(\psi): =  \widehat{\UA}_{\mathbb{Z}^m, \mathbb{Z}^n\nz}^\vb(\psi)$. %\comm{Is this true? -- At least I have never seen it.} 
    However, given the obvious relationship
    $$
    \widehat{\UA}_{P, Q}^\vb(\psi)  = {\UA}_{P+\vb, Q}(\psi) 
    $$
and the fact that logarithmic density is invariant under translation (see Lemma \ref{LDequiv}), Theorem \ref{not_on_line} immediately implies

\begin{corollary}\label{inhom_single}
    For any $m,n\in\mathbb{N}$ with $n>1$, suppose that $Q\subseteq \R^n\nz$ satisfies  \equ{nonprop}, and $P \subseteq \mathbb{R}^m$ is logarithmically dense in  $\R^m$.  
    Then  the set $$\widehat{\UA}^\vb_{P, Q}(\psi)
        \smallsetminus  \bigcup_{{R\in\mathcal{R}}} R$$ is   uncountable and dense for any $\vb \in \mathbb{R}^m$, any %positive
        non-increasing  
     %function 
$\psi: \R_{>0} \to \R_{>0}$ and any  countable collection $\mathcal{R}$ of proper analytic submanifolds of $\mr$.
\end{corollary}
We refer the reader to a recent preprint \cite{Agg} by Aggarwal for some upper estimates of the Hausdorff dimension of sets %$\widehat{\UA}_{m,n}(\psi)$ and 
$\widehat{\UA}_{m,n}^\vb(\psi)$. %\comm{Vasya: Should we also mention Kim+Kim then? I think it's equally relevant. DK: Kim+Kim is about the Hausdorff dimension of the complement to $\UA$ for slowly decreasing $\psi$. }

%\comm{Here we can send the reader to some later remark. Also here it would be nice to  have something more general resembling Theorem \ref{8.1}(a).
%Also we can mention here that Theorem \ref{not_on_line} implies the $\psi$-uniform case of Corollary \ref{corollary_pairs}, again referring to a later section.}
%\bigskip

%\centerline{MY CHANGES STOP HERE}

%\vfil\eject

\subsection{{The structure of the paper}}\label{struct}
%{
The main results of this paper are presented in \S\ref{2theorems}: namely, these are Theorem \ref{any_k_theorem}, which is a generalization of Theorem \ref{not_on_line}, and Theorem \ref{ch_theorem}. In \S\ref{pf1.6} we show how to deduce Theorem \ref{not_on_line} from Theorem \ref{any_k_theorem}. In \S\ref{appl} we {exhibit} a variety of examples and special cases, illustrating our main {theorems}.

{The proof of  Theorems \ref{any_k_theorem}  and   \ref{ch_theorem} occupies \S\S\ref{just_density_section}--\ref{proofs}. In \S\ref{just_density_section} we describe 
%two general conditions for the density %guaranteeing 
 some ways to prove  that a %certain 
%of 
collection of subspaces of} $M_{m,n}$ is dense, %these %conditions will 
%be used in the proof of our main results.
while in
\S\ref{general_conditions_section} we describe a general topological argument, Lemma \ref{usefullemma}, which relates density {to} total density
%. This argument is 
{and constitutes} the main ingredient of our proofs. In \S\ref{proofs} we apply the method developed in \S\ref{general_conditions_section} to prove our main results: Theorem \ref{any_k_theorem} is proven in \S\ref{proof_anyk_section} and Theorem \ref{ch_theorem} is proven in \S\ref{Ch_proof_section}.

In \S\ref{optimality} we introduce a formal framework that allows us to reformulate our   {results} in a uniform way and show that they are optimal. 
In \S\ref{logdense_sets_section} we discuss logarithmically dense sets in {detail},
%Namely, \S\ref{logdense_expl} shows alternative definitions of logarithmic density which are convenient to use in different contexts and gathers some examples and general properties of logarithmically dense sets, which can be used to produce a wide range of examples to our main statements. In Section \ref{logdense_optimal_section} we 
{in particular showing} that the logarithmic density requirement in  Theorem \ref{not_on_line} cannot be relaxed.

In Section \ref{pf8.1} we prove Theorem \ref{8.1}: namely, %we prove 
we show that for $P,Q$ as in this theorem 
the collections ${\mathcal L}_{P,Q}$ do not just have a totally dense subcollection, as follows from Theorem \ref{not_on_line}, but are totally dense. %However, in general it is not true that (TDS) implies total density of the collection itself: we shed the light on the difference
{We elaborate on the distinction between (TDS) and total density in \S \ref{subcoll_td_section}, {describing} an example of not totally dense collection with a totally dense subcollection}.
%of a collection which has property (TDS) but is not totally dense.
Finally, \S\ref{technicalsection} is technical: it  is devoted to establishing  Lemma \ref{opensetlemma},   an  auxiliary result {that} plays a key role in our argument.
%}

\subsection{Acknowledgments}
{The authors} are grateful {to} PRIMES for making this research opportunity possible, {and to Maxim Arnold and Tanya Khovanova for useful comments on a preliminary version of the paper}. The second-named author has been supported by  the NSF
 grant No.\ DMS-2155111. {This work was completed in Spring 2026, while the second-named author was in residence at the Simons Laufer Mathematical Sciences Institute in Berkeley, CA, supported by the National Science Foundation under Grant No.\ DMS-1928930. The hospitality of  SLMath is gratefully acknowledged.}

%\vskip+0.6cm

%is a special case of more general results. We state those in the next section, then prove them later. Also we have many interesting examples.

%\vskip 1in

 %We slightly change the notation and denote the set of $\psi$-uniform $m \times n$ matrices by $\UA_{P, Q}(\psi)$. The set denoted by $\UA_{m, n}(\psi)$ before should now be denoted by $\UA_{\mathbb{Z}^m, \mathbb{Z}^n \smallsetminus \{ 0 \}}(\psi)$; when it is clear from context, we will still use the notation $\UA_{m, n}(\psi)$ for convenience.

%Just as in the classical case, we would like to sort out the matrices $A$ which are $\psi$-uniform for trivial reasons. We call a matrix $A \in M_{m,n}$ {\it trivially singular with respect to $Q, P$} if there exist such $\q \in Q, \, \p \in P$ that $A  \q = \p$. 

%We denote the set of $\psi$-uniform with respect to $Q,P$, not trivially singular with respect to $Q, P$ matrices by $\UA_{P, Q}^*(\psi)$.

%\bigskip

\section{Main results}\label{main}

\subsection{Two general theorems}\label{2theorems}

%We will denote the unit $(d-1)$-sphere centered at the origin in $\mathbb{R}^d$ by ${\mathbb S}^{d-1}$. For a vector ${ \bf v} \in \mathbb{R}^n\smallsetminus\{0\}$, we define its angle as ${ \bf v}^* = \frac{ \bf v}{\lVert { \bf v } \rVert} \in {\mathbb S}^{n-1}$. 

%Given a hyperplane $\bf H$ of codimension 1, let $\bf v$ be a vector normal to $\bf H$. If $\bf H$ is non-linear, choose $\bf v$ so that $k\bf v\in\bf H$ for some positive $k$. Otherwise, for certainty, choose $\bf v$ so that the non-zero coordinate with smallest index is positive.

%We then define the angle of $\bf H$ as $\bf H^*=\bf v^*$ where $\bf v$ is defined as above. We define its magnitude $\lVert \bf H\rVert$ as its distance to the origin (in particular, $\lVert \bf H\rVert=0$ if and only if $\bf H$ is linear). 

%For any two points ${\bf x}$ and ${\bf y}$ , we will let $d({\bf x},{\bf y})$ denote the distance between these points. For an angle $\q \in {\mathbb S}^{n-1}$ and a set $\mathcal{L}$ of vectors or hyperplanes of codimension 1, we denote the set of elements of $\mathcal{L}$ whose angle is within $\epsilon$ distance of $\q$ by $\mathcal{L}^\q_\epsilon$ and the set of magnitudes of elements of $\mathcal{L}$ by $\widehat{\mathcal{L}}$.

%If $S\subseteq \prod_{i\in I} A_i$ for some index set $I$, we let $\text{pr}_{A_i}(S)=\text{pr}_i(S)$ be the projection mapping onto $A_i$. In particular, for a point $x = (x_i)_{i \in I}\in\prod_{i\in I} A_i$, we let $\text{pr}_i(x)=x_i$.

%Given Theorem \ref{kmww}, our goal 
{The goal of the paper} is to describe some sufficient conditions on the {sets {$P, Q$}} under which $\mathcal{L}_{P, Q}$ is totally dense
%!
{or contains a totally dense subcollection}. More generally, we would like to 
{fix a subset $\fR$ of $ \mathbb{R}^m\times\left( \mathbb{R}^{n} \nz \right) $ and}
work with {the collections}
$$
\mathcal{L}_\fR: = \{ L_{\p, \q}, \,\,\, (\p, \q) \in \fR  \}.
$$
%!
{Similarly to the special case $E = P\times Q$, we will say that $\fR$ {has property (TDS)} if $\mathcal{L}_{\fR}$ contains a  totally dense subcollection.}
%{In other words, 
%in the original Diophantine problem 
%one can study the solvability of the system \eqref{mainsystem} with the constraints on $\vp$ depending on the values of $\vq$.} \comm{Do we have good examples of that? - Not really. We can construct some exotic examples which will satisfy our main theorems, but they won't essentially use this dependence, since the conditions of main theorems are based on product structure. We can use statements like Lemma \ref{usefullemma} or Lemma \ref{usefulforproducts} directly and prove total density for such situations, but it is something significantly different. I suggest not stressing this part here}
%for some $R \subseteq \left( \mathbb{R}^{n} \setminus \{ 0 \} \right) \times \mathbb{R}^m$. 
One of the issues we face with this {set-up} is that, as one can easily observe, {\begin{equation}\label{scaling}
%\tag{2.0}
L_{h\vp,g\vq} = h L_{\vp,\vq}g^{-1}{\ \forall\,h \in \GL_m(\R),\ g \in \GL_n(\R);\text{ in particular }L_{c\vp,c\vq} =   L_{\vp,\vq}\ \forall\,c \in\R.}\end{equation} }Therefore there exists a variety of possible sets $\fR$ realizing the same collection $\mathcal{L}_\fR$. 
%hence  $\mathcal{L}_{P,Q}$ is totally dense, or contains a  totally dense subcollection, if and only if the same is true for $\mathcal{L}_{cP,dQ}$ for any $c,d > 0$. 
To {state} our main results it will be convenient to have some sort of   standard form for the set $\fR$. %Note that multiplying any elements of $\fR$ by a nonzero real valued scalars does not change the collection $\mathcal{L}_\fR$.
%For instance, 
{Namely we are going to}
%we can 
assume that the last $n$ coordinates of any element of $\fR$ form a unit vector.
More formally, %define
%!
{consider} the space $${\bf T}_{m,n}: = \mathbb{R}^m\times {\mathbb S}^{n-1} $$ and the projection map {extending $\pi$ defined in \equ{defpi} on $\mathbb{R}^n \smallsetminus \{ 0 \}$ to its product with $\R^m$:}
\eq{generalpi}{
\pr: \,\,\, {\mathbb{R}^m  \times \left(\mathbb{R}^n \smallsetminus \{ 0 \} \right)} \rightarrow \mathfrak{\bf T}_{m,n}, \,\,\,\,\,\,\,\,\,\,\,\, (\p, \q) \mapsto \left(\frac{\p}{\|\q\|}, \frac{\q}{\|\q\|} \right).
}
We will view $\tmn$ as a metric space with distance induced by the %supremum norm on $\R^m\oplus\R^n$
product of the distance functions on $\mathbb{R}^m$ and  ${\mathbb S}^{n-1} $, where $\R^m$ is endowed with the Euclidean metric, and the metric on $\mathbb{S}^{n-1}$ is induced by the Euclidean metric on $\R^n$. 
%\textcolor{orange}{metric on $\mathbb{R}^{m}\oplus\mathbb{R}^{n}$}. %\comm{Please check.}

It is clear that $\mathcal{L}_\fR = \mathcal{L}_{\pr(\fR)}$; thus, in our main results, we can {assume without loss of generality} that $\fR = \pr(\fR)$, or, equivalently, that {$\fR\subseteq {\bf T}_{m,n}$}.
Such a reduction clearly does not guarantee the unique parametrization of the set $\mathcal{L}_\fR$: to achieve uniqueness, one needs to work in ${\bf T}_{m,n}/\sim$, where {$(\p,\theta) \sim (-\p,-\theta)$}. However it will be more convenient for us to formulate and prove statements for ${\bf T}_{m,n}$. {Note that without any loss of generality one can impose an additional assumption that $\fR = -\fR$; that is,  starting with an arbitrary $\fR\subset {\bf T}_{m,n}$, one can symmetrize it by replacing it with $\fR\cup (-\fR)$.} %\comm{Do we need such a remark?}

%In most of our proofs, we will assume that $S = \pr(S)$.

%\begin{remark}
%    Clearly, considering $\pr(R)$ still 
%\end{remark}

%To make our notation clear, we will use fraktur font for subsets of ${\bf T}_{m,n}$; namely, we will be using the notation $\fR \subseteq { \bf T}_{m,n}$.%$, \, \mathfrak{Q} \subseteq {\mathbb S}^{n-1}$ and $\mathfrak{P} \subseteq \mathbb{R}^m$ for $\mathfrak{R} = \mathfrak{Q} \times \mathfrak{P}$.

%\vskip+0.3cm

\medskip
We are now ready to formulate our main results, Theorems 
\ref{any_k_theorem} and \ref{ch_theorem}. %\comm{Maybe we should switch the order of the theorems.}
%{The first one addresses the situation when the $\R^m$-part of $\pi(E)$, i.e. the set of ratios $\left\{\frac{\p}{|\q|}: (\p,\q)\in E\right\}$, is dense in some vector subspace ${H}$ of $\R^m$, and the ${\mathbb S}^{n-1}$-part of $\pi(E)$ is "big enough", where the latter notion depends on the codimension $k$ of ${H}$. It was designed to generalize Theorem \ref{not_on_line} and produce other interesting examples.} \comm{Please edit.}
The first one generalizes Theorem \ref{not_on_line} and produces other nontrivial examples. Its rough idea is {that (TDS) for $\fR$  can be verified if one shows that the closure of $E$ in $\tmn$  contains %set $\fR$ should be
a certain  $m$-dimensional submanifold $D$, 
with its $m$   dimensions split between (one or two) $(m-k)$-dimensional subspaces in the $\R^m$-part and (one or two) $k$-dimensional spheres in the ${\mathbb S}^{n-1}$-part of $\tmn$.} %have property (TDS+). 
{One can also work with $D$ being the product of one $(m-k)$-dimensional subspace and one $k$-sphere, and replace the assumption $D\subset \overline{E}$ with  $D\subset \overline{E\smallsetminus D}$. 
The second theorem has the same rough idea, but instead of subspaces in the $\R^m$-part we work with half-spaces. {We remark that %the statements of both  theorems above appear to be quite long; 
%however 
in \S\ref{optimality} we %will 
develop a certain formalism that allows to restate both theorems in a shorter and more conceptual way, at the same time demonstrating  their optimality.} 

%. The same formalism will be used to explain the optimality. 

%in \S\ref{optimality}
%in a shorter way 
%Both theorems will be   formalized and restated in \S\ref{optimality} %\comm{(I really hope so!)} 
%by introducing properties   (TDS$\pm$) of those subsets $D$ of $\tmn$.}

 %{In fact we will restate them in \S\ref{optimality} even in a more abstract way.} However we are going to

%with its $m$   dimensions split between an $(m-k)$-dimensional subspace in the $\R^m$-part and a $k$-dimensional sphere in the ${\mathbb S}^{n-1}$-part of $\tmn$. In fact, taking $\overline{\fR}$ exactly as small as the product of an $(m-k)$-dimensional subspace and a $k$-sphere is in general not enough for the existence of a totally dense subcollection of $\mathcal{L}_\fR$, see Lemma \ref{not_any_k}. However, adding one more $k$-sphere or one more $(m-k)$-dimensional subspace turns out to be sufficient.

\begin{theorem}\label{any_k_theorem}
    Let $\fR \subseteq { \bf T}_{m,n}$. Fix $0 \leq k \leq \min(n-1, m)$, and suppose %\textcolor{orange}
    {at least} one of the following three conditions holds:
    %\begin{enumerate}
        %\item\label
{\begin{equation}\label{anykA}\tag{\ref{any_k_theorem}a}
\begin{aligned}\text{there exist two {different}
$k$-dimensional subspheres }&\Phi_1,  \Phi_2\text{ of }\mathbb{S}^{n-1}\\ \text{ and an $(m-k)$-dimensional linear subspace }&{H} \text{ of }\mathbb{R}^m\quad\\  \text{ such that }
%\left(
         {H}\times \left(\Phi_1 \cup \Phi_2 \right) %\cap \mathbb{S}^{n-1} %\right) 
         \subseteq \overline{\fR};\quad
\end{aligned}
        \end{equation}
        %\item 
\begin{equation}\label{anykB}\tag{\ref{any_k_theorem}b}
\begin{aligned}\text{there exist a
$k$-dimensional subsphere }\Phi \text{ of }\mathbb{S}^{n-1}\\ \text{ and two  {different} $(m-k)$-dimensional linear subspaces }&{H}_1,{H}_2 \text{ of }\mathbb{R}^m\quad\\  \text{ such that }
%\left(
         \left( {H}_1 \cup {H}_2 \right) \times \Phi %\cap \mathbb{S}^{n-1} 
%\right) 
\subseteq \overline{\fR};\quad\qquad
\end{aligned}
%\text{there exists $\Phi$ and ${H}_1, {H}_2$ such that }%\left(
%\Phi %\cap \mathbb{S}^{n-1} 
%\right) 
%\times \left( {H}_1 \cup {H}_2 \right) \subseteq \overline{\fR},
\end{equation}
or \begin{equation}\label{anykC}\tag{\ref{any_k_theorem}c}
\begin{aligned}\text{there exist a
$k$-dimensional subsphere }\Phi \text{ of }&\mathbb{S}^{n-1}\\ \text{ and an $(m-k)$-dimensional linear subspace }&{H} \text{ of }\mathbb{R}^m\quad\\  \text{ such that }
%\left(
{H} %\cap \mathbb{S}^{n-1} %\right) 
       \times \Phi \subseteq \overline{\fR \smallsetminus \left( %\left( 
       {H} %\cap \mathbb{S}^{n-1} \right) 
       \times \Phi \right)}.\quad
\end{aligned}\end{equation}}
%\text{there exist $\Phi$ and ${H}$ such that }%\left( 
%       \Phi %\cap \mathbb{S}^{n-1} %\right) 
%       \times {H} \subseteq \overline{\fR \smallsetminus \left( %\left( 
%       \Phi %\cap \mathbb{S}^{n-1} \right) 
%       \times {H} \right)}.\end{equation}
%     
Then  %$\mathcal{L}_{\fR}$ contains a totally dense subcollection.
{$\fR$ has property {\rm (TDS)}}.
   % \end{enumerate}
\end{theorem}
%\comm{Is the following true:
%for any proper closed symmetric subset $D$ of $\left({H}_1 \cup {H}_2 \right) \times \Phi$ the collection $\mathcal{L}_{D}$ does not contain a totally dense subcollection? Yes.}

%\comm{Is the following true:
%for any
%$k$-dimensional subsphere $\Phi$   of $\mathbb{S}^{n-1}$, any two  {different} $(m-k)$-dimensional linear subspaces ${H}_1,{H}_2$   of $\mathbb{R}^m$, and    any proper closed symmetric subset $D$ of $\left({H}_1 \cup {H}_2 \right) \times \Phi$ there exists $E$ such  that $ D %\cap \mathbb{S}^{n-1} 
%\right) 
%\subseteq \overline{\fR}$, but the collection $\mathcal{L}_{E}$ does not contain a totally dense subcollection? These are the questions to consider when writing \S\ref{optimality}.}

{Let us briefly explain how the above theorem can be applied. If we are given a subset $E$ of $\tmn$ of the form  $E= \pi(P\times Q)$, where $\pi$ is as in \equ{generalpi}, to verify  %property 
(TDS) for $E$ using Theorem~\ref{any_k_theorem}(a) one needs to check that %the $\R^m$-part of $E$, that is, 
the set of ratios 
$${\left\{\left(\frac{\p}{\|\q\|}, \frac{\q}{\|\q\|} \right): \p\in P,\ \q\in Q\right\},}$$ is dense  in a product of some $(m-k)$-dimensional subspace ${H}$ of $\R^m$ and the union of two different $k$-spheres. %\comm{Vasya: This is not completely true. Maybe the set of ratios is dense, but it is reached by taking $\q$ with different directions, and so no set $(\theta, \phi \R_+)$ is in the closure. We need to replace $Q$ with $Q_{\varepsilon}^\phi$ for all $\Phi$ for this to work, however then it becomes hard to understand and we actually do it formally later. From the other hand, I don't see how it is more simple than the theorem itself. Do we really want this remark?}
Part (b) is similar, with one $(m-k)$-space and two $k$-spheres replaced by two $(m-k)$-spaces and one $k$-sphere. And to apply part (c) is a bit trickier: one needs to exhibit 
an $(m-k)$-dimensional subspace ${H}$ of $\R^m$ and a  $k$-subsphere $\Phi $ of  $\mathbb{S}^{n-1}$ such that ${H} %\cap \mathbb{S}^{n-1} %\right) 
       \times \Phi $ is in the closure of 
       $$
       \left\{\left(\frac{\p}{\|\q\|},\frac{\q}{\|\q\|}\right): \p\in P,\ \q\in Q\right\}\smallsetminus ({H} %\cap \mathbb{S}^{n-1} %\right) 
       \times \Phi ) =  \left\{\left(\frac{\p}{\|\q\|},\frac{\q}{\|\q\|}\right):  \text{ either }\p\notin H \text{ or } \frac{\q}{\|\q\|}\notin\Phi\right\}.
       $$}
%{Note  that \eqref{anykA} cannot hold when  $k=n-1$,  and  \eqref{anykB} cannot hold when $k=0$.} %\comm{Should we say anything else here?}

%\smallskip
%One can ask if {assuming that only one set of the form ${H} \times \Phi$ belongs} to $\overline{\fR}$ is ever enough to guarantee a totally dense subcollection. The answer is no. More specifically, conditions \eqref{anykA}, \eqref{anykB} and \eqref{anykC} are our ways to guarantee that after we remove a {set} of the form ${H} \times \Phi$ from $E$, we are still left with a dense collection of subspaces. The latter turns out to be a necessary condition for a collection to contain a totally dense subcollection; see Proposition \ref{not_any_k} for details. \comm{Please edit.}

\begin{remark}\label{illustration} Similarly to the discussion after Theorem \ref{not_on_line}, it is worthwhile to illustrate Theorem~\ref{any_k_theorem} by elaborating on the case $k=0$. Then ${H} = \R^m$, and $0$-dimensional subspheres are simply pairs of opposite points of $\mathbb{S}^{n-1}$. In fact %without loss of generality
 using the symmetrization argument discussed prior to Theorem~\ref{any_k_theorem} one can %replace each 
 simply work with single points instead of pairs. Thus in this case 
 %Theorem \ref{any_k_theorem} states that 
 {$\fR$ has property {\rm (TDS)}} if either 
%\begin{itemize}
%    \item 
\begin{equation}\label{ill_a}\tag{\ref{illustration}a}\text{there exist  non-opposite $\theta_1,\theta_2\in\mathbb{S}^{n-1}$ such that $\R^m\times \{\theta_1,\theta_2\} %\cap \mathbb{S}^{n-1} %\right) 
         \subseteq \overline{\fR}$,}
\end{equation}
or
\begin{equation}\label{ill_b}\tag{\ref{illustration}b}\text{there exists $\theta \in\mathbb{S}^{n-1}$ such that $\R^m\times \{\theta\} %\cap \mathbb{S}^{n-1} %\right) 
         \subseteq \overline{\fR\smallsetminus (\R^m\times \{\theta\})}$.}
         \end{equation}
\end{remark}
         
 %condition \eqref{anykA} is equivalent to the existence of non-opposite $\theta_1,\theta_2\in\mathbb{S}^{n-1}$ such that $\R^m\times \{\theta_1,\theta_2\} %\cap \mathbb{S}^{n-1} %\right) 
 %        \subseteq \overline{\fR}$, condition \eqref{anykC} is equivalent to the existence of $\theta \in\mathbb{S}^{n-1}$ such that $\R^m\times \{\theta\} %\cap \mathbb{S}^{n-1} %\right) 
 %        \subseteq \overline{\fR\smallsetminus (\R^m\times \{\theta\})}$, and condition \eqref{anykB} cannot hold.

\smallskip
Our second main theorem is meant to address the asymmetric situation   when %the set of ratios %$\left\{\frac{\p}{\|\q\|}: (\p,\q)\in E\right\}$ 
%\equ{ratios} is dense only in 
$\R^m$ is %not a linear space but 
replaced by a half-space ${H}$  of $\mathbb{R}^{m}$. %In fact for simplicity we will consider only half-spaces of full dimension. 
 Clearly it can be reduced to the   case ${H} = \R^m$ %of Theorem \ref{any_k_theorem} 
 by symmetrization under the additional assumption $Q = -Q$.
 %replacing   $\fR$ with $\fR\cup (-\fR)$. 
 It turns out that in the half-space case the symmetry condition can be weakened to the assumption that certain subsets of $\mathbb{S}^{n-1}$ contain the origin in their closed convex hulls.
 
 %But what if the symmetry of $Q$ is not assumed? 
 %The next statement generalizes the either--or clause mentioned in the previous paragraph by 
 %It turns out that in this case the sets $\{\theta_1,\theta_2\}$ and $\{\theta\}$ mentioned in the preceding paragraph 
 %can be replaced by  certain 
 %introducing subsets $\Theta_1,\Theta_2$ of $\mathbb{S}^{n-1}$ satisfying some additional  conditions.}
%However the above theorem is not applicable when $Q$ is not symmetric and $P$, for example, consists of vectors with a nonnegative first coordinate. 
%These are the kind of situations addressed by the next theorem.}

\begin{theorem}\label{ch_theorem}
    Let $\fR \subseteq { \bf T}_{m,n}$. Suppose there exist disjoint {(possibly empty)} sets $\Theta_1, \Theta_2 \subseteq \mathbb{S}^{n-1}$ and a closed linear half-space ${H}$ of $\mathbb{R}^{m}$ such that the following {three} conditions hold:
%    \begin{enumerate}
        %\item\label{bigtheta1} 
\begin{equation}\label{bigtheta1}\tag{\ref{ch_theorem}a} {H}\times\Theta_1  \subseteq \overline{\fR},
    \end{equation}
\begin{equation}\label{bigtheta2}\tag{\ref{ch_theorem}b} {H} \times \{\theta\}
 \subseteq \overline{\fR \smallsetminus \left( {H} \times \{\theta\} \right)} \text{ for any }\theta \in \Theta_2,
    \end{equation} and 
\begin{equation}\label{ch_theorem_condition}\tag{\ref{ch_theorem}c}
    0 \in \begin{cases}\conv \big( \overline{\Theta_1 \cup \Theta_2 \smallsetminus \{ \theta \}} \big) %\,\,\,\,\,\,\,\, 
    \text{ for any} \,\, \theta \in \Theta_1 &\text{if } \Theta_1\ne\varnothing,\\
    \conv \left( \overline{ \Theta_2} \right)&\text{if } \Theta_1 = \varnothing.\end{cases}
\end{equation}
%\begin{equation}\label{ch_theorem_condition}\tag{2.1c}
%    0 \in \conv \left( \overline{\Theta_1 \cup \Theta_2 \smallsetminus \{ \theta \}} \right) \,\,\,\,\,\,\,\, \text{for any} \,\, \theta \in \Theta_1.
 %   \end{equation}
 %   If $\Theta_1$ is empty and condition \eqref{ch_theorem_condition} is trivial, we need to check that \comm{я не понимаю формулировки we need to check, это то же самое что assume that?}
%\begin{equation}\label{bigtheta3}\tag{2.1d}0 \in \conv \left( \overline{\Theta_1 \cup \Theta_2} \right).
%    \end{equation}
    Then {$\fR$ has property {\rm (TDS)}}.
    (Here and hereafter $\conv (\Theta)$ stands for the convex hull of $\Theta$ in $\R^n$.)
    %$\mathcal{L}_{\fR}$ contains a totally dense subcollection.
\end{theorem}

%{We refer the reader to Definition \ref{generaltds} where conditions \eqref{bigtheta1} and \eqref{bigtheta2} are put in a more general context.}

%Note that if one in addition assumes that the sets $\Theta_1, \Theta_2$ are symmetric around the origin, then to verify \eqref{ch_theorem_condition}
%holds as long as one of them is 
%it is enough to check that either $\Theta_1$ is not contained in a straight line passing through the origin, or $\Theta_2\ne\varnothing$. However some interesting applications exist for non-symmetric $\Theta_i$.} 

{We remark that the introduction of two sets $\Theta_1$, $\Theta_2$ in the above theorem serves as a generalization of the either--or clause mentioned in Remark \ref{illustration}. Indeed,
to obtain conditions \eqref{ill_a}  or \eqref{ill_b}  one can take $\Theta_1 = \{\pm\theta_1,\pm\theta_2\}$ and $\Theta_2 = \varnothing$, or, respectively, $\Theta_2 = \{\pm\theta\}$ and $\Theta_1 = \varnothing$.
In fact, quite often Theorem \ref{ch_theorem} will be applied with one of the sets $\Theta_i$ being empty. It is instructive to simplify the assumptions of the theorem in these cases: when $\Theta_2= \varnothing$, one  assumes the existence of $\Theta_1  \subseteq \mathbb{S}^{n-1}$ satisfying   $$0 \in  \conv \big( \overline{\Theta_1 \smallsetminus \{ \theta \}} \big)  %\,\,\,\,\,\,\,\, 
  \text{ for any }\theta \in \Theta_1$$
  and \eqref{bigtheta1}.  
And  when $\Theta_1= \varnothing$, one  assumes the existence of $\Theta_2  \subseteq \mathbb{S}^{n-1}$ satisfying  $0 \in \conv \left( \overline{ \Theta_2} \right)$ and \eqref{bigtheta2}. %\comm{Not sure if it makes sense to have this remark, but I personally find it helpful.} 
However interesting consequences of Theorem \ref{ch_theorem} exist when both $\Theta_1$ and $\Theta_2$ are nonempty.}

{We also remark} that, in the presence of \eqref{bigtheta1} and \eqref{bigtheta2}, condition \eqref{ch_theorem_condition} %of  the condition that the convex hull contains the origin 
{happens to be} not only sufficient, but also necessary to guarantee {(TDS) for $\fR$}.
%the existence of a totally dense subcollection  {of $\mathcal{L}_{\fR}$}. 
For  details, see Proposition \ref{ch_thm_optimal_new}.  

\smallskip
%As can be easily seen, Theorem \ref{ch_theorem} deals with the case when the set of ratios $\left\{\frac{\p}{\|\q\|}: \p\in P,\,\q\in Q\right\}$ is dense in some half-space of $\R^m$. However this is not assumed in some natural Diophantine set-ups, such as the one described in Theorem \ref{8.1}(a). 

\smallskip
In the remaining part of the section we show that several special cases described in the introduction, as well as some other natural examples, are consequences of these theorems.

%\comm{Maybe this should be moved further in the text?}

\subsection{Proof of  Theorem \ref{not_on_line}}\label{pf1.6}

As the first application, we %will 
show how Theorem \ref{any_k_theorem} implies Theorem \ref{not_on_line}. %Let us say that 
%$\theta\in \mathbb{S}^{n-1}$ is a \underbar{limit direction} of $Q$ if 
%$\theta = \lim_{i\to\infty}\frac{\q_{i}}{|\q_i|}$ for some sequence $\{ \q_i \}_{i=1}^{\infty}$ %be a sequence 
%    of elements of $Q$ with
%    $\lim\limits_{i \rightarrow \infty} |\q_i| = \infty$. \comm{This definition is used only in Observation \ref{replacement of compatibility}, right? and even then we can probably be OK without it. -- Now also in proposition \ref{examples_for_anyk}, and I believe it's the easiest way to formulate the latter}
First we %will show a 
{establish the following} simple fact.

\begin{lemma}\label{compatibility_from_logdense}
    Suppose {a nonempty subset  $P$ of $ \mathbb{R}^m$} is %a 
    logarithmically dense in some 
    %$s$-dimensional 
    cone ${H} \subseteq \mathbb{R}^m$, %$0 \leq s \leq m$, 
    and $Q\subset \R^n\nz$ is unbounded. Then 
    the set %\equ{ratios}
\begin{equation}\label{comp_logdens_lemma_condition}%\tag{2.3}
        \left\{ \frac{\p}{\|\q\|}:      \p \in P, \, \q \in Q  \right\}
    \end{equation}%  
is dense in ${H}$.
\end{lemma}

\begin{proof} If %$s = 0$, the statement is trivial: 
{${H} = \{ 0 \}$, the statement is trivial:}   since $Q$ is unbounded, {the set \eqref{comp_logdens_lemma_condition} 
%\equ{ratios} 
has $0$  in its closure}. %From now on assume that $s \geq 1$. 
{Otherwise fix} an arbitrary $\phi \in \mathbb{S}^{m-1} \cap {H}$, $\varepsilon > 0$ and %real numbers 
$0 < a < b$;  it is enough to show that %for any 
%. Fix an arbitrary .
there exist %such 
$\p \in P_{\varepsilon}^{\phi}$ and $\q \in Q$ such that $\frac{\|\p\|}{\|\q\|} \in (a, b)$.
%(this notation was defined in \equ{pepsilon}).
    {Since $P$ is logarithmically dense in ${H}$, there exists a sequence} $\{ \p_i \}_{i=1}^{\infty}$ %be a sequence 
    of elements of $P_{\varepsilon}^{\phi}$ such that 
    $$
    \|\p_{i+1}\| > \|\p_i\| \text{ for all } i\in \N, \,\,\,\,\,\,\,\, \lim\limits_{i \rightarrow \infty} \|\p_i\| = \infty \,\,\,\,\,\,\,\, \text{and} \,\,\,\,\,\,\,\, \lim\limits_{i \rightarrow \infty} \frac{\|\p_{i+1}\|}{\|\p_i\|} = 1.
    $$
    %such a sequence exists . 
    Let $N \in \mathbb{N}$ be such that for any $i \geq N$ one has $1 < \frac{\|\p_{i+1}\|}{\|\p_i\|} < \frac{b}{a}$, and let us fix $\q \in Q$ such that $a\|\q\| > \|\p_N\|$. We will show that $\frac{|\p_{i}|}{|\q|} \in (a, b)$ for some $i > N$. Let $i-1$ be the maximal index such that $\|\p_{i-1}\| \leq a \|\q\|$; it is well defined by {the} construction of $\{ \p_i \}$. {It follows that} $\|\p_i\| > a\|\q\|$. {On} the other hand,   {$i-1 \geq N$ by definition of $\q$, hence}
    $$
    \|\p_i\| < \|\p_{i-1}\| \frac{b}{a} \leq b \|\q\|, 
    $$
    which shows that $a < \frac{\|\p_i\|}{\|\q\|} < b$.
\end{proof}

\begin{proof}[Proof of Theorem \ref{not_on_line}] Recall that we are given a set $P \subseteq \mathbb{R}^m$ {that is} logarithmically dense in {an $(m-k)$-dimensional subspace ${H}$ of $\R^m$} and a subset $Q$ of $ \R^n\nz$ such that \equ{twohalfspheres} holds. %Note that by Lemma \ref{halfsphere_convergence} 
{From the compactness of the Grassmanian of $k$-dimensional subspaces of $\R^n$ it follows that} there exists a $k$-halfsphere $\Phi \subseteq \mathbb{S}^{n-1}$ such that  $$\Phi \subseteq \bigcap\limits_{C > 0} \overline{\pi\big(Q\smallsetminus B_C(0)\big)}.$$
    {Now define $\fR := \pr(P \times Q) \cup \big(-\pr(P \times Q)\big);$ it is enough to show} that $\mathcal{L}_{\fR}$ contains a totally dense subcollection. {Let us consider two cases.}

\medskip

%\begin{enumerate}
    %\item[
    \noindent{\bf Case 1.} There exist {two  $k$-halfspheres $\Phi_1$ and $\Phi_2$ not contained in a single $k$-sphere such that}
    $$
    \Phi_1 \cup \Phi_2 \subseteq \bigcap\limits_{C > 0} \overline{\pi\big(Q\smallsetminus B_C(0)\big)}.
    $$
    Fix $\varepsilon > 0$. For any $\theta \in \Phi_i, \,\, i=1, 2,$ the set $Q_{\varepsilon}^{\theta}$ is unbounded, thus we can apply Lemma \ref{compatibility_from_logdense} to show that for any ${\bf x} \in {H}$ there exists $\q \in Q_{\varepsilon}^{\theta}$ and $\p \in P$ such that
    $$
    \left\|\theta -  \frac{\q}{\|\q\|} \right\| < \varepsilon \,\,\,\, \text{and} \,\,\,\, \left\| { \bf x } - \frac{\p}{\|\q\|} \right\| < \varepsilon \,\,\,\,\,\,\,\, \iff \,\,\,\,\,\,\,\, \dist\Big( \pr\big((\p, \q)\big), ({\bf x},\theta) \Big) < \varepsilon.
    $$
    %\comm{(I guess for the above equivalence we need to work with the distance induced by the supremum norm.)} 
    Since $\varepsilon > 0, \, \phi \in \Phi_i$ and ${ \bf x } \in {H}$ were chosen arbitrarily, we conclude that 
    $$
    {H}\times \left( \Phi_1 \cup \Phi_2 \right)  \subseteq \overline{\pr(P \times Q)} \subseteq \overline{\fR}.
    $$
    Since {${H} = -{H}$  and  $\fR = -\fR$, %is centrally symmetric, 
    we can %use the fact 
    conclude that $- \big({H}\times( \Phi_1 \cup \Phi_2 ) \big) = {H}\times\big( (-\Phi_1) \cup (-\Phi_2) \big)$, and %since $\fR = -\fR$, we conclude that 
    hence} ${H}\times\big( ( -\Phi_1) \cup (-\Phi_2) \big)   \subseteq \overline{\pr(P \times Q)} \subseteq \overline{\fR}.$ We can now change the notation and denote the whole $k$-sphere containing $\Phi_i$ by $\Phi_i$; then \eqref{anykA} holds, and thus by Theorem \ref{any_k_theorem} %$\mathcal{L}_{\fR}% = \mathcal{L}_{P, Q}
    %$ contains a totally dense subcollection.
    {$\fR$ has property {\rm (TDS)}}.

\medskip

    %\item[
    \noindent{\bf Case 2.}   There exists %such 
    {a $k$-halfsphere $\Phi$ such that}
    $$
    \Phi \subseteq \bigcap\limits_{C > 0} \overline{\pi\big(Q\smallsetminus B_C(0)\big) \smallsetminus \Phi}.
    $$
    Note that {in view of \equ{twohalfspheres}} for any $C>0$ one of the halfspheres $\Phi_1$ or $\Phi_2$ does not belong to the same $k$-sphere as $\Phi$. Similarly to Case 1 we can show that 
    $$
    {H} \times \Phi \subseteq \overline{\pr(P \times Q) \smallsetminus \left( {H} \times \Phi \right)} \subseteq \overline{\fR\smallsetminus \left( {H} \times \Phi \right)}.
    $$
    Replacing $\Phi$ with the corresponding sphere and again using {the} symmetry of ${H}$ and $\fR$, we conclude that \eqref{anykC} holds. {Hence}, {again by Theorem \ref{any_k_theorem}, $\fR$ has property {\rm (TDS)}}.  %$\mathcal{L}_{\fR} %= \mathcal{L}_{P, Q}
    %$ contains a totally dense subcollection.
\end{proof}

%\comm{Maybe mention that it can be also derived from the other theorem. -- It's actually not true now, when our Theorem 1.6 is not only for k=0.}
%\comm{This proof produces two questions: do we have meaningful examples when (1) both $\Theta_1$ and $ \Theta_2$ are nonempty? (2) ${H}$ is an honest half-space? It will be nice to mention them here.}
% And we should also remark that the more restrictive conditions of Theorem \ref{8.1}(b) make it possible to strengthen the conclusion to the total density of the collection itself, referring the reader to \S\ref{proof8.1}.}

\subsection{Other applications and examples}\label{appl}
Let us say that 
$\theta\in \mathbb{S}^{n-1}$ is a \underbar{limit direction} of $Q$ if 
$\theta = \lim_{i\to\infty}\frac{\q_{i}}{\|\q_i\|}$ for some sequence $\{ \q_i \}_{i=1}^{\infty}$ %be a sequence 
    of elements of $Q$ with
    $\lim\limits_{i \rightarrow \infty} \|\q_i\| = \infty$, {with 
    %the same definition of 
    $\phi\in \mathbb{S}^{m-1}$ 
    {defined analogously as}  a limit direction of $P$.} 
    %\comm{This definition is used only in Observation \ref{replacement of compatibility}, right? and even then we can probably be OK without it. -- Now also in proposition \ref{examples_for_anyk}, and I believe it's the easiest way to formulate the latter} 
    The following simple observation is useful to check the conditions of {Theorems  %Theorem 
    \ref{any_k_theorem} and \ref{ch_theorem}}:

\begin{observation}\label{replacement of compatibility}
    Let $\theta$ be a limit direction of $Q$ and $\phi$ be a limit direction of $P$. If for any $\varepsilon > 0$ 
\begin{equation}\label{Obs_24_set}%\tag{2.4}
    {\text{the set }\left\{ \frac{x}{y}: \,\,\, x \in |P|_{\varepsilon}^{\phi}, \, y \in |Q|_{\varepsilon}^{\theta} \right\}\text{is dense in }\mathbb{R}_{\geq 0},}
    \end{equation}
     then $ \mathbb{R}_{\geq 0} \phi\times  \{ \theta \} \subseteq %\overline{\fR}$ (where $\fR = 
    \pr(P \times Q)$.
\end{observation}

%{ 
Due to Lemma \ref{compatibility_from_logdense} and Observation \ref{replacement of compatibility}, logarithmically dense sets provide a variety of examples for {both of our main theorems}. 
%Theorem \ref{ch_theorem} and Theorem \ref{any_k_theorem}. 
%We start with 
{Let us now consider} some examples where Theorem \ref{ch_theorem} can be applied.

\begin{example}\label{ch_examples}
    Let $n =2$, $ m = 1$, and   let $P \subseteq \R_{\geq 0}$ be any  set logarithmically dense in $\R_{\geq 0}$
    %; for instance, $P = \mathbb{N}, \, P = \mathbb{P}$ or $P = \{ k^3: \,\,\, k \in \mathbb{N} \}$ 
    %(hereafter we denote the set of prime numbers by $\mathbb{P}$).
    {(see \S\ref{logdense_expl} for examples). Consider the following cases:}
    \begin{enumerate}
        \item[(i)] %Let 
        $\bv_1 = (1,0)%^{\top}
        $, $\bv_2 = (1,2)%^{\top}
        $, $\bv_3 = (-2,1)%^{\top}
        $, $\bv_4 = (-2,-1),%^{\top}
        $  $\bv_5 = (1,-2)%^{\top}
        $, and $Q = \bigcup_{i=1}^5\mathbb{N} \bv_i$.
        %\cup \ldots \cup \mathbb{N} \bv_5$.
        \item[(ii)] %Let 
        $\bv_1 = (1,0)%^{\top}
        $, $\bv_2 = (-1,2)%^{\top}
        $,  $\bv_3 = (-1,-2)%^{\top}
        $, and  $Q = (0, 1)%^{\top} 
        + \bigcup_{i=1}^3\mathbb{N} \bv_i$.
        %\left( \mathbb{N} \bv_1 \cup \mathbb{N} \bv_2 \cup \mathbb{N} \bv_3 \right)$ .
        \item[(iii)] %Let 
        $Q = \mathbb{Z} \times (\mathbb{N} + 10) %= \{ (x, y)^{\top} \in \mathbb{Z}^2: \,\,\, y > 1000 \}
        $.
        \item[(iv)]  %Let 
        $Q = \big\{ (s, s^2), (s, -s^4): s \in \Z \big\}\smallsetminus \{(0,0)\}$.
        \item[(v)] %Let 
        $Q = \big\{ (s, s^2), (s, -|s|), (0, -|s|): s \in \Z \big\}\smallsetminus \{(0,0)\} $. 

        \item[(vi)]   $\{ a_s\}$ is an unbounded increasing sequence of positive real numbers equidistributed modulo $\pi$, and $Q = \{ (a_s \cos a_s, a_s \sin a_s) \}$ {(in particular $a_s = s$ works)}.
    \end{enumerate}

%-----------------------------------------------------------------------------------------------------------------------------------------------------------------------------------------------------------------------------------------------------------------------------------------------------------------------------------------------------------------------------------

\begin{figure}[htbp]
\centering

\begin{minipage}{0.32\textwidth}
\centering
% --- TikZ block 1 ---
\begin{tikzpicture}[scale=0.25]

    % ---------- PARAMETERS ----------
    \def\xmin{-10}
    \def\xmax{10}
    \def\ymin{-10}
    \def\ymax{10}
    \def\ptsize{4pt}
    \def\smax{12}
    % --------------------------------

    % Axes (no captions)
    \draw[->] (\xmin,0) -- (\xmax,0);
    \draw[->] (0,\ymin) -- (0,\ymax);

    % Light integer grid
    \foreach \x in {\xmin,...,\xmax} {
        \foreach \y in {\ymin,...,\ymax} {
            \fill[gray!40] (\x,\y) circle (0.6pt);
        }
    }

   % vectors
\coordinate (v1) at (1,0);
\coordinate (v2) at (1,2);
\coordinate (v3) at (-2,1);
\coordinate (v4) at (-2,-1);
\coordinate (v5) at (1,-2);

% draw dotted rays
\draw[dotted] (0,0) -- (10,0);
\draw[dotted] (0,0) -- (5,10);
\draw[dotted] (0,0) -- (-10,5);
\draw[dotted] (0,0) -- (-10,-5);
\draw[dotted] (0,0) -- (5,-10);

% draw the generating vectors as arrows
\draw[->,thick] (0,0) -- (1,0);
\draw[->,thick] (0,0) -- (1,2);
\draw[->,thick] (0,0) -- (-2,1);
\draw[->,thick] (0,0) -- (-2,-1);
\draw[->,thick] (0,0) -- (1,-2);

% labels for vectors
\node[above] at (1.8,-0.4) {$\mathbf v_1$};
\node[right] at (1,2) {$\mathbf v_2$};
\node[left] at (-2,1) {$\mathbf v_3$};
\node[left] at (-2,-1) {$\mathbf v_4$};
\node[right] at (1,-2) {$\mathbf v_5$};

% integer lattice points up to k = 8
\foreach \k in {1,...,9} {
    \fill (\k,0) circle (4pt);}
\foreach \k in {1,...,4} {    
    \fill (\k,2*\k) circle (4pt);
    \fill (-2*\k,\k) circle (4pt);
    \fill (-2*\k,-\k) circle (4pt);
    \fill (\k,-2*\k) circle (4pt);
}

% origin
\fill (0,0) circle (2pt);

\end{tikzpicture}

\par\smallskip
\caption{{The set $Q$ in Case} (i)}\label{fig1}
\end{minipage}
\hfill
\begin{minipage}{0.32\textwidth}
\centering
% --- TikZ block 3 ---
\begin{tikzpicture}[scale=0.25]

    % ---------- PARAMETERS ----------
    \def\xmin{-10}
    \def\xmax{10}
    \def\ymin{-10}
    \def\ymax{10}
    \def\ptsize{4pt}
    \def\smax{12}
    % --------------------------------

    % Axes (no captions)
    \draw[->] (\xmin,0) -- (\xmax,0);
    \draw[->] (0,\ymin) -- (0,\ymax);

    % Light integer grid
    \foreach \x in {\xmin,...,\xmax} {
        \foreach \y in {\ymin,...,\ymax} {
            \fill[gray!40] (\x,\y) circle (0.6pt);
        }
    }

    % base point
\coordinate (s) at (0,1);

% vectors (as displacements)
\coordinate (v1) at (1,0);
\coordinate (v2) at (-1,2);
\coordinate (v3) at (-1,-2);

% dotted rays from (0,1)
\draw[dotted] (s) -- ++(10,0);
\draw[dotted] (s) -- ++(-4.5,9);
\draw[dotted] (s) -- ++(-5.5,-11);

% draw vectors from (0,1)
\draw[->,thick] (s) -- ++(1,0);
\draw[->,thick] (s) -- ++(-1,2);
\draw[->,thick] (s) -- ++(-1,-2);

% vector labels (placed away from lines)
\node[above] at (1.2,1.2) {$\mathbf v_1$};
\node[left] at (-1.3,3.2) {$\mathbf v_2$};
\node[left] at (-1.3,-1.2) {$\mathbf v_3$};

% integer points: (0,1) + k v_i, k=1..6
\foreach \k in {1,...,9} {
    \fill ({\k},{1}) circle (4pt); }

\foreach \k in {1,...,4} {    
    \fill ({-\k},{1+2*\k}) circle (4pt);}

\foreach \k in {1,...,5} {   
    \fill ({-\k},{1-2*\k}) circle (4pt);
}

% base point
\fill (0,1) circle (2pt) node[left] {$(0,1)$};

% origin (for reference)
\fill (0,0) circle (1.2pt);

\end{tikzpicture}

\par\smallskip
\caption{{The set $Q$ in Case} (ii)}\label{fig2}
\end{minipage}
\hfill\begin{minipage}{0.32\textwidth}
\centering
% --- TikZ block 1 ---
\begin{tikzpicture}[scale=0.25]

    % ---------- PARAMETERS ----------
    \def\xmin{-10}
    \def\xmax{10}
    \def\ymin{-1}
    \def\ymax{19}
    \def\ptsize{4pt}
    \def\smax{12}
    % --------------------------------

    % Axes (no captions)
    \draw[->] (\xmin,0) -- (\xmax,0);
    \draw[->] (0,\ymin) -- (0,\ymax);

    % Light integer grid
    \foreach \x in {\xmin,...,\xmax} {
        \foreach \y in {\ymin,...,\ymax} {
            \fill[gray!40] (\x,\y) circle (0.6pt);
        }
    }

    % -------- integer points (solid black) ----------
    \foreach \s in {-10,...,10} \foreach \k in {10,...,18} {
        \pgfmathsetmacro{\xx}{\s}
        \pgfmathsetmacro{\yy}{\k}
        \fill (\xx,\yy) circle (2pt);
    }

\end{tikzpicture}

\par\smallskip
\caption{{The set $Q$ in Case} (iii)}\label{fig3}
\end{minipage}

\end{figure}

%-----------------------------------------------------------------------------------------------------------------------------------------------------------------------------------------------------------------------------------------------------------------------------------------------------------------------------------------------------------------------------------

    Let $E := {\pr(P \times Q)}$; we claim that in all of the above cases  the collection $\mathcal{L}_{\fR} =\mathcal{L}_{P, Q}$ contains a totally dense subcollection.  Indeed:
\begin{enumerate}
    \item[(i)] Let
    $
    \Theta_1 = \left\{ \frac{\bv_i}{\|\bv_i\|} : i = 1, \ldots, 5 \right\}$ and  $\Theta_2 = \varnothing$.
    It follows from Observation \ref{replacement of compatibility} and Lemma \ref{compatibility_from_logdense} that $ \mathbb{R}_{\geq 0} \times \Theta_1\subseteq \overline{\fR}$; %= 
    %\overline{\pr(P \times Q)}$; 
    thus \eqref{bigtheta1} holds. It is also clear that \eqref{ch_theorem_condition} holds, since any four out of five vectors $\bv_i$ contain zero in convex hull, hence Theorem \ref{ch_theorem} 
    applies.
    %the collection $%\mathcal{L}_{\fR} =
    %\mathcal{L}_{P, Q}$ contains a totally dense subcollection.

    \item[(ii)] Here we take
    $
    \Theta_1 = \varnothing$ and $\Theta_2 = \left\{ \frac{\bv_i}{\|\bv_i\|}: i = 1, 2, 3 \right\} 
    $.
    Again by Observation \ref{replacement of compatibility} and Lemma \ref{compatibility_from_logdense} we see that $ \mathbb{R}_{\geq 0} \times\Theta_2 \subseteq \overline{\fR}$; 
    %= 
    %\overline{\pr(P \times Q)}$; 
    and since $\mathbb{R}_{\geq 0} \bv_i \,\cap\, Q = \varnothing$, \eqref{bigtheta2} holds. Condition \eqref{ch_theorem_condition} holds as well since $0 \in \conv\left( \{ \bv_1, \bv_2, \bv_3 \} \right)$, thus Theorem \ref{ch_theorem} is applicable again.
    %the collection $\mathcal{L}_{\fR} = \mathcal{L}_{P, Q}$ contains a totally dense subcollection.

    \item[(iii)] Let $\Theta_1 = \big\{ \theta = (\theta_1, \theta_2)%^{\top} 
    \in \mathbb{S}^1: \,\,\, \theta_2 > 0 \big\}$. %\theta = (\theta_1, \theta_2)^{\top} \in \mathbb{S}^1$ and suppose $\theta_2 > 0$. 
    It is not hard to see that for any $\theta\in\Theta_1$ and $\varepsilon > 0$ the set
    $
    \left\{ \frac{\|\p\|}{\|\q\|}:  \p \in P, \, \q \in Q_{\varepsilon}^{\theta} \right\}
    $
    is dense in $\mathbb{R}_{\geq 0}$. So, by Observation \ref{replacement of compatibility}, 
    ${\R_{\geq 0}\times\Theta_1} \subseteq \overline{\fR}$. 
    %= 
    %\overline{\pr(P \times Q)}$; 
    Since $$\overline{\Theta_1 \smallsetminus \{ \theta \}} = \big\{(\theta_1, \theta_2)%^{\top}
    \in \mathbb{S}^1: \,\,\, \theta_2 \geq 0 \big\} \supseteq \{ (\pm 1, 0)%^{\top}
    \} \,\,\,\,\,\,\,\, \text{for any} \,\, \theta \in \Theta_1,$$  %condition \eqref{ch_theorem_condition} holds
    it follows that $0 \in  \conv \big(  \overline{\Theta_1 \smallsetminus \{ \theta \}} \big)$, and Theorem \ref{ch_theorem}
    can be applied with  $\Theta_2 = \varnothing$. %the collection $\mathcal{L}_{\fR} = \mathcal{L}_{P, Q}$ contains a totally dense subcollection.
{Note that here it is not true that $0 \in  \conv \big(  {\Theta_1 \smallsetminus \{ \theta \}} \big)$, so it is essential to take the closure.}
\end{enumerate}

%-----------------------------------------------------------------------------------------------------------------------------------------------------------------------------------------------------------------------------------------------------------------------------------------------------------------------------------------------------------------------------------

\begin{figure}[htbp]
\centering

\begin{minipage}{0.32\textwidth}
\centering
% --- TikZ block 3 ---
\begin{tikzpicture}[scale=0.25]\label{(iv)-pic}

    % ---------- PARAMETERS ----------
    \def\xmin{-10}
    \def\xmax{10}
    \def\ymin{-16}
    \def\ymax{4}
    \def\ptsize{4pt}
    \def\smax{12}
    % --------------------------------

    % Axes (no captions)
    \draw[->] (\xmin,0) -- (\xmax,0);
    \draw[->] (0,\ymin) -- (0,\ymax);

    % Light integer grid
    \foreach \x in {\xmin,...,\xmax} {
        \foreach \y in {\ymin,...,\ymax} {
            \fill[gray!40] (\x,\y) circle (0.6pt);
        }
    }

    % -------- dotted зфкфищдф ----------
    \draw[dotted, domain=-2:2, samples=40, smooth]
        plot ({\x}, {\x*\x});

    \draw[dotted, domain=-2:2, samples=40, smooth]
        plot ({\x}, {-\x*\x*\x*\x});

    % -------- integer points (solid black) ----------
    \foreach \s in {-2,...,2} {
        \pgfmathsetmacro{\xx}{\s}
        \pgfmathsetmacro{\yy}{\s*\s}
        \fill (\xx,\yy) circle (\ptsize);
    }

    \foreach \s in {-2,...,2} {
        \pgfmathsetmacro{\xx}{\s}
        \pgfmathsetmacro{\yy}{-\s*\s*\s*\s}
        \fill (\xx,\yy) circle (\ptsize);
    }

\end{tikzpicture}
\par\smallskip
\caption{{The set $Q$ in Case} (iv)}\label{fig4}
\end{minipage}
\hfill
\begin{minipage}{0.32\textwidth}
\centering
% --- TikZ block 1 ---
\begin{tikzpicture}[scale=0.25]

    % ---------- PARAMETERS ----------
    \def\xmin{-10}
    \def\xmax{10}
    \def\ymin{-10}
    \def\ymax{10}
    \def\ptsize{4pt}
    \def\smax{12}
    % --------------------------------

    % Axes (no captions)
    \draw[->] (\xmin,0) -- (\xmax,0);
    \draw[->] (0,\ymin) -- (0,\ymax);

    % Light integer grid
    \foreach \x in {\xmin,...,\xmax} {
        \foreach \y in {\ymin,...,\ymax} {
            \fill[gray!40] (\x,\y) circle (0.6pt);
        }
    }

    % -------- dotted зфкфищдф ----------
    \draw[dotted, domain=-3:3, samples=40, smooth]
        plot ({\x}, {\x*\x});

        % draw dotted rays
\draw[dotted] (0,0) -- (-9,-9);
\draw[dotted] (0,0) -- (9,-9);
\draw[dotted] (0,0) -- (0,-9);

    % -------- integer points (solid black) ----------
    \foreach \s in {-3,...,3} {
        \pgfmathsetmacro{\xx}{\s}
        \pgfmathsetmacro{\yy}{\s*\s}
        \fill (\xx,\yy) circle (\ptsize);
    }

    % integer lattice points up to k = 8
\foreach \k in {1,...,9} {
    \fill (0,-\k) circle (\ptsize);
    \fill (-\k,-\k) circle (\ptsize);
    \fill (\k,-\k) circle (\ptsize);
    }

\end{tikzpicture}

\par\smallskip
\caption{{The set $Q$ in Case} (v)}\label{fig5}
\end{minipage}
\hfill
\begin{minipage}{0.32\textwidth}
\centering
% --- TikZ block 3 ---
\begin{tikzpicture}[scale=0.25]

    % ---------- PARAMETERS ----------
    \def\xmin{-10}
    \def\xmax{10}
    \def\ymin{-12}
    \def\ymax{8}
    \def\ptsize{4pt}
    \def\smax{12}
    % --------------------------------

    % Axes (no captions)
    \draw[->] (\xmin,0) -- (\xmax,0);
    \draw[->] (0,\ymin) -- (0,\ymax);

    % Light integer grid
    \foreach \x in {\xmin,...,\xmax} {
        \foreach \y in {\ymin,...,\ymax} {
            \fill[gray!40] (\x,\y) circle (0.6pt);
        }
    }

    % -------- dotted spiral curve ----------
    \draw[dotted, domain=0:12, samples=200, smooth]
        plot ({\x*cos(\x r)}, {\x*sin(\x r)});

    % -------- integer points (solid black) ----------
    \foreach \s in {1,...,\smax} {
        \pgfmathsetmacro{\xx}{\s*cos(\s r)}
        \pgfmathsetmacro{\yy}{\s*sin(\s r)}
        \fill (\xx,\yy) circle (\ptsize);
    }

\end{tikzpicture}

\par\smallskip
\caption{{The set $Q$ in Case} (vi)}\label{fig6}
\end{minipage}

\end{figure}

%-----------------------------------------------------------------------------------------------------------------------------------------------------------------------------------------------------------------------------------------------------------------------------------------------------------------------------------------------------------------------------------

\begin{enumerate}
    \item[(iv)]   Let 
    $    \Theta_1 = \varnothing$ and $\Theta_2 = \left\{ \pm(0, 1)%^{\top} 
    \right\}$.
    Then for any $\theta \in \Theta_2$ and any $\varepsilon > 0$ the set $Q_{\varepsilon}^{\theta} \smallsetminus \R_{\geq 0} \theta$ is unbounded (it contains all except finitely many points from one of the parabolas). It means that 
    \eq{closurecomplement}{{\R_{\geq 0}\times\Theta_2} \subseteq \overline{\fR \smallsetminus \left( {\R_{\geq 0}\times\Theta_1} \right)}.} Since $0 \in \conv(\Theta_2)$, the result follows. %\comm{A question: can we replace this example by $Q = \big\{ (s, \pm s^2): s \in \N\}$? -- The theorem will work, however it will give us a symmetric set $Q$, which we can turn into a symmetric P and then apply Theorem 1.6. The current example does not follow from 1.6. DK: why symmetric? $-Q = \{ (-s, \pm s^2): s \in \N\}\ne Q$, how do we use Theorem 1.6 then? -- That's true, it's 2.1, not 2.2. }
    
    \item[(v)] In this case one can take 
    $   \Theta_1 = \left\{ (0, -1)%^{\top}
    , \left(\pm \frac{1}{\sqrt{2}}, -\frac{1}{\sqrt{2}}\right)%^{\top}
    \right\}$ and $\Theta_2 = \left\{ (0, 1)%^{\top}
    \right\}$. We have already shown that %$\Theta_2 \times {\R_{\geq 0}} \subseteq \overline{\fR \setminus \left( \Theta_2 \times {\R_{\geq 0}} \right)}$ 
{\equ{closurecomplement} holds} in case (iv), and one can also check \eqref{bigtheta1} in the same way it was done in (i). It remains to check \eqref{ch_theorem_condition} and apply Theorem \ref{ch_theorem}. 
    %The rest of this proof follows the same lines as before. \comm{No! this is the only case when both sets are nonempty, it has to be explained in a bit more detail. Can you do it?}

    \item[(vi)] Equidistribution modulo $\pi$ %means 
    {implies} that every cone centered at zero contains infinitely many points of $Q$. By Observation \ref{replacement of compatibility} and Lemma \ref{compatibility_from_logdense} we conclude that ${\R_{\geq 0}\times\mathbb{S}^1}     \subseteq \overline{\fR}$. Since $\overline{\mathbb{S}^1 \smallsetminus \{ \theta \}} = \mathbb{S}^1$ for any $\theta \in \mathbb{S}^1$,   Theorem \ref{ch_theorem} %$\mathcal{L}_{P,Q}$ %contains a totally dense subcollection.
    can be applied with   $\Theta_1 = \mathbb{S}^1$ and  $\Theta_2 = \varnothing$. 
\end{enumerate}
\end{example}

\medskip

{Now let us turn to some more applications of Theorem \ref{any_k_theorem}. The following list of examples shows how to generate 
a number of situations similar to or more general than that of Theorem \ref{8.1}(a).} %\comm{Your Proposition \ref{examples_for_anyk} is nice but way too long. I am leaving it for now, but  prefer two propositions that I carved from it.}

\begin{proposition}\label{prop_logdenseP}
    For $m,n\in\N$ with $ n> 1$ and nonempty subsets $P\subset \R^m$ and $Q\subset \R^n\nz$,  %the collection $\mathcal{L}_{P,Q}$ contains a totally dense subcollection 
     {$(P,Q)$ has property {\rm (TDS)}} in all of the   cases below:
    \begin{enumerate}
        \item[\rm(a)]\label{parta_Ex2.6} Fix $0 \leq k \leq \min(n-2, m)$, let $F_1, F_2$ be two different $(k+1)$-dimensional subspaces of $\R^n$, and let ${H}$ be an $(m-k)$-dimensional subspace of $\R^m$ such that the set of limit directions of $Q$ contains  $(F_1\cup F_2)\cap \mathbb{S}^{n-1}$, and $ P$ is logarithmically dense in ${H}$.
        %\cup \ldots \cup \mathbb{N} \bv_5$.
        
         \item[\rm(b)]\label{partb_Ex2.6} Fix $0 \leq k \leq \min(n-1, m-1)$, let $F$ be a $(k+1)$-dimensional subspace of $\R^n$, and let ${H}_1,{H}_2$ be two different $(m-k)$-dimensional subspaces of $\R^m$ such that the set of limit directions of $Q$ contains  $F\cap\mathbb{S}^{n-1}$, and $ P$ is logarithmically dense in both ${H}_1$ and ${H}_2$.
        \item[\rm(c)] \label{partc_Ex2.6}
        %Let 
        Fix $0 \leq k \leq \min(n-1, m)$, let $F$ be a $(k+1)$-dimensional subspace of $\R^n$, and let ${H}$ be an $(m-k)$-dimensional linear subspace of $\R^m$ such that the set of limit directions of $Q$ contains  $F\cap\mathbb{S}^{n-1}$,  $P$ is logarithmically dense in ${H}$, and 
\begin{equation}\label{additional}%\tag{2.5}
\text{either }P\cap {H} = \varnothing\ \ \text{  or }\ \ Q \cap F = \varnothing.
        \end{equation}
        %\comm{Did I get this part right? -- That was weaker than we want (in particular it wouldn't imply some parts of Example 2.8.), I changed it.}
    \end{enumerate}
%In addition, in each of the cases above we can replace either $F_1,F_2,F$ or $H_1,H_2,H$ by their  half-spaces (but not both at the same time).
\end{proposition}
%\comm{A question: can "$= \varnothing$" in \eqref{additional} be replaced by "bounded"? -- Yes, since we can generally just ignore bounded sets. They don't contribute to neither log density nor limit directions, so we can always remove any bounded subsets from $P$ and $Q$. DK: maybe I should have said "finite", in which case of course the answer is yes. Ignoring bounded can be a problem: look at $P = \Q,\,Q = \mathbb{S}^{1}$ or $P = \{0,1\},\,Q = \Q^{2}\nz$.  -- In general that's true, but if we already assume that the other conditions in (c) hold, we can ignore bounded sets. For the case when bounded sets probide total density: should we actually have an example of such situation in the text?} 

%\comm{To be honest, I don't like the appearance of half-spaces here. I understand that it makes the proposition stronger. But do we really need it? I think we had enough asymmetric examples applying the previous theorem. Maybe we can keep everything symmetric here, and at the end of the section make a remark that the proofs can be slightly modified with half-spaces to produce a stronger statement? I would stick with only symmetric examples when illustrating the lasto two propositions of the section.}

\begin{example}\label{ex_logdenseP}  Take any $0 \leq k \leq \min(n-2, m)$,   let $P$ be logarithmically dense in {some} $(m-k)$-dimensional subspace of $\R^m$ {and  let} 
\begin{equation}\label{exampleQ}%\tag{2.6}
Q  = \big(Q_1\times 
\{0^{n-k-1}\} \big) \cup \big(\{0^{n-k-1}\}\times Q_2\big)\nz,\end{equation}
where {$Q_1,Q_2\subset \R^{k+1}$} 
%are logarithmically dense in $\R^{k+1}$, or, more generally, 
have every direction in ${\mathbb S}^{k}$ as their limit direction. {Then the assumptions of  Proposition~\ref{prop_logdenseP}(a) will be satisfied}. This
 can be thought of as a %far-reaching  
 higher-dimensional generalization of our motivational example~\equ{niceexample}. {Or {else} one can replace \eqref{exampleQ} by 
 $$Q=Q_1\times 
\{0^{n-k-1}\}+\bv,$$ where $\bv$ does not belong to   $\Span_\R(Q_1\times 
\{0^{n-k-1}\})$, and keep the same $P$; then 
Proposition~\ref{prop_logdenseP}(c) will apply. Another option is to simply take  $Q=Q_1\times 
\{0^{n-k-1}\}$ and, as a compensation, require the logarithmic density of $P$   in {two different} $(m-k)$-dimensional subspaces; {this can be handled} by part (b) of the proposition.}
\end{example}
 %\comm{Maybe here we can also come up with the simplest examples of cases (b) and (c)? -- We can, but they will be essentially the same, and I feel like the text becomes too heavy. Shold we? DK: I tried to write what I had in mind, hope it is not too heavy. But some examples below are rather heavy, perhaps we should make them lighter.}
%\medskip

 {Our next} observation %in this section 
 is that {one can}, roughly speaking, interchange the roles of $P$ and $Q$ in Proposition \ref{prop_logdenseP}. This can be achieved due to a built-in symmetry of Observation \ref{replacement of compatibility}. 
Here is  what this symmetry produces: 

\begin{proposition}\label{prop_logdenseQ}
    For $m,n\in\N$ with $ n> 1$ and nonempty subsets $P\subset \R^m$ and $Q\subset \R^n\nz$,  {$(P,Q)$ has property {\rm (TDS)}}
    %the collection $\mathcal{L}_{P,Q}$ contains a totally dense subcollection 
    in all of the   cases below:
    \begin{enumerate}
        \item[\rm(a)]\label{parta_Ex2.7} Fix $0 \leq k \leq \min(n-2, m)$, let $F_1, F_2$ be two different $(k + 1)$-dimensional
linear subspaces of $\R^n$, and let ${H}$ be an $(m-k)$-dimensional subspace of $\R^m$ such that $Q$ is logarithmically dense in both $F_1$ and $F_2$, and 
the set of limit directions of $P$ contains  $H\cap\mathbb{S}^{m -1}$.
        %\cup \ldots \cup \mathbb{N} \bv_5$.
        
        \item[\rm(b)]\label{partb_Ex2.7} Fix $0 \leq k \leq \min(n-1, m-1)$, let $F$ be a $(k + 1)$-dimensional
linear subspace of $\R^n$, and let ${H}_1,{H}_2$ be two different $(m-k)$-dimensional subspaces of $\R^m$ such that $Q$ is logarithmically dense in $F$, and 
the set of limit directions of $P$ contains  $(H_1\cup H_2)\cap\mathbb{S}^{m -1}$.
        \item[\rm(c)] \label{partc_Ex2.7}
        %Let 
        Fix $0 \leq k \leq \min(n-1, m)$, let $F$ be a $(k + 1)$-dimensional
linear subspace of $\R^n$, and let ${H}$ be an $(m-k)$-dimensional linear subspace of $\R^m$ such that $Q$ is logarithmically dense in   $F$,   
the set of limit directions of $P$ contains  $H\cap\mathbb{S}^{m-1}$, and {\eqref{additional} holds}.
%{ one of the two conditions holds: either $P\cap span_{\R} \Psi = \varnothing$ or $Q \cap F = \varnothing$. } 
    \end{enumerate}
    %In addition, in each of the cases above we can replace either $F_1,F_2,F$ or $H_1,H_2,H$ by their  half-spaces (but not both at the same time).
%{  In addition, in each of the cases above we can either replace the $(k+1)$-dimensional subspaces $F$ by their half-spacess or replace the $(m-k-1)$-dimensional spheres $\Psi$ by their halfspheres (but not both at the same time).}
\end{proposition}

%\comm{Write a few sentences about examples symmetric to the discussion after the previous proposition.}

One can construct a series of examples {illustrating Proposition \ref{prop_logdenseQ} in a similar way to what was done in Example} \ref{ex_logdenseP}:  take $0 \leq k \leq \min(n-1, m-1)$, let $Q$ be logarithmically dense in some $(k+1)$-dimensional subspace of $\R^n$, and let 
$$
P  = \big(P_1\times 
\{0^{k}\} \big) \cup \big(\{0^{k}\}\times P_2\big)
%\nz, \comm{\quad\text{do we need }\nz?}
$$
where $P_1, P_2 \subset \R^{m-k}$ have every direction in $\mathbb{S}^{m-k-1}$ as their limit direction. This {satisfies the %examples to 
assumptions of} Proposition \ref{prop_logdenseQ}(b). {Applications of parts (a) and (c) of Proposition \ref{prop_logdenseQ} can  also be %cooked up 
{constructed} similarly to those in Example \ref{ex_logdenseP}}.
 %\comm{DK: I'll edit it later.}
\smallskip

{Let us now discuss the proof of the two propositions above. %It is rather straightforward to
{One can readily} reduce their three cases  to the corresponding cases of Theorem \ref{any_k_theorem}, with $\Phi_1,\Phi_2,\Phi$ being the intersections of the corresponding subspaces with $\mathbb{S}^{n-1}$. This way} conditions \eqref{anykA}, \eqref{anykB} or \eqref{anykC} hold in the cases (a), (b) and (c) respectively. As an illustration, we will show the deduction of part (c); the first two parts can be obtained in a similar way.

%{
\begin{proof}[Proof of Propositions \ref{prop_logdenseP}(c) and %Proposition 
\ref{prop_logdenseQ}(c)] 
First assume that $k \neq m$. Fix {$\theta \in \Phi:= F \cap \mathbb{S}^{m-1}$}, $\phi \in {H} \cap \mathbb{S}^{m-1}$ and $\varepsilon > 0$. We will prove that \eqref{Obs_24_set} holds under assumptions of either of {the} propositions:
    
    \begin{itemize}
        \item %{\bf In assumption of Proposition \ref{prop_logdenseP}:} 
        {Assuming that the set of limit directions of $Q$ contains  $F\cap\mathbb{S}^{n-1}$ and  $P$ is logarithmically dense in ${H}$:}
        since $|P|_{\varepsilon}^{\phi}$ is logarithmically dense in $\R_{\geq 0}$ and $Q_{\varepsilon}^{\theta}$ is unbounded, \eqref{Obs_24_set} %in $\R_{\geq 0}$  
    follows  from Lemma \ref{compatibility_from_logdense}.
    \item %{\bf In assumption of Proposition \ref{prop_logdenseQ}:} 
    {Assuming that $Q$ is logarithmically dense in   $F$ and   
the set of limit directions of $P$ contains  $H\cap\mathbb{S}^{m-k-1}$:} since $|Q|_{\varepsilon}^{\theta}$ is logarithmically dense in $\R_{\geq 0}$ and $P_{\varepsilon}^{\phi}$ is unbounded, we can apply Lemma \ref{compatibility_from_logdense} with $|Q|_{\varepsilon}^{\theta}$ in place of $P$ and $P_{\varepsilon}^{\phi}$ in place of $Q$ and deduce that
    $$
    {\text{the set }\left\{ \frac{y}{x}: \,\,\, x \in |P|_{\varepsilon}^{\phi}, \, y \in |Q|_{\varepsilon}^{\theta} \right\}\text{is dense in }\mathbb{R}_{\geq 0},}
    $$
    which is equivalent to \eqref{Obs_24_set}.
    \end{itemize}
    %{Recall that we have assumed that either $F$ or $H$, but not both, can be   half-spaces.} 
    Let $\fR: = \pr(P \times Q) $. By Observation \ref{replacement of compatibility}, ${H} \times \Phi \subseteq \overline{\fR}$. %Therefore, 
    %$$
    %\left( {H} \cup -{H} \right) \times \left( \Phi \cup -\Phi\right) \subseteq \overline{\fR},
    %{H} \times \Phi \subseteq \overline{\fR},
    %$$
    %where ${H} \cup -{H}$ and $\Phi \cup -\Phi$ are {now} a full subspace and a full sphere. \comm{Can you explain what exactly goes wrong when both $F$ and $H$ are half-spaces?}
    {On the other hand, \eqref{additional} implies that 
    %Since 
    $
    %\fR \cap \left( {H} \cup -{H} \right) \times \left( \Phi \cup -\Phi\right) = \varnothing,
    \fR \cap \left( {H}\times  \Phi \right) = \varnothing
    $, hence}
    \eqref{anykC} holds, and the conclusion follows from Theorem \ref{any_k_theorem}. %Thus, $\mathcal{L}_\fR = \mathcal{L}_{P, Q}$ contains a totally dense subcollection.
\smallskip

    It remains to show that the case $k = m$, i.e, ${H} = \{0\}$, also works. Then logarithmic density in ${H}$ is a vacuous condition. Since $P \neq \varnothing$, there exists an element $\p_0 \in P$. 
    %and $P\,\cap\, {H} = \varnothing$ implies that there exists   $\p_0 \in P\nz$. 
    %The set
    {Under the assumptions of either of the propositions} $\Phi$ is contained in the set of limit directions of $Q$, and thus of $Q \smallsetminus B_C(0)$ for any $C > 0$. Let $\fR = \pr(\{ \p_0 \} \times Q )$. 
    If $P \neq \{ 0 \}$, then we can additionally assume that we chose $\p_0 \neq 0$. 
   {This implies} 
\begin{equation}\label{0inclosure}%\tag{2.6}
\{ 0 \} \times\Phi \subseteq \overline{\fR \smallsetminus ( \{ 0 \} \times\Phi) },\end{equation} and it remains to apply Theorem \ref{any_k_theorem} again.
    If $P = \{ 0 \}$, then the condition $P \cap \Pi = \varnothing$ does not hold, thus $Q \cap F = \varnothing$ {has to hold, which, as above, implies \eqref{0inclosure}}. 
    %$ \{ 0 \} \times\Phi \subseteq \overline{\fR \smallsetminus ( \{ 0 \} \times\Phi) \} }$.
\end{proof}

%\comm{OK, maybe we don't need  detailed proofs here, but at least we need to explain the   idea to swap $P$ with $Q$ using some example! -- I did something like this in the next one}

%\comm{Make a remark about half-spaces, then have some examples. Or vice versa.}

\begin{remark}\label{half-spaces} \rm Our last observation in this section is that we do not always have to consider the whole subspaces in Proposition \ref{prop_logdenseP} and \ref{prop_logdenseQ} for {their conclusions} to hold. It turns out that 
%this argument is quite general: in 
each of the cases of {these propositions one can replace either $F_1,F_2,F$ or $H_1,H_2,H$ (but not both at the same time) by their %\comm{(closed?)}
half-spaces}.

 To explain why this works, 
%\comm{(maybe instead of the full proof just explain the main idea?)}
%we need to 
let us recall that in the proof of Propositions \ref{prop_logdenseP} and \ref{prop_logdenseQ} we used $\fR = \pr(P \times Q)$; {however  one} can always use a larger and symmetric set $\pr(P \times Q) \cup - \pr(P \times Q)$. As an illustration, we show how this idea can be implemented in the proof of part (c) above in case $F$ is a half-space.
%First of all, we can always remove some points from $P$ and assume $F \cap -P = \varnothing$; this does not affect neither logarithmic density in $F$ nor the condition that the set of limit direction contains $F \cap \mathbb{S}^{m-1}$.
%Let us denote the subspace $F \cup -F$ by $F'$. 
{One can  follow the proof of Proposition \ref{prop_logdenseP} and \ref{prop_logdenseQ}   and, removing some points from $P$ if necessary,    show that $${H} \times F \subseteq \overline{\pr(P \times Q) \smallsetminus \left( {H} \times (F \cup -F) \right)}.$$ Then, using the fact that ${H} = -{H}$, one can similarly show that $${H} \times -F \subseteq \overline{-\pr(P \times Q) \smallsetminus \left( {H} \times (F \cup -F) \right)}.$$ It remains to combine the two statements above} and
%$$
%{H} \times F' \subseteq \overline{\fR \setminus \left(  {H} \times F' \right)}
%$$
%and can finish the proof by 
apply Theorem \ref{any_k_theorem}.
%\comm{DK: I'll edit it later.}
{We note} that this argument does not work if  both $F$ and ${H}$ are replaced by half-spaces simultaneously, since
$$
({H} \times F) \cup -({H} \times F) \neq ({H} \cup -{H}) \times (F \cup -F).
$$
\end{remark}

%\comm{Thanks for these examples. I won't read them today, maybe we can edit them together on zoom.}

{\begin{example}\label{thelastexample}
      {Suppose that $n = m+1$,
    %$n = 3, m = 2$ 
    %{$m,n\in\N$ with $n\ge \max(m,2)$}. %and $k = 1$.  L
    take 
        $$
        Q = \left\{ \big(q_1, \dots,q_{m}, e^{-(q_1^2 + \cdots  +q_{m}^2)}%,0,\dots,0
        \big) : \,\,\, q_1, \dots,q_{m} \in \Z\right\},
        $$
        and let $P \subseteq \R^m$ be any unbounded set.
    We claim that %$\mathcal{L}_{P,Q}$ contains a totally dense subcollection.
    $(P,Q)$ has property (TDS). 
    Indeed: in this case %again 
     we can take $k=m-1$. Clearly $Q$ is logarithmically dense in the $(k+1)$-dimensional subspace $F:=\R^{m}\oplus\{0\}$ of $\R^n$, and at the same time  $Q \cap F%\left( \R^2\oplus\{0\}\right) 
     = \varnothing$.
    On the other hand} the set of limit directions of any unbounded set is nonempty, {hence it contains $H\cap\mathbb{S}^{m-1}$ for some  ray %(one-dimensional half-space) 
    $H$ in $\R^m$. %at least one zero-dimensional halfsphere of ${\mathbb S}^{m-1}$.
    Thus  the conditions of Proposition \ref{prop_logdenseQ}(c) will be  satisfied if ${H}$ is taken to be a one-dimensional half-space.} 
    %\comm{DK: I'll edit it later.}
\end{example}
}

\section{Two sufficient conditions for density}\label{just_density_section}

%It will be convenient for us to say {\it $\mathcal{L}$ is dense in $M_{m,n}(\mathbb{R})$ } meaning that the set $\bigcup\limits_{\bfL \in \mathcal{L}} \bfL$ is dense. 
In this section we %formulate and prove 
{describe} two conditions {(in Propositions \ref{density_any_k} and \ref{ditect product dense}, {to be used in the proof of Theorems \ref{any_k_theorem} and \ref{ch_theorem}} respectively)} guaranteeing that {for a given subset $\fR$ of $ {\bf T}_{m,n}$, the} collection $\mathcal{L}_{\fR}$ is dense {in $\mr$} (but not necessarily totally dense).
%{\color{brown} Condition shown in Proposition \ref{density_any_k} will be important in the proof of Theorem \ref{any_k_theorem}, and condition shown in Proposition \ref{ditect product dense} plays a key role in the proof of Theorem \ref{ch_theorem}.}
%One can also easily see that if $\p = (p_1, \ldots, p_m)$, then
%\begin{equation}\label{direct product}
%\bfL_{\p,\q} = \bfL_{\q, p_1} \times \ldots \times \bfL_{\q, p_m}.
%\end{equation}
{First, let} us note the following simple but useful property:

\begin{lemma}\label{closure}
For any $\fR \subseteq {\bf T}_{m,n}$,
    $$
    \overline{\bigcup\limits_{\br \in \fR} \bfL_{\br}} = \overline{\bigcup\limits_{\br \in \overline{\fR}} \bfL_{\br}},
    $$
    where $\overline{\fR}$ denotes the closure of $\fR$. In particular, $\mathcal{L}_{\fR}$ is dense if and only if $\mathcal{L}_{\overline{\fR}}$ is dense.
\end{lemma}
\begin{proof}
Suppose $\br \in \overline{\fR}$ and $x \in  \bfL_{\br}$. If $\{ \br_i\}_{i=1}^\infty\subseteq \fR$ converges to $\br\in \overline{\fR}$, then $\bfL_{\br}\subseteq \overline{\bigcup_{i=1}^\infty \bfL_{\br_i}}$, therefore $x \in \overline{\bigcup\limits_{\br \in \fR} \bfL_{\br}}$, which implies the inclusion $\overline{\bigcup\limits_{\br \in \overline{\fR}} \bfL_{\br}} \subseteq \overline{\bigcup\limits_{\br \in \fR} \bfL_{\br}}$. The opposite inclusion is trivial.
%The "only if" case is trivial by inclusion. Suppose now $\mathcal{L}_{\overline{\fR}}$ is totally dense; fix $\br_0 = (\q_0, \p_0) \in \mathbb{R}^n \times \mathbb{R}^m$, open neighborhood $W \subseteq M_{m,n}$ ${\bfL} \cap W \neq \varnothing$ and let $\Omega$ be as in Definition \ref{loc_dens_definition}.
%    Define $\fR': = \{ \br \in \fR: \,\,\, \bfL_{\br} \cap {\bfL} \cap \Omega \neq \varnothing  \}$; it is easy to check that 
%    $$
%    \{ \br \in \overline{\fR}: \,\,\, \bfL_{\br} \cap {\bfL} \cap \Omega \neq \varnothing  \} \subseteq \overline{\fR'}.
%    $$
%    Thus, by part \ref{closure1}, 
%    $$
%    \overline{\bigcup\limits_{(\q, \p) \in {\fR}: \,\,\bfL_{\p,\q} \cap {\bfL} \cap \Omega \neq \varnothing} \bfL_{\p,\q}} \supseteq \overline{\bigcup\limits_{\br \in \fR'} \bfL_{\br}} = \overline{\bigcup\limits_{\br \in \overline{\fR}} \bfL_{\br}} \supseteq \overline{\bigcup\limits_{(\q, \p) \in \overline{\fR}: \,\,\bfL_{\p,\q} \cap {\bfL} \cap \Omega \neq \varnothing} \bfL_{\p,\q}} \supseteq \Omega,
%    $$
%    which completes the proof.
\end{proof}

  We will also use the following {elementary} lemma:

\begin{lemma}\label{generic dimension}
    Let ${F}$ be an $s$-dimensional linear subspace of $\mathbb{R}^n$ and ${H}$   an $l$-dimensional linear subspace of $\mathbb{R}^m$. Then the set 
\begin{equation}\label{alg_set}
    \left\{A \in M_{m,n}: \dim \left( A{F} + {H} \right) \neq \min(s+l, m)\right\}
    \end{equation}
    is a proper algebraic subset of $M_{m,n}$;
    {hence its complement}
    %. In particular, %almost every $A_0 \in M_{m,n}$ has an open neighborhood $W$ such that 
   % the set
    $$
    \{ A \in M_{m,n}: \dim \left( A{F} + {H} \right) = \min(s+l, m)\} 
    $$
    is open and dense in $M_{m,n}$.
    %\comm{Maybe better to say that the set of such $A_0$ is open and dense? you don't really use Lebesgue measure, right?}
\end{lemma}

%\comm{Vasya: I wrote up a detailed proof, but this is an obvious linear algebra argument - so maybe we should shorten it or completely remove and just say that this lemma is an obvious observation?}

\begin{proof}
    Let $\q^1, \ldots, \q^s$ be a basis of ${F}$ and $\p^1, \ldots, \p^l$ %be 
    a basis of ${H}$. Then $A$ belongs to the set \eqref{alg_set} if and only if the rank of %joint
{the augmented} matrix
    \begin{equation}\label{matrix_combined}
    \begin{pmatrix}
        A \q^1 \,| \ldots \,|\, A\q^s \,|\, \p^1 \,|\, \ldots \,|\, \p^l
    \end{pmatrix}
    \end{equation}
    is smaller than $r: = \min(s+l, m)$. This is equivalent to %all the 
    {vanishing of all its $r \times r$ minors},
    %submatrices being degenerate,
     which, in turn, is an algebraic condition on {the} coefficients of $A$. {Hence} the set \eqref{alg_set} is algebraic, and it remains to prove that it does not coincide with $M_{m,n}$. 
    Let ${\bf w}^1, \ldots, { \bf w}^{r-l}$ be a collection of $m$-vectors completing $\p^1, \ldots, \p^l$ to a linearly independent set. Since $r-l \leq s$, there exists a matrix $A \in M_{m,n}$ which maps $\q^1, \ldots, \q^{r-l}$ to ${\bf w}^1, \ldots, { \bf w}^{r-l}$. %In this case 
    {For such $A$} the matrix $\eqref{matrix_combined}$ has rank $r-l$, {which implies that} 
    %matrix 
    $A$ does not belong to the set \eqref{alg_set}. %and 
    {Thus} %\eqref{alg_set} 
    {the latter} is a proper algebraic set.
\end{proof}

{Finally, the following fact will be useful:
\begin{lemma}\label{opensetlemma}
    Fix $0 \leq k \leq \min(n-1, m)$. Let $\Phi$ be a $k$-dimensional subsphere of $\mathbb{S}^{n-1}$ and ${H} \subseteq \mathbb{R}^m$ an $(m-k)$-dimensional linear subspace. Let $U$ be an open subset of ${H} \times \Phi$. Then the set 
    $$
{\bigcup\limits_{
%(\p, \theta)
\br\in U} \bfL_{
%\p, \theta
\br}}
    $$
    has a nonempty interior in $M_{m,n}$.
\end{lemma}
Since the proof of Lemma \ref{opensetlemma} is rather long and technical, we will provide it separately in \S\ref{technicalsection}.}
%{\bf Proof of Proposition \ref{not_any_k}.}
We are now ready to formulate and prove the main results of this section.

\begin{proposition}\label{density_any_k}
     Fix $0 \leq k \leq \min(n-1, m)$. Let $\Phi$ be a $k$-dimensional subsphere of $\mathbb{S}^{n-1}$ and ${H} \subseteq \mathbb{R}^m$ an $(m-k)$-dimensional linear subspace. Then  the set%for almost every $A_0 \in M_{m,n}$ there exists an open set $W \in M_{m,n}$ containing $A_0$ such that 
     \begin{equation}\label{density_in_W}
         W = \big\{A \in M_{m,n}: \,\, \text{there exists a unique up to sign pair} \, (\p, \theta) \in {H} \times \Phi \,\, \text{such that} \,\, A \theta = \p\big\} 
     \end{equation}
     is open and dense in $M_{m,n}$.
In particular, the collection $\mathcal{L}_{{H} \times \Phi}$ is dense.

\noindent {Furthermore,}
%However: 
if ${D}$ is a proper closed symmetric subset of ${H} \times \Phi$, then the collection $\mathcal{L}_{{D}}$ is not dense.
\end{proposition}

%\comm{Need to write up the proof of the last part.}

\begin{proof}
Let $\q^1, \ldots, \q^{k+1}$ be a basis of $F = \Span_{\mathbb{R}} \Phi$, and let  $\p^1, \ldots, \p^{m-k}$ be
a basis of ${H}$. Let $\ba^i : = A \q^i \in \mathbb{R}^m$. 
    %The condition 
    {To prove} the proposition, %\eqref{density_in_W} %is equivalent to the existence of 
    it suffices to show the existence and uniqueness up to proportionality of a tuple 
\begin{equation}\label{nonzerotuple}
    0 \neq (\lambda_1, \ldots, \lambda_{k+1}, \mu_1, \ldots, \mu_{m-k}) \in \mathbb{R}^{m+1}
    \end{equation} %for which
    {such that}
\begin{equation}\label{lin_dependence}
    \lambda_1 \ba^1 + \ldots + \lambda_{k+1} \ba^{k+1} = \mu_1 \p^1 + \mu_{m-k} \p^{m-k}
    \end{equation}
for any $A$ in some open and dense set.
%\comm{This is formally not correct, \eqref{density_in_W} is about $W$ which you first need to exhibit.}
    Note that, by the linear independence of %the collection 
    $\{ \p^1, \ldots, \p^{m-k} \}$, condition \eqref{nonzerotuple} is equivalent to $0 \neq (\lambda_1, \ldots, \lambda_{k+1})$.
We consider two cases.
\medskip

\noindent{\bf Case 1.} Suppose $k + 1 \leq m$. By Lemma \ref{generic dimension}, %for every $A_0 \in W$ there exists an open set $W \in M_{m,n}$ containing $A_0$ such that
the set $W$ of $A \in M_{m,n}$ such that
$$
\dim(A F) = k+1 \,\,\,\,\,\,\,\, \text{and} \,\,\,\,\,\,\,\, \dim(AF + {H}) = m \,\, \Longrightarrow \,\, \dim(AF \cap {H}) = 1
$$
is open and dense. Since $\dim(A F) = k+1$, vectors $\{ \ba^i \}$ are linearly independent, and thus \eqref{lin_dependence} is a nonzero vector from a one-dimensional vector space $AF \cap {H}$, proving the desired existence and uniqueness up to scaling of the vector \eqref{nonzerotuple}.

\medskip
\noindent{\bf Case 2.} Suppose $k = m$, and thus $\dim H = m-k = 0$. Applying Lemma \ref{generic dimension} to subspaces $F_i: = \Span(\q^1, \ldots, \q^{i-1}, \q^{i+1}, \ldots, \q^{k+1})$, we deduce: 
%for almost every $A_0 \in M_{m,n}$ there exists an open set $W \in M_{m,n}$ containing $A_0$ such that 
there exists an open and dense set $W \subseteq M_{m,n}$  such that for any  $A \in W$ every $k$-element subcollection of $\{\ba^1, \ldots, \ba^{k+1}\}$ is linearly independent. It follows that there exists a unique up to proportionality %collection 
$(k+1)$-tuple $(\lambda_1, \ldots, \lambda_{k+1})$ such that 
$
\lambda_1 \ba^1 + \ldots + \lambda_{k+1} \ba^{k+1} = 0.
$

\medskip

{{For the proof of the "furthermore" part} let us fix a proper closed symmetric subset ${D}$ of $H \times \Phi$, and let $U \subseteq H \times \Phi$ be an open set such that $\left( U \cup -U \right) \times {D} = \varnothing.$ By Lemma \ref{opensetlemma}, the set
$$
V: = W \cap \Big( \bigcup\limits_{(\p, \theta) \in U} \bfL_{\p, \theta} \Big)
$$
has nonempty interior in $M_{m,n}$. %It remains to show 
{We claim} that $V \cap \Big( \bigcup\limits_{(\p, \theta) \in {D}} \bfL_{\p, \theta} \Big) = \varnothing$.
Indeed: let $A \in V$. Since $A \in W$, there exists a unique up to sign pair $(\p, \theta) \in H \times \Phi$ such that $A \in \mathcal{L}_{\p, \theta}$. Since $A \in \bigcup\limits_{(\p, \theta) \in U} \bfL_{\p, \theta} $, we know that $(\p, \theta) \in U \cap -U$. This observation completes the proof.}
\end{proof}

%\begin{remark}\label{continuous_on_halfsphere}
    %\begin{itemize}
        %\item 
 %       \rm
%        Let $k\ge 1$ and let $\Phi'$ be a $(k-1)$-dimensional subsphere of $\Phi$. Applying Lemma \ref{generic dimension} one more time,
        %to the same argument, 
%        {one can see} that the set 
%        $$
%        W' = \{ A \in W: A {\theta} \notin {H} \,\, \text{for ${\theta} \in \Phi'$} \}
%        $$
%    is still an open and dense set in $M_{m,n}$ (as an intersection of two open dense sets).
%    Furthermore, given  $A \in W'$ one can fix an open halfsphere $\Phi_1$ of $\Phi$ with boundary $\Phi'$; then the functions
%    $$
 %   {\theta}: \,\, W' \rightarrow \Phi_1 \,\,\,\, \text{and} \,\,\,\, {\bx}: \,\,W' \rightarrow {H}: \,\,\,\, A {\theta}(A) = {\bx}(A)
 %   $$
 %   are well defined. Since the vectors $\ba^i$ in \eqref{lin_dependence} depend continuously on $A$, the functions ${\theta}$ and ${\bx}$ are continuous on $W'$.
    %\end{itemize}
%\end{remark}

%For $x \in M_{m,n}(\mathbb{R})$, we define $R_x: = \{ (\q, \p) \in {\bf T}_{m,n}:  \,\,\, x \q = \p \}$. Let us note: 

%\begin{itemize}
%    \item The set $R_x$ is symmetric: $(\q, \p) \in R_x$ if and only if $(-\q, -\p) \in R_x$; 
%    \item $R_x$ is homeomorphic to ${\mathbb S}^{n-1}$. 
%\end{itemize}

%We begin with a convenient sufficient condition for density.

%\vskip+0.3cm

%The second 
{The next} statement provides a more specific condition   guaranteeing density; {it will be used in the proof of} Theorem \ref{ch_theorem}.
%in case when $\p$ is chosen from a set with maximal possible dimension. \comm{Sorry, I don't understand this description. More specific than what? and there is no $\p$ in the statement of the proposition which makes it confusing.}

%\comm{Merge with the above proposition}

%We will then show that Proposition \ref{ditect product dense} cannot be improved. {Namely,  if $\fQ \subseteq {\mathbb S}^{n-1}$ is such %a set that 
%    that %the convex hull of 
%{$0\notin\conv\left(\overline{\fQ}\right)$} and
%    as a subset of $\mathbb{R}^n$ contains zero
% ${H}$  an open half-space of ${\mathbb{R}^m}$, then one can find $\fR\in\tmn$ containing
%with respect to some linear hyperplane,
%and let $K = B_C(0) \cap \Pi$. %\comm{What is $C$? what is $B_C$, a ball with radius $C$? in which norm?}
    %\begin{enumerate}
       % \item[\rm(a)] \label{general_direct_dense} Suppose $\fR \subseteq {\bf T}_{m,n}$ is such %a set 
       % that $\fQ \times K \subseteq \overline{\fR}$. Then the collection $\mathcal{L}_{\fR}$ is dense in $B_C(0)$.
        %\subseteq M_{m,n}$. 
        %\item[\rm(b)] \label{part_direct_dense} 
%       ${H} \times \fQ$ with $\mathcal{L}_{\fR}$ not dense in $M_{m,n}$. More precisely:}
%\smallskip

\begin{proposition}\label{ditect product dense}
    Let ${H}$ be a half-space of $\mathbb{R}^m$ and {let} $\Theta \subseteq \mathbb{S}^{n-1}$. The collection $\mathcal{L}_{{H} \times \fQ}$ is dense if and only if
    $
    0 \in \conv(\overline{\Theta}).
    $
\end{proposition}

\begin{proof}
{In view of Lemma \ref{closure}, we can assume that $H$ is a closed half-space and prove the statement for the %bigger 
{larger} collection $\mathcal{L}_{{H} \times \overline{\fQ}}$.
%= \overline{\fR}$ 
%is  closed. in this case, the %condition 
%{inclusion} $\overline{H} \times \overline{{\fQ}} \subseteq {\fR}$ holds.
    
    First suppose $0 \in \conv(\overline{\Theta})$.} We fix $A \in M_{m,n}$ and show that there {exist} $\theta \in \fQ$ and $ \p \in \overline{{H}}$ such that $A \in \bfL_{\p, \theta}$.
   By the convex hull condition, there exist $\theta^{1}, \ldots, \theta^{k} \in \overline{\fQ}$ and positive real scalars $\lambda_1, \ldots, \lambda_k$ such that 
\begin{equation}\label{convex}
    \sum\limits_{i=1}^k \lambda_i \theta^{i} = 0 \,\,\,\, \Longrightarrow  \,\,\,\, \sum\limits_{i=1}^k \lambda_i \cdot A \theta^{i} = 0.
    \end{equation}
    Suppose that {$\p^{i} := A \theta^{i} \notin   {{H}} $ for all} $i = 1, \ldots, k$. Since {the complement of  ${{H}}$ is an} open half-space, the %whole 
    {cone generated by $\p^{1},\dots,\p^{k}$} is {contained in} $\left( {{H}} \right)^c$, contradicting \eqref{convex}. Thus there exists $i$
    such that $A \theta^{i} = \p\in {{H}}$, and {hence} $A \in \bigcup\limits_{{\bfL} \in \mathcal{L}_{H \times \overline{\Theta}}} \bfL$, which {implies} the desired statement.

    Now suppose $0 \notin \conv(\overline{\Theta})$. By the Separating Hyperplane Theorem (see \cite[Chapter 2.5.1]{BV04} for reference) the %re exists a nonempty 
    open 
    subset { $$U := \{ \alpha \in \mathbb{R}^{n}: \,\,\, \inf\limits_{\theta \in \Theta} \alpha^{\top} \theta < 0 \}$$ of $\R^n$ is nonempty}.
    Suppose $\bv$ is a unit vector which defines ${H}$ by 
    $$
    {H} = \{ {\bf x} \in \mathbb{R}^m: \,\,\, \bv^{\top} { \bf x} \geq 0 \}.
    $$
    Let $W \subseteq M_{m,n}$ be an open set defined by $W = \left\{ A \in M_{m,n}: \,\,\,  \left( \bv^{\top} A \right)^{\top} \in U \right\}$. %It is easy to show 
    %\comm{Sorry, why is it open? In fact I don't think it is...} 
    {We claim} that $W \cap \bigcup\limits_{(\p, \theta) \in {H} \times \fQ} \bfL_{\p, \theta} = \varnothing$. Indeed: for any {$A\in W$}, $\theta \in \Theta$ and ${\p} \in {H}$,
    $$
    \begin{aligned}
    \dist(A, \bfL_{\p, \theta}) &= \|A \theta - \p\| = \| \bv^{\top}\| \cdot\| A \theta - \p\| \geq  \|\bv^{\top} (A \theta - \p)\| \\ &= \|\alpha^{\top} \theta - \bv^{\top} \p \|  \stackunder[1pt]{{}\geq{}}{\scriptstyle  \text{since} \  \alpha^{\top} \theta < 0, \, \bv^{\top} \p \geq 0}  \|\alpha^{\top} \theta\| \geq C(A) > 0,
    \end{aligned}
    $$
    where $\alpha = \left( \bv^{\top} A \right)^{\top} \in U$ and $C(A) = -\inf\limits_{\theta \in \Theta} \alpha^{\top} \theta > 0$. Since $C(A)$ is independent %on 
    {of} {the} choice of $\theta \in \Theta$ and $\p \in {H}$, this completes the proof.
\end{proof}

\begin{corollary}\label{sections_density}
    If ${\theta} \in {\mathbb S}^{n-1}$ is such that $\mathbb{R}^m \times \{ {\theta} \} \subseteq \overline{\fR}$, then the collection $\mathcal{L}_{\fR}$ is dense in $M_{m,n}$.
\end{corollary}

\begin{proof}
    {Consider $\fQ = \{ {\theta}, -{\theta}  \}$; then %$\overline{\fR} \cup - \overline{\fR} \supseteq \fQ \times \mathbb{R}^m$ ${H}$ be an arbitrary open half-space of $\mathbb{R}^m$. Then 
    we have  
    $$\overline{\fR} \cup (- \overline{\fR}) \supseteq  \mathbb{R}^m\times\fQ \supseteq {H}\times \fQ $$
    for any half-space ${H}$ of $\mathbb{R}^m$. Since the convex hull of $\fQ$ contains {the origin}, it follows from Proposition~\ref{ditect product dense} %(b) 
    that} $\mathcal{L}_{\overline{\fR} \cup (- \overline{\fR})}$ (and thus $\mathcal{L}_{\fR}$) is dense.
\end{proof}

\section{Some general conditions {implying} total density}\label{general_conditions_section}

In this section we provide some general statements describing %the 
useful sufficient conditions for %the 
{a} collection $\mathcal{L}_{\fR}$ to be totally dense.
We start with the following simple lemma:

\begin{lemma}\label{anglelemma}
    For any $\delta > 0$ there exists $C(\delta) > 0$ with the following property.
    Suppose ${\theta}_1, {\theta}_2 \in {\mathbb S}^{n-1}$ %satisfy the condition 
    {are such that} 
   % $dist({\theta}_1, \pm {\theta}_2) > \delta$. 
    \eq{delta}{\dist({\theta}_1, \pm {\theta}_2)  > \delta.} 
    Fix $\varepsilon > 0$, %and 
    let $\p_1, \p_2 \in \mathbb{R}^m$, and take ${Y}_1 \in \bfL_{\p_1, {\theta}_1}$ and ${Y}_2 \in \bfL_{\p_2, {\theta}_2}$ 
    %be such that 
    %\begin{itemize}
    %    \item ${Y}_i \in \bfL_{{\theta}_i, \p_i}$, $i = 1, 2$; and
   %     \item 
   with %\eq{epsilon}{
   $\|{Y}_1 - {Y}_2\| < \varepsilon$.
   % \end{itemize}
    Then there exists ${Y} \in \bfL_{\p_1, {\theta}_1} \cap \bfL_{\p_2, {\theta}_2}$ such that $$\|{Y} - {Y}_i\| < C(\delta) \varepsilon, \ i=1, 2.$$
\end{lemma}

%{We remark that the above lemma will be applied with $|\cdot|$ being the supremum norm on $\R^n$ and on $\mr$, and the distance  on ${\mathbb S}^{n-1}$  being induced by the Euclidean norm. However it is clear that the statement is independent of the choice of the norm, with the function $\delta\mapsto C(\delta)$ being norm-dependent. Thus in the course of the proof it will be safe to switch to the Euclidean norm}.

\begin{proof}%[Proof of Lemma \ref{anglelemma}]
%One can also easily see that if $\p = (p_1, \ldots, p_m)$, then
%\begin{equation}\label{direct product}
%\bfL_{\p,{\theta}} = \bfL_{p_1, {\theta}} \times \ldots \times \bfL_{p_m, {\theta}}.
%\end{equation}
    In view of \eqref{direct product} and the %fact that we are working in 
    {equivalence between the supremum norm and the Euclidean norm,} it is enough to prove the statement in the case $m = 1$, {that is, for rational affine hyperplanes in $\R^n$}.  %Now, suppose ${\theta}_i = (q_1^i, \ldots, q_n^i)$ 
    {That is, we are given ${\theta}_1, {\theta}_2 \in {\mathbb S}^{n-1}$ 
    satisfying
    \equ{delta},  $p_1, p_2 \in \mathbb{R}$, and two vectors
     ${\by}_i \in \bfL_{p_i, {\theta}_i}\subset \R^n$ 
    %be such that 
    %\begin{itemize}
    %    \item ${\by}_i \in \bfL_{{\theta}_i, \p_i}$, $i = 1, 2$; and
   %     \item 
  % satisfying \equ{epsilon}
  {with %\eq{epsilon}{
   $\|{\by}_1 - {\by}_2\| < \varepsilon$}.}

   % From now on in this proof we 
   % will work {with the} Euclidean norm $| \cdot |_2$. 
    Note that ${\theta}_i$ are unit vectors orthogonal to $\bfL_{p_i, {\theta}_i}$. Let ${H}$ be %a 2-dimensional 
    {the} affine plane in $\mathbb{R}^n$ % = M_{1,n}$ 
    spanned by ${\theta}_1, {\theta}_2$ and containing ${\by}_1$; let $\ell_i = \bfL_{p_i, {\theta}_i} \cap {H}$ %$\ell_2 = \bfL_{{\theta}_2, \p_2} \cap {H}$ 
    and ${\by} = \ell_1 \cap \ell_2$. {Also} let ${\by_2'}% = {\by}_1 + k \cdot {\theta}_2
    $ be the orthogonal projection of ${\by}_1$ on $\ell_2$ {(equivalently, on $\bfL_{p_2, {\theta}_2}$)}; then 
   { ${\|{\by}_1 - {\by_2'}\| %\leq |{\by}_1 - {\by}_2|_2 %\leq \sqrt{n} |{\by}_1 - {\by_2'}| 
   % \leq \sqrt{n} |{\by}_1 - {\by}_2| < \sqrt{n}
   < \varepsilon}$ as well}.
    Suppose ${\gamma}$ is the angle between $\ell_1$ and $\ell_2$; then ${\gamma} > \delta$. Therefore {$\|{\by} - {\by}_1\| = \frac{\|{\by}_1 - {\by_2'}\|}{\sin {\gamma}} < \frac{
    %\sqrt{n} 
    \varepsilon}{\sin \delta}$}, %\comm{(maybe a picture here?)} 
    
  \begin{figure}[h!]
  \begin{center}
\begin{tikzpicture}[>=stealth, thick]

    % Define Coordinates
    % X is the intersection point at origin
    \coordinate (X) at (0,0);
    
    % Define the lines
    % Horizontal line (now l1, was l2)
    % Extends from left (-5) to right (1.5)
    \coordinate (L_horiz_left) at (-5, 0);
    \coordinate (L_horiz_right) at (1.5, 0);
    
    % Slanted line (now l2, was l1)
    % Extends from top-left (145 degrees) to bottom-right (-35 degrees)
    \coordinate (L_slant_top) at (145:5);
    \coordinate (L_slant_bot) at (-35:1.5);
    
    \draw (L_slant_top) -- (L_slant_bot) node[above, at start] {\Large $\ell_2$};
    \draw (L_horiz_left) -- (L_horiz_right) node[below, at start] {\Large $\ell_1$};

    % Define points on the lines
    % X1 is a point on horizontal line (l1)
    \coordinate (X1) at (-2.5, 0);
    
    % X2' is the orthogonal projection of X1 onto slanted line (l2)
    \coordinate (X2p) at ($(X)!(X1)!(L_slant_top)$);

    % Draw the projection segment
    \draw (X1) -- (X2p);

    % Draw the Right Angle Marker at X2'
    % Using coordinate calculations to draw a square perpendicular to the slanted line
    \draw[thin] 
        ($(X2p)!0.25cm!(X)$) -- 
        ($($(X2p)!0.25cm!(X)$) + ($(X2p)!0.25cm!(X1)$) - (X2p)$) -- 
        ($(X2p)!0.25cm!(X1)$);

    % Draw and Label Points
    \fill (X) circle (2pt);
    % Positioning X: 'below' might hit the tail of the slanted line.
    % 'below left' moves it away from the intersection slightly into clear space.
    \node at (X) [below left=7pt] {\Large ${\by}$};
    
    \fill (X1) circle (2pt) node[below=5pt] {\Large ${\by}_1$};
    \fill (X2p) circle (2pt) node[above=5pt, xshift=3pt] {\Large ${\by_2'}$};

    % Draw the Angle theta
    % Angle between slanted line (top part) and horizontal line (left part)
    % Using 'pic text' is safer than quotes for preventing compilation errors
    \pic[draw, pic text={\Large ${\gamma}$}, angle radius=0.8cm, angle eccentricity=1.4] {angle = L_slant_top--X--L_horiz_left};

\end{tikzpicture}

\caption{Configuration {in} the plane $H$}\label{planeH}

\end{center}

\end{figure}
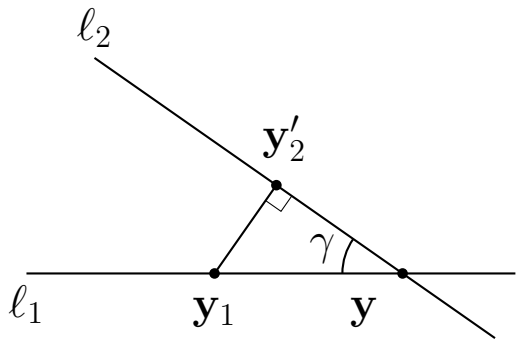

  \medskip  
    
    \noindent and thus {$\|{\by} - {\by}_1\| < \frac{
    %\sqrt{n}
    \varepsilon}{\sin \delta} $ {and $\|{\by} - {\by}_2\| < \left(\frac{%\sqrt{n}
    1}{\sin \delta} + 1\right) \varepsilon$}.}
\end{proof}

The next lemma is {a useful tool for studying total density}.

\begin{lemma}\label{usefullemma}
    Let $\fR \subseteq {\bf T}_{m,n}$. Fix $\br_0 = (\p_0, {\theta}_0) \in {\bf T}_{m,n}$,
    %and an open subset $W$ of $M_{m,n}$ intersecting ${\bfL}$.
   % Suppose there exists 
   {$\varepsilon>0$  %\varepsilon(W) > 0$ 
    and an open subset $U \subseteq M_{m,n}$} intersecting {$L:=\bfL_{\br_0}$} such that \eq{densityinU}{\text{the collection }\mathcal{L}_{\fR \smallsetminus \left(  B_{\varepsilon}(\br_0) \cup B_{\varepsilon}(-\br_0)\right)} \text{ is dense in }U.} Then \begin{equation}\label{local_dens_condition}
 \text{there exists a  {nonempty} %convex  
    open %set
    $\Omega \subseteq U$ such that }\Omega \subseteq \overline{\bigcup\limits_{\br \in \fR: \,\,\bfL_{\br} \cap {\bfL} \cap \Omega \neq \varnothing} {\bfL_{\br}}}.
 \end{equation}  
\end{lemma}

%{Note that the convexity of $\Omega$ is not needed for our argument and is included as part of the conclusions of the lemma for later applications.}
%it will be useful for additional 
%{
% We don't need convexity for total density, but we need it for local; in either case, it comes for free (we just need to check that the resulting set is convex). Since adding this (useless for this text) convexity condition allows us to directly use this lemma for local density without changes, I would suggest to leave it here in some form.}

\begin{proof}[Proof of Lemma \ref{usefullemma}]
    Let {$\rho_0 \geq 1$ be such %a constant 
    that $\|{Y}\| < \rho_0$} for any ${Y} \in {U}$. 
   Replacing $U$ by its subset, we can assume that $U$ is a $\sigma$-ball centered at $Y_0 \in {\bfL}$ {for some $0 < \sigma < \frac{\varepsilon}{2}$}.
    %is some parameter.
    
    Denote the orthogonal complement to $\bfL_{{\bf 0},{\theta}_0}$ by ${L'}$. Let $\rho$ be a function  defined %on ${\bfL}$, 
    by 
    $$
    \rho(Y) = \frac{\sigma - \|Y-Y_0\|}{C \left({\varepsilon}/{2 {\rho_0}} \right)}, \,\,\,\,\,\, \text{where } {C(\cdot}) \, \text{is {as} in Lemma \ref{anglelemma}.}
    $$ 
   % Let $U \subseteq W$ be an open set from the collection $\mathcal{U}$ intersecting ${\bfL}$; we can shrink $U$ if necessary and assume that $U$ is a ball of radius $\sigma$ centered at $y \in {\bfL}$; clearly, the same $\delta_U$ works for the re-defined $U$. 
   We define $B(Y): = Y + B_{\rho(Y)}(0) \cap {L'}$ and
   $$\Omega: = \bigcup\limits_{{Y} \in {\bfL}: \,\,\, \|Y - Y_0\| < \sigma} B(Y).$$
  It remains to show that $\Omega \subseteq U$ satisfies the conditions of {the lemma}. 

  \begin{itemize}
      \item $\Omega$ is open by {the} continuity of $\rho$. Indeed, let $Y=Y_1 + Y_2$, where $Y_1 \in {\bfL}$ and $Y_2 \in {L'}$ (this {decomposition} is unique). If $Y \in \Omega$, then $\|Y_1 - Y_0\| < \sigma$ and $\|Y_2\| < \rho(\|Y_1\|)$. By continuity of the function $\rho$ there exists $\delta > 0$ such that
      \begin{equation}\label{defdelta}
      \|Y_1' - Y_1\| <\delta \,\,\,\,\,\, \Longrightarrow \,\,\,\,\,\, \rho( Y_1') - \delta > \|Y_2\|,
      \end{equation}
      where $Y_1' \in {\bfL}$. It can be shown that $B_{\delta}(Y) \subseteq \Omega$: if $Y' = Y_1'+Y_2' \in B_{\delta}(Y)$, then ${\|Y_1' - Y_1\|} < \delta$ and thus, by \eqref{defdelta},
      $$
      \|Y_2'\| < \|Y_2\| + \delta < \rho(Y_1'),
      $$
      which by definition means that $Y' \in \Omega$.

  %    \item $\Omega$ is convex by construction. \comm{It's annoying that matrices are sometimes uppercase and sometimes lowercase, can we choose?} {Indeed,} let $v = v_1 + v_2$ and $u = u_1 + u_2$, where $v_1, u_1 \in {\bfL}$ and $v_2, u_2 \in \bfL_{\perp}$. We show that $w = tv + (1-t)u \in \Omega$, where $t \in [0,1]$ is arbitrary. It suffices to show that 
 %     $$
 %     |tv_2 + (1-t)u_2| < \rho\left( tv_1 + (1-t)u_1 \right).
 %     $$
%
 %     It directly follows as
  %    $$
 %     |tv_2 + (1-t)u_2| \leq t|v_2| + (1-t)|u_2|< t \rho(v_1) + (1-t) \rho(u_1)=
  %    $$
  %    $$
  %    \frac{\sigma - t|v_1-y_0| - (1-t)|u_1 - y_0|}{C \left({\varepsilon}/{2 {\rho_0}} \right)} \leq \frac{\sigma - | tv_1 + (1-t)u_1-y_0|}{C \left({\varepsilon}/{2 {\rho_0}} \right)} = \rho\left( tv_1 + (1-t)u_1 \right).
    %  $$
      \item Finally we check %the condition 
      {that the inclusion in \eqref{local_dens_condition} holds}. Let $V \subseteq \Omega$ be an open subset. {In view of %the density assumption,
      \equ{densityinU}} there {exist} $Y = Y_1 + Y_2 \in V$ and $\br = (\p, {\theta}) \in \fR \smallsetminus \big( B_{\varepsilon}(\br_0) \cup B_{\varepsilon}(-\br_0)\big)$  such that $Y \in \bfL_{\br}$ (where, as before, we assume $Y_1 \in {\bfL}$ and $Y_2 \in {L'}$). 
    It is now enough to show that 
    \begin{equation}\label{nonemptyint}
    \bfL_\br\cap {\bfL}\cap \Omega = \bfL_\br\cap {\bfL}\cap B_{\sigma}(Y_0) \neq \varnothing.
    \end{equation}
    \end{itemize}
    %There are two possible options:
{Now let us consider two cases.}

    %\begin{enumerate}
        %\item 
        \medskip

%\begin{enumerate}
    %\item[
    \noindent{\bf Case 1.} $\dist({\theta}_0, \pm {\theta}) > \varepsilon$: in this case by Lemma \ref{anglelemma} there exists $Z \in \bfL_{\br} \cap {\bfL}$ for which 
        $$
        \begin{aligned}
        \|Y_1- Z\| &< C(\varepsilon) \cdot \|Y - Y_1\| = C(\varepsilon) \cdot \|Y_2\|  <  C(\varepsilon) \rho(Y_1) \\
&=
        C(\varepsilon) \frac{\sigma - \|Y_1-Y_0\|}{C \left({\varepsilon}/{2 {\rho_0}} \right)} \leq \sigma - \|Y_1-Y_0\|.
        \end{aligned}
        $$
        {This} implies
        $$
        \|Y_0 - Z\| \leq \|Y_1 - Z\| + \|Y_1 - Y_0\| < \sigma,
        $$
        and thus $Z \in B_{\sigma}(Y_0)$, {hence} \eqref{nonemptyint} holds.

       \medskip

%\begin{enumerate}
    %\item[
    \noindent{\bf Case 2.} $\dist(\p_0, \pm \p) > \varepsilon$: let us show that in this case $\dist({\theta}_0, \pm {\theta})  > \frac{\varepsilon}{2\rho}$. 
        Indeed, assume that $\|\p - \p_0\| > \varepsilon$.  Since $Y \in \bfL_{\br}$, we know that $Y {\theta} = \p$ and hence
        $$
        \p - \p_0 = Y{\theta} - Y_0 {\theta}_0 = Y( {\theta} - {\theta}_0) + (Y-Y_0) {\theta}_0.
        $$
        Since $\sigma < \frac{\varepsilon}{2}$ and $Y \in U$, one has $\|Y - Y_0\| < \frac{\varepsilon}{2}$; therefore
        $$
        \|Y( {\theta} - {\theta}_0)\| \geq \|\p - \p_0\| - \|(Y-Y_0) {\theta}_0\| > \varepsilon - \frac{\varepsilon}{2} = \frac{\varepsilon}{2}
        $$
        and thus, as $\|Y\| < \rho$, one has $\|{\theta} - {\theta}_0\| > {\varepsilon}/{2 {\rho_0}}$. Analogously, the condition $\|\p + \p_0\| > \varepsilon$ implies $\|{\theta} - {\theta}_0\| > {\varepsilon}/{2 {\rho_0}}$.
        The remaining part of   the proof is the same as in the previous case, with the only difference that $C(\varepsilon)$ is replaced by $C \left({\varepsilon}/{2 {\rho_0}} \right)$.
   % \end{enumerate}
    %\end{itemize}
    %This completes the proof.
 \end{proof}

{Before formulating the next proposition, which is a direct application of Lemma \ref{usefullemma},} {let us} introduce some notation. {For  a collection  $\{ N_0, N_1, \ldots, N_l \}$ of sets, where  $l \in\Z_{\ge 0}$, %let 
%Consider a collection  $\{ N_0, N_1, \ldots, N_l \}$ of closed { and symmetric (that is, $S = -S$)} subsets of ${\bf T}_{m,n}$. 
define 
$$N: = \bigcup\limits_{i=0}^l N_i,\quad %\,\,\,\,\,\,\,\, 
N_i^* := \bigcup\limits_{j\neq i} N_j,
\ \text{ and } \ 
%\,\,\,\,\,\,\,\, S: = 
%\bigcup\limits_{ i<j} N_i \cap N_j 
N_i' := N\smallsetminus N_i^*
$$
($N_i'$ is the set of points that belong to $N_i$ and to no other set  from this collection). 
%Also for $i = 0,1,\dots,l$ define
%$$
%{N}'_i := N_i \smallsetminus S; %= N_i  \smallsetminus \bigcup\limits_{j\ne i} N_j,
%\,\,\,\,\,\,\,\, \text{and} \,\,\,\,\,\,\,\, {M}' := M \smallsetminus S = M \smallsetminus N;
%$$
%that is, the sets of points that belong to $N_i$ %(respectively, to $M$)
%and do not belong to any other set from the collection.}

\begin{proposition}\label{usefulforproducts}
    Let $\fR \subseteq {\bf T}_{m,n}$ and $l \in\Z_{\ge 0}$. Suppose there exist centrally  symmetric closed sets $N_0 \subseteq \overline{\fR}\smallsetminus \fR$ and $N_i \subseteq \overline{\fR}$, $i = 1, \ldots, l$, such that the following holds:
    \begin{itemize}
        \item[\rm(i)] $N_i \subseteq \overline{N_i'}$ for any $i = 0, \ldots, l$;
        \item[\rm(ii)] The collections
    $
  \mathcal{L}_{N}$   and $\mathcal{L}_{ {N}^*_i}, \  i=1, \ldots, l$, 
  %{  \text{(and $\mathcal{L}_{N_0},\,\,$ if $l=0$)}   }
    are dense.
    \end{itemize}
Then %$\mathcal{L}_{\fR}$ contains a subcollection which is totally dense.
$E$ has property {\rm (TDS)}.
\end{proposition}

We note that if $l>0$, then $N_i^*\subset N$ for any $i=1, \ldots, l$, hence the density of  $\mathcal{L}_{N}$ follows from that of $\mathcal{L}_{ {N}^*_i}$; however if $l=0$ one needs to check the density of $\mathcal{L}_{N}$ separately.
%\comm{A question: can we use $N_0$ instead of $M$? this might simplify  notation quite a bit.}

\begin{proof}[Proof of Proposition \ref{usefulforproducts}]
Let $U_i$ be {an open neighborhood of ${N}'_i$ such that} $U_i \cap N_i^*= \varnothing$. % Let $V$ be an open neighborhood of ${M}'$ that $V \cap N = \varnothing$.
Such neighborhoods exist {because} the sets $N_i^*$ are closed; moreover: we can assume that these neighborhoods are {centrally} symmetric and disjoint. For instance: fix some $i = 0, \ldots, l$ and let $U_i$ consist of $\frac{\dist(\br, N_i^*)}{2}$-neighborhoods of points $\br \in N_i'$. We %will
define 
$$
\fR': = \fR \cap  \bigcup\limits_{i=0}^l U_i 
$$
and prove that $\mathcal{L}_{\fR'}$ is totally dense.

    A key observation {here} is that the collections $\mathcal{L}_{\fR' \smallsetminus U_i}$ are dense for $i = 1, \ldots, l$. Indeed: 
    For each point $\br \in N_i'$ and any $\rho > 0$ {we have} $\br \in \overline{\fR \cap B_{\rho}(\br)}$. {It implies that}
        $$
        N_i \subseteq \overline{N_i'} \subseteq \overline{\fR \cap U_i}, \,\,\,\,\,\,\,\, i = 1, \ldots, l,  \,\,\,\,\,\,\,\, \text{and} \,\,\,\,\,\,\,\, N_0 \subseteq \overline{N_0'} \subseteq \overline{\fR \cap U_i } \smallsetminus \fR'.
        $$

    Therefore, 
    $$
    \overline{\fR' \smallsetminus U_i} \supseteq \bigcup\limits_{j\neq i} \overline{\fR \cap U_j} \supseteq N_i^* \,\,\,\,\,\,\,\, \text{for} \,\,i = 1, \ldots, l,
    $$
    thus {the collections $\mathcal{L}_{\overline{\fR' \smallsetminus U_i}}$, and hence (in view of  {Lemma} \ref{closure}) $\mathcal{L}_{\fR' \smallsetminus U_i}$, are} dense for $i = 1, \ldots, l$. In addition, the collection $\mathcal{L}_{\fR'}$ is dense.

     It remains to prove the second %property 
    {part of the definition} of total density. We will show that \eqref{local_dens_condition} holds for any $\br_0 \in \fR'$ and any open set $U$ in $M_{m,n}$ intersecting $\bfL: = \bfL_{\br_0}$.
    %there exists a nonempty open set $\Omega \subseteq U$ such that . There are two cases:
    %By Proposition \ref{closure}, it is enough to prove total density for $\overline{\fR}$. % we will prove that $\mathcal{L}_{\fR'}$ where $\fR' = \left( \fR \cup N \right) \smallsetminus N'$ is totally dense.
    %Note that
    %\begin{equation}\label{kminusn_in_closure}
    %    K \subseteq \overline{\fR'} \,\,\,\,\,\,\,\,\,\,\,\, \text{for any} \,\,\,\, i = 1, \ldots, l.
    %\end{equation}

     \medskip

%\begin{enumerate}
    %\item[
    \noindent{\bf Case 1.} Consider $\br_0 \in U_i \cap \fR$, {where $i = 1, \ldots, l$}. The set $U_i$ is open, {hence} there exists $\varepsilon > 0$ such that $B_{\varepsilon}(\br_0) \subseteq U_i$. {Since each $U_i$ is symmetric,   $B_{\varepsilon}(-\br_0) \subseteq U_i$ as well}, and therefore
        $$
        \fR' \smallsetminus \big( B_{\varepsilon}(\br_0) \cup B_{\varepsilon}(\br_0) \big) \supseteq \fR' \smallsetminus U_i,
        $$
        the collection $\mathcal{L}_{\fR' \smallsetminus \left( B_{\varepsilon}(\br_0) \cup B_{\varepsilon}(\br_0) \right)}$ is dense. By Lemma \ref{usefullemma}, for any open $U$ intersecting $\bfL$ %\bfL_{\br_0}$ \
        %there exists a  {nonempty} %convex  
   % open %set 
   % $\Omega \subseteq U$ such that 
   condition \eqref{local_dens_condition} is satisfied.

        \medskip

%\begin{enumerate}
    %\item[
    \noindent{\bf Case 2.} Consider $\br_0 \in U_0 \cap \fR$. The set $N_0$ is closed and $\br_0 \notin N_0$, {hence} $\dist(\br_0, N_0) > 0$. Take $\varepsilon > 0$ %small enough, 
    such that $2 \varepsilon < \dist(\br_0, N_0)$ and $B_{\varepsilon}(\br_0) \subseteq U_0$; by symmetry, we immediately have $B_{\varepsilon}(-\br_0) \subseteq U_0$. In this case, 
        $$
        N_0 \subseteq \overline{\fR \cap U_0 \smallsetminus \big( B_{\varepsilon}(\br_0) \cup B_{\varepsilon}(\br_0) \big)}
        $$
        and thus 
        $$
        \overline{\fR' \smallsetminus \big( B_{\varepsilon}(\br_0) \cup B_{\varepsilon}(\br_0) \big)} \supseteq N.
        $$
         Since $\mathcal{L}_{N}$ is dense,   the collection $\mathcal{L}_{\fR' \smallsetminus \left( B_{\varepsilon}(\br_0) \cup B_{\varepsilon}(\br_0) \right)}$ is dense, and the validity of \eqref{local_dens_condition} for any open $U$ intersecting $\bfL$  again follows from Lemma \ref{usefullemma}.% there exists a  {nonempty} %convex  
   % open set $\Omega \subseteq U$ such that  holds. }
    %\end{itemize}
    
   % \comm{Vasya sorry, I do not understand thе final part of the  proof. You are showing that \eqref{local_dens_condition} holds, but \eqref{local_dens_condition} involves $\Omega$, and you are not defining it here. Either modify \eqref{local_dens_condition} or construct $\Omega$ before you start the two cases. Also how to understand "\eqref{local_dens_condition} holds for any $\br_0 \in \fR'$"? there is no $\br_0$ in \eqref{local_dens_condition}. Please change! -- I modified the text, please check. There is $\br_0$ in Lemma \ref{usefullemma}, but I tried to make it more visible now.}
\end{proof}

%\section{On a  generalization of Caratheodory's theorem}

%\subsection{Convex hulls containing the origin}

%For each $\theta\in {\mathbb S}^{n-1}$, let $S_\theta=\{\phi\in {\mathbb S}^{n-1}: \,\,\,\phi\cdot\theta\ge 0\}$. The following is a useful equivalent condition to convex hull of a set containing the origin:

%\begin{lemma}\label{zero_in_conv}
%    For $Q \subseteq {\mathbb S}^{n-1}\subseteq \mathbb{R}^n$ the following statements are equivalent:

%    \begin{enumerate}
%        \item\label{partone} $0 \in \conv(Q)$;
%        \item\label{parttwo} For any $\theta \in {\mathbb S}^{n-1}$ the intersection $S_{\theta} \cap Q$ is nonempty.
%    \end{enumerate}
%\end{lemma}

%\begin{proof}
%    {\bf \ref{partone} $\, \Rightarrow \,$ \ref{parttwo}.} By \ref{partone}, there exist such elements $q_1, \ldots, q_s \in Q$ and scalars $\lambda_1, \ldots, \lambda_s > 0$ such that $\sum\limits_{i=1}^s \lambda_i q_i = 0$. Suppose there exists $\theta$ such that $S_{\theta} \cap Q = \varnothing$; in this case, $\theta \cdot q_i  < 0$ for any $i$, and thus 

%    $$
%    0 = \theta \cdot 0 = \theta \cdot \left( \sum\limits_{i=1}^s \lambda_i q_i \right) = \sum\limits_{i=1}^s \lambda_i \theta \cdot q_i < 0,
%    $$

%    which provides a contradiction.

%    \vskip+0.2cm

%    {\bf \ref{parttwo} $\, \Rightarrow \,$ \ref{partone}.} A special case of { \it Separating hyperplane theorem} (see \cite{BV04}, Section 2.5.1) for two convex sets: $\{ 0 \}$ and the convex hull of $Q$.
%\end{proof}

\section{{Proof of the} main results}\label{proofs}

\subsection{Proof of Theorem \ref{any_k_theorem}}\label{proof_anyk_section}
{Recall that we are given $\fR \subseteq { \bf T}_{m,n}$ and  $0 \leq k \leq \min(n-1, m)$ such that  one of the conditions \eqref{anykA},  \eqref{anykB} or  \eqref{anykC} holds. In
each of these cases} we  apply Proposition \ref{usefulforproducts}.

\begin{proof}[Proof of Theorem \ref{any_k_theorem}] We label the cases according to the conditions met in each of them. 
\medskip

    \noindent{\bf Case \ref{anykA}.} {We are given $\Phi_1, \Phi_2$ and ${H}$ such that $
        {H} \times \left( \Phi_1 \cup \Phi_2 \right) %\cap \mathbb{S}^{n-1} %\right) 
        \subseteq \overline{\fR}$.} Let 
    $$
    {N_1 := %\left( 
    {H} \times\Phi_1 %\cap \mathbb{S}^{n-1} \right)
    %\times {H} 
    \,\,\,\,\,\,\,\, \text{and} \,\,\,\,\,\,\,\, N_2 := %\left( 
   {H} \times \Phi_2 %\cap \mathbb{S}^{n-1} \right)}
    %\times {H}
    .}
    $$
    The sets $N_1$ and $N_2$ are closed symmetric subsets of ${ \bf T}_{m,n}$, and the set $S := {H} \times\left( \Phi_1 \cap \Phi_2 %\cap \mathbb{S}^{n-1} 
    \right) %\times {H}
    $ belongs to the closure of both sets $N_i' := N_i \smallsetminus S$. By Proposition \ref{density_any_k}, the collections $\mathcal{L}_{N_1}$ and $\mathcal{L}_{N_2}$ are dense. Therefore, by Proposition \ref{usefulforproducts} %$\mathcal{L}_{\fR}$ contains a totally dense subcollection
    {$E$ has property {\rm (TDS)}}.

    \medskip

    \noindent{\bf Case \ref{anykB}.}  {Now we are given $\Phi$, ${H}_1$ and ${H}_2$ such that $
        %\Phi %\cap \mathbb{S}^{n-1} 
%\right) 
%\times 
\left( {H}_1 \cup {H}_2 \right)\times\Phi  \subseteq \overline{\fR}$. We can argue as in} the previous case, with the difference that we define
    $$
    {N_1 := %\left( 
    %\Phi %\cap \mathbb{S}^{n-1} \right) 
    %\times 
    {H}_1 \times\Phi\,\,\,\,\,\,\,\, \text{and} \,\,\,\,\,\,\,\, N_2 := %\left( 
    %\Phi %\cap \mathbb{S}^{n-1} \right) 
    %\times 
    {H}_2\times\Phi.}
    $$
    
    \noindent{\bf Case \ref{anykC}.} {Here we are given $\Phi$ and ${H}$ such that %\left( 
       $%\Phi %\cap \mathbb{S}^{n-1} %\right) 
       %\times 
       {H} \times\Phi\subseteq \overline{\fR \smallsetminus \left( %\left( 
       {H} \times\Phi \right)}$.} We will replace the set $\fR$ with {$\fR \smallsetminus   \left( {H} \times\Phi \right)$} and prove that $\mathcal{L}_{\fR}$ contains a totally dense subcollection. {Indeed}, let $N_0: = %\left( 
       {H} \times\Phi$ {and $l=0$}; then $N_0 \subseteq \overline{\fR} \smallsetminus \fR$. Since the set $S$ (as defined in Proposition \ref{usefulforproducts}) is empty, $N_0' = N_0$ and we trivially have that $N_0 \subseteq \overline{N_0'}$. Finally, by Proposition \ref{density_any_k} the set $\mathcal{L}_{N_0}$ is dense, which allows us to use Proposition \ref{usefulforproducts} and conclude that %$\mathcal{L}_{\fR}$ contains a totally dense subcollection
       {$E$ has property {\rm (TDS)}}.
\end{proof}

\subsection{Proof of Theorem \ref{ch_theorem}}\label{Ch_proof_section}

We start {this subsection} with a simple observation from convex geometry, which {happens to be} useful for Theorem \ref{ch_theorem} {and} may also have a separate interest. Let us recall the classical Caratheodory Theorem (\cite{caratheodory}, proven in {full generality} by Steinitz in \cite{steinitz}; see \cite[Chapter 2.2]{egg}   for details):

\begin{thm}\label{caratheodory_theorem}
    Let $S$ be a subset of $\mathbb{R}^d$, and suppose {that}
    $
    0 \in \conv(S).
    $
    Then there exists {a subset $S'$ of $S$ of cardinality at most  $d+1$
    %0 < k \leq d+1$ and elements $s_1, \ldots, s_k \in S$, 
    such that $0 \in \conv(S')$}.
\end{thm}

We %show 
{will need} the following %straightforward generalization 
{modification} of %Caratherodory's 
Theorem \ref{caratheodory_theorem}:

\begin{theorem}\label{generalized_caratheodory}
    Let $S$ be a subset of $\mathbb{R}^d$ and suppose %$S = S_1 \sqcup S_2$, where 
    {$S_1\subset S$ is such that}
\begin{equation}\label{general_caratheod_condition}
    \text{for any} \,\, s \in S_1, \,\,\, 0 \in \conv\left( S \smallsetminus \{ s \} \right).
\end{equation}
Then {there exists  a subset $S'$ of $S$ of cardinality} 
%, there exist such $S_2' \subseteq S_2, \,\, S_1' \subseteq S_1$ such that the set $S' = S_1' \sqcup S_2'$ has 
at most $(d+1)(d+2)$ such that
\begin{equation}\label{new_caratheod_condition}
    {\text{for any} \,\, s \in S_1, \,\,\, 0 \in \conv\left( S' \smallsetminus \{ s \} \right).}
\end{equation}
%elements and \eqref{general_caratheod_condition} holds with $S$ replaced with $S'$.
\end{theorem}

\begin{proof}
    It follows from \eqref{general_caratheod_condition} that $0 \in \conv(S)$. By Theorem \ref{caratheodory_theorem} there exists a {subset $\tilde{S} = \{ s_1, \ldots, s_k \}$ of $S$ such that} $k \leq d+1$ and $0 \in \conv(\tilde{S})$. Let us assume without loss of generality that $${\tilde S\cap  S_1 = \{s_1, \ldots, s_r\}  \text{ for some } r = 0,\dots, k.}$$ For any $i \in \{ 1, \ldots, r \}$ the set $\conv\left( S \smallsetminus \{ s_i \} \right)$ contains the origin, so by Theorem \ref{caratheodory_theorem} there exists %such a set 
    {a subset $\tilde{S}_i$ of    $S \smallsetminus \{ s_i \}$ of cardinality $ \leq d+1$ such that} $0 \in \conv\big( \tilde{S}_i \big)$. %It is easy to check 
    {We claim} that the set $$S' := \tilde{S} \cup \bigcup\limits_{i=1}^r \tilde{S}_i$$ satisfies the {conclusion} of the theorem. Indeed, 
    %it is easy to see that 
    the cardinality of $S'$ is at most $$ k + r   (d+1) \leq d+1 + (d+1)^2 = (d+1)(d+2).$$ Now take $s \in S_1$. If  $s\in\tilde S$, then   $s = s_i$ for some $1 \leq i \leq r$,   and therefore
         $$
       \conv\left( S' \smallsetminus \{s\} \right) = \conv\left( S' \smallsetminus \{s_i\} \right) \supseteq  \conv\big( \tilde{S}_i \big) \ni 0.
        $$
        {Otherwise we have $S' \smallsetminus \{s\}  \supseteq    \tilde{S}$, hence}
        $
        \conv\left( S' \smallsetminus \{s\} \right)  \supseteq  \conv\big( \tilde{S} \big) \ni 0.
        $ 
{This finishes the proof.}\end{proof}

%\begin{remark}
%    Iterating this process can be used to prove more general statements: namely, one can wish to remove any prescribed number of points from several parts of a given set instead of just one. \comm{I am not sure if we really need this remark.}
%\end{remark}

\begin{proof}[Proof of Theorem \ref{ch_theorem}] 
{Recall that we are given a subset} $\fR$ of $ { \bf T}_{m,n}$, two subsets $\Theta_1, \Theta_2$ of $ \mathbb{S}^{n-1}$ and a closed linear half-space ${H}$ of $\mathbb{R}^{m}$  satisfying   \eqref{bigtheta1},  \eqref{bigtheta2} and \eqref{ch_theorem_condition}.
%Before proving Theorem \ref{ch_theorem}
 {Let us also define $$
\Theta : =  \Theta_1 \cup \Theta_2.$$
%$\Theta': = \overline{\Theta_1 \cup \Theta_2}$, and define 
%The following statements are straightforward:
Then  \eqref{bigtheta1} and \eqref{bigtheta2} immediately imply, by way of monotonicity and taking closures,} %\comm{(maybe add some explanations? -- I don't know what particularly: \eqref{bigtheta1} and \eqref{bigtheta2} already say that ${\Theta}_i \times {H} \subseteq \overline{\fR}$ ... OK)} 
that
\begin{equation*}
    {H \times {\overline{\Theta}} } \subseteq \overline{\fR},
    \end{equation*}
    and it follows from \eqref{ch_theorem_condition} that 
\begin{equation}\label{conv_of_theta_prime}
        0 \in \conv\big( {\overline{\Theta}}\big).
    \end{equation}
    %by monotonicity and taking closures
%\begin{itemize}
%    \item 
 %   \begin{equation}
 %       {\overline{\Theta}} \times \Pi \subseteq \overline{\fR}
 %   \end{equation}
%
 %   (follows from conditions \eqref{bigtheta1} and \eqref{bigtheta2} by monotonicity and taking closures). 
%
 %   \item \begin{equation}\label{conv_of_theta_prime}
 %       0 \in \conv( {\overline{\Theta}})
 %   \end{equation}
%
 %   (since ${\overline{\Theta}} = \overline{\Theta_1 \cup \Theta_2}$).
%\end{itemize}

%We will need two more useful statements:

%As a first step 
{Next we are going to modify the sets $\Theta_1$ and $\Theta_2$ by defining
\eq{deftheta}{
%\Theta': = \overline{\Theta_1 \cup \Theta_2}, \,\,\, 
\Theta_2': = \big\{ \theta \in {\overline{\Theta}}:  {H}\times \{\theta\} %\times \Pi
\subseteq \overline{\fR \smallsetminus \left({H}\times  \{ \theta \}  \right)}\big\}    \ \text{ and }  \  \Theta_1': = {\overline{\Theta}} \smallsetminus \Theta_2'.
}
Note that \eqref{bigtheta2} immediately implies that $\Theta_2\subseteq \Theta_2'$. Let us prove the following two   useful statements:
\begin{lemma}\label{claim_theta1}
One has%    \begin{enumerate}
        \eq{claimA}{ %One has 
        {\Theta_1' \subseteq \Theta_1}}
        and
       \eq{claimB}{\text{every $\theta \in \Theta_1'$ is an isolated point of } %the set 
        { \Theta}; \text{ equivalently, } {\theta \in \Theta_1'\ \Longrightarrow\ \overline{\Theta}} \smallsetminus \{ \theta \} = \overline{\Theta \smallsetminus \{ \theta \}}.}
\end{lemma}}

\begin{proof}
    Let ${\Lambda}$ denote the set of limit points of   ${\Theta}$ (and thus also of ${\overline{\Theta}}$). It %is enough for us 
    {suffices} to show that 
    \begin{equation}\label{L(theta)}
        {\Lambda} \subseteq \Theta_2'.
    \end{equation}
    Indeed: in this case if $\theta \in {\overline{\Theta}}$ is not an isolated point of ${ \Theta}$, then $\theta \in \Theta_2' = {\overline{\Theta}} \smallsetminus \Theta_1'$, which immediately implies \equ{claimB}. To show that \equ{claimA} follows, notice that \eqref{L(theta)} yields
    $$
    \Theta_2 \cup \big({\overline{\Theta}} \smallsetminus \Theta)  \subseteq \Theta_2 \cup {\Lambda} \subseteq \Theta_2',
    $$
    and thus $\Theta_1' \subseteq \Theta \smallsetminus \Theta_2 = \Theta_1$.

    It remains to show that \eqref{L(theta)} holds. Take $\theta \in {\Lambda}$, and let $\{ \theta_i \}$ be a sequence of elements of $\Theta$ converging to $\theta$ (where $\theta_i \neq \theta$ for any $i$). Since $\theta_i \in \Theta$, we have $H \times\{ \theta_i \} %\times {H} 
    \subseteq \overline{\fR}$; and since $\theta_i \neq \theta$, 
    $$
    H \times\{ \theta_i \} %\times {H} 
    \subseteq \overline{\fR \smallsetminus \left( {H} \times \{\theta\} \right)} \,\,\,\,\,\,\,\, \Longrightarrow \,\,\,\,\,\,\,\, H \times\bigcup\limits_{i=1}^{\infty} \{ \theta_i \} %\times {H} 
    \subseteq \overline{\fR \smallsetminus \left( {H} \times \{\theta\} \right)};
    $$
    since $\theta \in \overline{\bigcup\limits_{i=1}^{\infty} \{ \theta_i \}}$, taking the closures of both sides shows that ${H} \times \{\theta\} \subseteq \overline{\fR \smallsetminus \left( {H} \times \{\theta\} \right)}$, which by definition means that $\theta \in \Theta_2'$ and completes the proof.
\end{proof}

%\begin{corollary}{(Directly follows from \equ{claimB})}\label{cor1_convexproof}
 %   For any $\theta \in \Theta_1'$, 
%    $$
 %   {\overline{\Theta}} \setminus \{ \theta \} = \overline{{\overline{\Theta}} \setminus \{ \theta \}}.
%    $$
%\end{corollary}

Combining {\equ{deftheta}, \equ{claimA} and \equ{claimB}}, we conclude that
\begin{equation}\label{thetaprime_conv_condition}
    \text{for any } \theta \in \Theta_1', \,\,\,\, \,\,\,\, 0 \in \conv\left({\overline{\Theta}} \smallsetminus \{ \theta \}\right) = \conv(\Theta_1' \sqcup \Theta_2' \smallsetminus \{ \theta \}).
\end{equation}
{Using} Theorem \ref{generalized_caratheodory} we can replace $\Theta_1'$ and $\Theta_2'$ by their finite subsets such that \eqref{thetaprime_conv_condition} holds. If the set $\Theta_1'$ happens to be empty, we apply Theorem \ref{caratheodory_theorem} and replace the set ${\overline{\Theta}} = \Theta_2'$ by its finite subset such that \eqref{conv_of_theta_prime} holds. 

\medskip

{Now} replace the set $\fR$ {with} its subset $\fR \smallsetminus \left( {H}\times \Theta_2% \times {H} 
\right)$, and then symmetrize {by} replacing $\fR$ with $\fR \cup (-\fR)$ (the latter, {as was  mentioned before,} does not affect the collection $\mathcal{L}_{\fR}$). Theorem \ref{ch_theorem} will follow if we prove that 
%$\mathcal{L}_{\fR}$ contains a totally dense subcolleciton.
{$E$ has property {\rm (TDS)}}. We will use {Proposition} \ref{usefulforproducts} to prove this.

\smallskip

Let $\Theta_1' = \{ \theta_1, \ldots, \theta_l \}$ be an enumeration of {the} finite set $\Theta_1'$. We define symmetric closed sets 
$$
{N_0 : = ({H}\times \Theta_2 %\times {H}
) \cup \big(- \left({H}\times \Theta_2 %\times {H} 
\right)\big), \,\,\,\,\,\,\,\, N_i : = ({H}\times \{ \theta_i \} %\times {H}
) \cup \big(- \left( {H}\times \{ \theta_i \} %\times {H}
\right)\big),} \,\,\,\,\,\,\, i = 1, \ldots, l.
$$
The pairs of sets $N_i$ and $N_j$ can only intersect %by 
{at} their boundary points, therefore $N_i \subseteq \overline{N_i'}$ for any $i = 0, \ldots, l$. It also follows from \eqref{thetaprime_conv_condition} and Proposition \ref{ditect product dense}    that the collections 
$$
\mathcal{L}_{N_0 \cup N \smallsetminus N_i'}, \,\,\,\, i = 1, \ldots, l
$$
are dense. In the case   $l=0$, it follows from \eqref{conv_of_theta_prime} and Proposition \ref{ditect product dense}    that the collection $\mathcal{L}_{N_0}$ is dense. {It remains to apply Proposition \ref{usefulforproducts} and finish the proof.} %$\mathcal{L}_{\fR}$ contains a totally dense subcollection.
\end{proof}

\section{{On the optimality of
Theorems 
\ref{any_k_theorem} and \ref{ch_theorem}}\label{optimality} % and} %Theorem 
%our results
}

%\subsection{On the optimality of conditions in Theorem \ref{ch_theorem} and Theorem \ref{any_k_theorem}}

%{ 

%Before proving Lemma \ref{ch_not_tot_dense},

%\comm{OK, something should be said right here. Do we or do we not have optimality of these theorems? or some kind of partial optimality? The way it is written right now  is too mysterious.}

%\comm{Move to section 3}
%\comm{Some introduction: we are going to discuss in what sense our theorems are optimal.}

In this section we show that our main results,  Theorems 
\ref{any_k_theorem} and \ref{ch_theorem}, are optimal in a certain sense. In order to do %it
{so} more efficiently it will be convenient to introduce a certain terminology and  {use it to restate our results}.
%{Recall that   the first two parts
%of Theorem
%\ref{any_k_theorem} were  of the following type:
%if the closure of ${E}$ contains a certain closed subset ${D}$ of $\tmn$, then the collection $\mathcal{L}_{\fR}$   contains a totally dense subcollection. Namely, ${D}$ was equal to   ${H}\times \left(\Phi_1 \cup \Phi_2 \right)$ in part (a), and $\left( {H}_1 \cup {H}_2 \right) \times \Phi$ in part (b). Moreover, in part (c) the same conclusion was reached under the assumption that the closure of ${E}\smallsetminus {D}$ contains   ${D} = {H}   
%       \times \Phi$. Our first result in this section   shows that a similar statement fails if ${D}$ is replaced with  any proper closed symmetric subset ${D}'$ of $D$.}
\begin{definition}\label{generaltds} Suppose we are interested in a certain property of subsets of some topological space $X$, which we will here refer to as property {\rm ($*$)}. Let $D_0$ be a subset of $X$, and let $\{D_\theta\}_{\theta \in \Theta}$ be a family of subsets of $X$, where $\Theta$ is a set of indices (perhaps uncountable).
Say that a collection $\big(D_0,\{D_\theta\}_{\theta \in \Theta}\big)$   \underline{has property {\rm ($*\pm$)}}  if 
%any dense subset $E$ of $E$ has property (TDS); equivalently, if 
any subset  $E$ of $X$ such that
%\begin{itemize}
%\item 
$$D_0\subseteq \overline{E}\qquad\text{and}\qquad
%\item $
D_\theta\subseteq \overline{E\smallsetminus D_\theta}\text{ for any }\theta \in \Theta$$
%\end{itemize}
has property  {\rm ($*$)}. {That is,} to prove {\rm ($*$)} for $E$ it is enough to check that $E$ contains $D_0$ in its closure, and 
that for each $\theta$ every point of $D_\theta$ can be approximated by points of $E$ that are not in $D_\theta$.

Also, let us say that:
\begin{itemize}
\item $D_0$  \underline{has property {\rm ($*$+)}} if $(D_0,\varnothing)$  has property {\rm ($*\pm$)}; that is, if $D_0\subseteq \overline{E}\text{ guarantees  {\rm ($*$)} for }E$.  
\item  $\{D_\theta\}_{\theta \in \Theta}$  \underline{has property {\rm ($*$--)}} if $(\varnothing,\{D_\theta\})$  has property {\rm ($*\pm$)}; equivalently, if $$D_\theta\subseteq \overline{E\smallsetminus D_\theta}\text{ for any $\theta \in \Theta$ implies  that $E$ has property {\rm ($*$)}}. $$ When $\Theta$ is a singleton, we will say that $D$  {has property {\rm ($*$--)}} if $(\varnothing,\{D\})$  has property {\rm ($*\pm$)}; equivalently,  if $D \subseteq \overline{E\smallsetminus D}$ implies  that $E$ has property {\rm ($*$)}.
\end{itemize}
\end{definition}
{Note that  ($*$+) clearly implies ($*$); however finding a set $D_0$ with property ($*$+) is much more valuable than with just ($*$), since it immediately establishes ($*$) for any set having $D_0$ in its closure.  On the other hand ($*$--) for $D$ does not necessarily imply ($*$) for $D$, but  does imply it for other sets.}

%\comm{Should we discuss it some more? Does it still look horrible? I am not sure if we really need to have this formalism, but you can see how nicely it transforms our theorems below, so perhaps it's worth it... -- Honestly I think it still looks rather scary, but maybe it's fine for Section 6. However I think we should provide reader an option to understand the formulations of Proposition 6.2 and 6.4 without reading this definitions. I'm not sure if it can be done without formulating them twice...}
\smallskip

All this looks somewhat formal and abstract. Yet upon inspection of our main theorems one can observe that  they actually exhibit certain subsets of $\tmn$ having {the} properties defined above, with ($*$) replaced by (TDS). Namely, the following are equivalent ways of stating our theorems:
\medskip

%\vfill\eject
\noindent{\bf Theorem 
 \ref{any_k_theorem}.}
Fix $m,n\in\N$ and $0 \leq k \leq \min(n-1, m)$, let $\Phi_1 \ne \Phi_2$ and $\Phi$ be $k$-dimensional subspheres   $\mathbb{S}^{n-1}$, and let ${H}_1\ne {H}_2$ and ${H}$ be $(m-k)$-dimensional linear subspaces   of $\mathbb{R}^m$. Then:
    \begin{enumerate}
        \item[\rm (ab)]   %The sets 
        ${H}\times \left(\Phi_1 \cup \Phi_2 \right)$ and $\left( {H}_1 \cup {H}_2 \right) \times \Phi$  have property {\rm (TDS+)};
 \item[\rm (c)]  %the set 
 ${H} %\cap \mathbb{S}^{n-1} %\right) 
       \times \Phi $ has property {\rm (TDS--)}.
\end{enumerate}
%\end{reptheorem}
 
\medskip
\noindent{\bf Theorem 
\ref{ch_theorem}.}  For $m,n\in\N$, let $\Theta_1, \Theta_2$ be two disjoint {(possibly empty)} subsets of $\mathbb{S}^{n-1}$ such that \eqref{ch_theorem_condition} holds,
and let  ${H}$ be a closed linear half-space  of $\mathbb{R}^{m}$. Then the collection \eq{collectionplusminus} {\left(H\times \Theta_1, \big\{H\times \{\theta\}\big\}_{\theta \in \Theta_2}\right)}}   has property {\rm (TDS$\pm$)}. 

\medskip
 We now explain in what sense the above theorems cannot be improved. %The next proposition shows that no proper symmetric closed subset of  ${H}\times \left(\Phi_1 \cup \Phi_2 \right)$, $\left( {H}_1 \cup {H}_2 \right) \times \Phi$ and 
 %${H} \times \Phi$ respectively can be used in place of those sets. 
 
\begin{proposition}\label{any_k_theorem_optimality_new}
    For any $m,n\in\N$ and  any   $0 \leq k \leq \min(n-1, m)$, the following  holds:
    %\begin{enumerate}
        %\item\label
        \begin{itemize}
        \item[\rm (ab)]
        $ {H}\times \left(\Phi_1 \cup \Phi_2 \right) $ and $\left( {H}_1 \cup {H}_2 \right) \times \Phi$ as in Theorem 
 \ref{any_k_theorem}(ab) are minimal   closed symmetric sets with property {\rm (TDS+)};
 %; in other words, their proper  closed symmetric subsets do not have this property.
 \item[\rm (c)] sets of the form ${H} \times \Phi$ as in Theorem 
 \ref{any_k_theorem}(c) are minimal   closed symmetric sets with property {\rm (TDS-)}.
 
       \end{itemize}
\end{proposition}

We remark that, by %way of 
unwrapping the definitions,   Proposition \ref{any_k_theorem_optimality_new}(ab) asserts that no proper  closed symmetric subset $D$ of $ {H}\times \left(\Phi_1 \cup \Phi_2 \right) $ and $\left( {H}_1 \cup {H}_2 \right) \times \Phi$ as in  Theorem 
 \ref{any_k_theorem}(ab) has property (TDS+); that is, for any such $D$ it is possible to find $E\subset \tmn$ such that $D\subseteq \overline{E}$ but $\mathcal{L}_E$ does not contain a totally dense subcollection. Likewise, Proposition \ref{any_k_theorem_optimality_new}(c) asserts that no proper  closed symmetric  $D\subset {H}\times   \Phi$ as in  Theorem 
 \ref{any_k_theorem}(c) has property (TDS--); that is, for any such $D$ it is possible to find $E\subset \tmn$ such that $D\subseteq \overline{E\smallsetminus D}$, but $\mathcal{L}_E$ does not contain a totally dense subcollection. %\comm{Here I am actually not sure that the statement is correct. Can you double-check? according to Proposition \ref{any_k_theorem_optimality}(c) below we can do even better: find $E\subset \tmn$ such that $D\subseteq \overline{E\smallsetminus ({H}\times   \Phi)}$, which is a smaller set. Or is there a mistake somewhere? -- It is correct}
\medskip

The optimality of Theorem  \ref{ch_theorem} can be phrased in a similar way:

\begin{proposition}\label{ch_thm_optimal_new}
 Suppose that two disjoint  sets $\Theta_1, \Theta_2 \subseteq \mathbb{S}^{n-1}$ are such that condition \eqref{ch_theorem_condition} does not hold. Then for any closed linear half-space ${H}$ of $\mathbb{R}^{m}$ the collection \equ{collectionplusminus}
 does not have property {\rm (TDS$\pm$)}.
\end{proposition}

In plain terms this means the following: for any $\Theta_1, \Theta_2 \subseteq \mathbb{S}^{n-1}$   such that   \eqref{ch_theorem_condition} fails and   any closed linear half-space ${H}$ of $\mathbb{R}^{m}$ there exists a set $\fR \subseteq { \bf T}_{m,n}$ such that conditions \eqref{bigtheta1} and \eqref{bigtheta2} hold, but %the collection $\mathcal{L}_{\fR}$ does not contain a totally dense subcollection.
$\fR$ does not have property (TDS). This gives a %pretty 
{fairly} sharp converse to Theorem  \ref{ch_theorem}.

%\vfil\eject
\medskip
%\centerline{\comm{MY CHANGES STOP HERE}}
We now proceed to prove the two propositions above. The proof of Proposition \ref{any_k_theorem_optimality_new} is based on the following lemma:

\begin{lemma}\label{not_any_k}
    Fix $0 \leq k \leq \min(n-1, m)$, let $\Phi$ be a $k$-dimensional subsphere of $\mathbb{S}^{n-1}$, and  {let}  ${H}$ be {an} $(m-k)$-dimensional linear subspace of $\mathbb{R}^m$. Suppose the collection $\mathcal{L}_{\fR \smallsetminus \left( {H} \times \Phi \right)}$ is not dense. Then 
    $\fR$ does not have property {\rm (TDS)}. %the collection $\mathcal{L}_{\fR}$ does not contain a totally dense subcollection.
\end{lemma}

%\begin{proof} Let $W$ be an open set \eqref{density_in_W} in $M_{m,n}$, and let $U$ be its open subset such that 
%$$
%\bigcup\limits_{\bfL \in \mathcal{L}_{\fR \smallsetminus \left( {H} \times \Phi \right)}} \bfL \cap U = \varnothing.
%$$
%Let $\bfL_0 \in \mathcal{L}_{{H} \times \Phi}$ be intersecting $U$. Since every point in $U$ can only belong to one element of $\mathcal{L}_{{H} \times \Phi}$ according to Proposition \ref{density_any_k},
%$$\bigcup_{\br \in \fR \,:\,\bfL_{\br}\cap \bfL_0 \cap U  \ne
                          %\varnothing}\bfL_{\phi, \p} = \bfL_0,$$
%and the latter holds for any subset of $\fR$. Thus $\mathcal{L}_{\fR}$ does not contain a totally dense subcollection.
%\end{proof}

%
\begin{proof} {Recall that in Proposition \ref{density_any_k} we have constructed an open dense set $W\subset M_{m,n}$ %be an open set
as in \eqref{density_in_W}
%in $M_{m,n}$, and let $U$ be its open subset such that 
with the property that every point in $W$ can only belong to one element of $\mathcal{L}_{{H} \times \Phi}$. Since $\mathcal{L}_{\fR \smallsetminus \left( {H} \times \Phi \right)}$ is not dense, there exists nonempty open $U\subset W$ such that 
$$
\bigcup\limits_{\bfL \in \mathcal{L}_{\fR \smallsetminus \left( {H} \times \Phi \right)}} \bfL \cap U = \varnothing.
$$
Let $\bfL_0 \in \mathcal{L}_{{H} \times \Phi}$ be intersecting $U$; then it follows that $$\bigcup_{\br \in \fR \,:\,\bfL_{\phi, \p}\cap \bfL_0 \cap U  \ne
                          \varnothing}\bfL_{\phi, \p} = \bfL_0,$$
and the latter holds for any  $E'\subset \fR$ in place of $\fR$. Thus $\mathcal{L}_{\fR}$ does not contain a totally dense subcollection.}
\end{proof}

%\medskip
{\begin{proof}[Proof of Proposition \ref{any_k_theorem_optimality_new}(ab).]
    %\begin{enumerate}
     %   \item[\rm (a)] 
     For part (a), let us fix $0 \leq k \leq \min(n-2, m)$ and ${H}\times \left(\Phi_1 \cup \Phi_2 \right)$ as in Theorem \ref{any_k_theorem}(a). Fix a closed symmetric proper subset $D \subset {H}\times \left(\Phi_1 \cup \Phi_2 \right)$. {Since (TDS+) implies (TDS)}, it is enough to prove that $\mathcal{L}_{D}$ does not contain a totally dense subcollection. 
     %Indeed: if we {assume, by way of contradiction,}  that $D$ has property (TDS+), {it would follow that since $D \subseteq \overline{D} = D$, we would conclude that $D$ has property (TDS).

        There are two possibilities: either the closed symmetric set $D \cap \left( {H} \times \Phi_1 \right)$ is a proper subset of ${H} \times \Phi_1 $, or the closed symmetric set $D \cap \left( {H} \times \Phi_2 \right)$ is a proper subset of ${H} \times \Phi_2$. Without loss of generality, assume the latter; then, by Proposition \ref{density_any_k}, the collection $\mathcal{L}_{D \cap \left( {H} \times \Phi_2 \right)}$ is not dense. Since $$D \smallsetminus \left( {H} \times \Phi_1 \right) \subseteq D \cap \left( {H} \times \Phi_2 \right) ,$$
        the collection $\mathcal{L}_{D \smallsetminus \left( {H} \times \Phi_1 \right)}$ is also not dense. By Lemma \ref{not_any_k}, $\mathcal{L}_{D}$ does not contain a totally dense subcollection.
        %\item[\rm (b)] 
        Part (b) is completely analogous to part (a).
        %\end{enumerate}
\end{proof}
        
        %Before proving   
        For the proof of Proposition \ref{any_k_theorem_optimality_new}(c)  we will  need the following simple lemma:

        \begin{lemma}\label{distance_bounded_from_zero_lem}
            Fix $(\p_0, \theta_0) \in \tmn$. Let $K \subseteq M_{m,n}$ be a compact set which does not intersect $\bfL_{\p_0, \theta_0}$. Then there exists an open neighborhood $U$ of $(\p_0, \theta_0)$ in $\tmn$ such that 
            $$
            \Big( \bigcup\limits_{(\p, \theta) \in U} \bfL_{\p, \theta} \Big) \cap K = \varnothing.
            $$
        \end{lemma}
        \begin{proof}
            %Let us assume that $\p_0 \neq 0$. %Suppose $K$ is contained in the $C$-ball in $M_{m,n}$, centered at the origin.%, and let $\delta: = \dist(K, \bfL_{\p_0, \theta_0})$.
            Consider the function 
            $$
            \Delta: \,\,\, M_{m,n} \times \GL_n(\R) \rightarrow \R_{\geq 0}, \,\,\, (A, g) \mapsto \dist\left( \bfL_{\p_0, \theta_0}, \frac{1}{\|g \theta_0\|} K g - A \right).
            $$
            Since $\Delta$ is continuous and $\dist(K, \bfL_{\p_0, \theta_0})>0$, there exist an open neighborhood $U_1'$ of $0$ in $M_{m,n}$ and an open neighborhood $U_2'$ of $I_n$ in $\GL_n(\R)$ such that $\Delta(A,g) > 0$ for $(A, g) \in U_1' \times U_2'$. Now consider the sets 
            $$
            U_1: = \{ \p_0 + A \theta_0: A \in U_1' \} \,\,\,\, \text{and} \,\,\,\, U_2: = \left\{ \frac{g \theta_0}{\|g \theta_0\|}: g \in U_2'  \right\}.
            $$
            %Clearly, $U_1$ is an open neighborhood of $\p_0$ in $\R^m$ and $U_2$ is an open neighborhood of $\theta_0$ in $\mathbb{S}^{n-1}$.
            Since the maps 
            $$
            A \mapsto A \theta_0 \,\,\,\, \text{and} \,\,\,\, g \mapsto g \theta_0
            $$
            are open,   $U_1$ is open in $\R^m$ and the set $\left\{ g \theta_0: g \in U_2'  \right\}$ is open in $\R^n$. The radial projection from $\R^n$ to $\mathbb{S}^{n-1}$ is also an open map, thus $U_2$ is   open in $\mathbb{S}^{n-1}$.
            We will show that the set $U: = U_1 \times U_2$ satisfies the conclusion of the lemma. We need the following observation  that is similar to \eqref{scaling}: 
        \eq{scalingandshifting}
{\bfL_{A \q + \p,g\q} = \left( A + \bfL_{\p,\q}\right) g^{-1} \ \forall\, g \in \GL_n(\R),\ A \in M_{m,n}, \ 
        \p \in \R^m\text{  and }\q \in \R^n.}
 Then one can write
            $$
         \bigcup\limits_{(\p, \theta) \in U} \bfL_{\p, \theta} = \bigcup\limits_{(A, g) \in U_1' \times U_2'} \bfL_{A \theta_0 + \p_0, \frac{g \theta_0}{\|g \theta_0\|}} \underset{\equ{scalingandshifting}}{=} \bigcup\limits_{(A, g) \in U_1' \times U_2'} \left( A + \bfL_{\p_0, \theta_0} \right) \left( \frac{g }{\|g \theta_0\|} \right)^{-1}.
            $$
      It remains to notice that 
            $$
            \left( A + \bfL_{\p_0, \theta_0} \right) \left( \frac{g }{\|g \theta_0\|} \right)^{-1} \cap K = \varnothing \,\,\,\,\,\, \iff \,\,\,\,\,\, \bfL_{\p_0, \theta_0} \cap \frac{1}{\|g \theta_0\|} K g - A  = \varnothing,
            $$
            and the latter is guaranteed for $(A, g) \in U_1' \times U_2'$ by the definition of $U_1'$ and $U_2'$.    \end{proof}
        
       % \smallskip

\begin{proof}[Proof of Proposition \ref{any_k_theorem_optimality_new}(c).]
         Fix a closed symmetric proper subset $D$ of ${H} \times \Phi$. To prove the proposition, it is enough to construct a set $E \subset \tmn$ such that $D\subseteq \overline{E\smallsetminus D}$, but $\mathcal{L}_E$ does not contain a totally dense subcollection. We will do slightly more and construct such an $E$ satisfying $D\subseteq \overline{E\smallsetminus {H} \times \Phi}$.
         
         By Proposition \ref{density_any_k} there exists an open neighborhood $V \subseteq M_{m,n}$ such that $$\Big( \bigcup\limits_{(\p', \theta') \in D} \bfL_{\p', \theta'} \Big) \cap V = \varnothing.$$ Let $K$ be a compact subset of $V$ with nonempty interior. By Lemma \ref{distance_bounded_from_zero_lem}, any $(\p', \theta') \in D$ has an open neighborhood $U(\p', \theta')$ in $\tmn$ such that $\Big( \bigcup\limits_{(\p, \theta) \in U(\p', \theta')} \bfL_{\p, \theta} \Big) \cap K = \varnothing.$ Then define $$E: = \bigcup\limits_{(\p', \theta') \in D} U(\p', \theta').$$ We claim that the above set %$E$ 
         satisfies the two desired conditions. 

        Indeed: since a nonempty open set in $\tmn$ cannot be contained in $H \times \Phi$, every neighborhood of any point $(\p', \theta') \in D$ contains a point from $E \smallsetminus \left( H \times \Phi \right)$, thus $D \subseteq \overline{E \smallsetminus \left( H \times \Phi \right)}$. At the same time, 
        $$
        \Big( \bigcup\limits_{(\p, \theta) \in E} \bfL_{\p, \theta} \Big) \cap K = \Big( \bigcup\limits_{(\p', \theta') \in D} \bigcup\limits_{(\p, \theta) \in U(\p, \theta)} \bfL_{\p, \theta} \Big) \cap K = \bigcup\limits_{(\p', \theta') \in D} \Big( \bigcup\limits_{(\p, \theta) \in U(\p, \theta)} \bfL_{\p, \theta}  \cap K \Big) = \varnothing.
        $$
        Since $K$ has nonempty interior, the collection $\mathcal{L}_{\fR}$ is not dense and does not contain a totally dense subcollection.
\end{proof}}

%\comm{Sorry, this seems to be the end of the discussion of optimality of Theorem \ref{any_k_theorem}, but I have no idea how it ended. In what sense the theorem is optimal? And is it really? suppose there is no $k$ such that one of \eqref{anykA}, \eqref{anykB} and \eqref{anykC} holds. Does it mean that there is no totally dense subcollection? probably it doesn't. Then what have you actually proved? need to give some conclusions here. It is written much better for  Theorem 
%\ref{ch_theorem} below.}
\medskip

%Now we can prove Lemma \ref{ch_not_tot_dense}.

%In this section 
Our next goal is to prove Proposition \ref{ch_thm_optimal_new}.
%show that the restrictions on $P$ and $Q$ %introduced in %both 
%{given in Theorem %\ref{any_k_theorem} and %Theorem 
%\ref{ch_theorem}} cannot be simultaneously relaxed and still guarantee total density. 
{The next lemma is an intermediate step in this direction.} %We start with Theorem \ref{ch_theorem}.

\begin{lemma}\label{ch_not_tot_dense}
    Let ${H}$ be a half-space of $\mathbb{R}^m$ and suppose that $\fR \subseteq \mathbb{S}^{n-1} \times {H}$ (we can always assume this condition, since one of any two opposite elements of $\tmn$ belongs to this set). Suppose that there exists $\theta_0 \in \mathbb{S}^{n-1}$ such that
\begin{equation}\label{nonch_theorem_condition}
   \text{the collection } \, \mathcal{L}_{\fR \smallsetminus \left( {H}\times \{ \theta_0 \}  \right)} \, \text{is not dense.}
    \end{equation}
    Then $\fR$ does not have property {\rm (TDS)}. %$\mathcal{L}_{\fR}$ does not contain a totally dense subcollection.
\end{lemma}

\begin{proof}%[Proof of Lemma \ref{ch_not_tot_dense}.]
If the collection $\mathcal{L}_{\fR}$ is not dense, it automatically does not contain a totally dense subcollection; so from now on assume 
%this collection 
{$\mathcal{L}_{\fR}$} is dense and contains a totally dense subcollection {$\mathcal{L}$}. Take $\theta_0$ satisfying \eqref{nonch_theorem_condition}; then there exists a nonempty open  $W \subseteq M_{m,n}$ such that $$W \cap \bigcup\limits_{(\p, \theta) \in \fR \smallsetminus \left( {H}\times \{ \theta_0 \}  \right)} \bfL_{\p, \theta} = \varnothing.$$ Fix any subspace from $\mathcal{L}_{\fR}$ intersecting $W$; %{and different from $\bfL_{\theta_0, \p}$}
 it has to be of the form $\bfL_{\p, \theta_0}$. %\comm{(I am confused here, what is $\p$ here and are the two  $\p$s the sane  or different?)} 
Suppose $\bfL_{\p, \theta_0} \in \mathcal{L}$ is a subspace intersecting $W$. Since all subspaces of %such
{the} form $\bfL_{\p, \theta_0}$ are parallel, {it follows that}
$$
\bigcup\limits_{\bfL \in \mathcal{L}: \,\,\, \bfL \cap \bfL_{\p, \theta_0} \cap W{\ne  \varnothing}} \bfL = \varnothing.
$$
It implies that no subspaces intersecting $W$ {other than $\bfL_{\p, \theta_0}$} can belong to %a totally dense subcollection
{$\mathcal{L}$}, which means that %totally dense subcollection
{$\mathcal{L}$}  is not dense; this contradiction completes the proof.
\end{proof}

%\vskip+0.3cm

%Lemma \ref{ch_not_tot_dense} alone does not mean that the Khintchine-Jarnik result fails to hold. \comm{Do we want some nearly trivial example here? I want to add some sort of condition of P and Q being "discrete enough" and prove that then Kh-J theorem fails - need to find a good formulation; it definitely works for subsets of integers.} 

%\vskip+0.3cm
%\begin{proposition}\label{ch_thm_optimal}
% Fix two disjoint {(possibly empty)} sets $\Theta_1, \Theta_2 \subseteq \mathbb{S}^{n-1}$ such that condition \eqref{ch_theorem_condition} does not hold. It is enough to show that for any closed linear half-space ${H}$ of $\mathbb{R}^{m}$ there exists a set $\fR \subseteq { \bf T}_{m,n}$ such that conditions \eqref{bigtheta1} and \eqref{bigtheta2} hold, but the collection $\mathcal{L}_{\fR}$ does not contain a totally dense subcollection.
%\end{proposition}

\begin{proof}[Proof of Proposition \ref{ch_thm_optimal_new}.]

 Fix two disjoint (possibly empty) sets $\Theta_1, \Theta_2 \subseteq \mathbb{S}^{n-1}$ such that condition \eqref{ch_theorem_condition} does not hold and a closed linear half-space ${H}$ of $\mathbb{R}^{m}$. It is enough to show that there exists a set $\fR \subseteq { \bf T}_{m,n}$ such that conditions \eqref{bigtheta1} and \eqref{bigtheta2} hold, but the collection $\mathcal{L}_{\fR}$ does not contain a totally dense subcollection.

    First assume that $\Theta_1$ is nonempty. {Since \eqref{ch_theorem_condition} fails, there exists} ${\theta_0} \in \Theta_1$ such that $$0 \notin \conv \big( \overline{\Theta_1 \cup \Theta_2 \smallsetminus \{ {\theta_0}\}} \big).$$ Let $\alpha \in \mathbb{S}^{n-1}$ and $\delta > 0$ be such that $\alpha^{\top} \theta > \delta$ for any $\theta \in \overline{\Theta_1 \cup \Theta_2 \smallsetminus \{ {\theta_0} \}}$,
    {and in particular for any $\theta \in   \Theta_2$}; such $\alpha$ and $\delta$ exist due to the Separating Hyperplane Theorem. 

    {Now for any $\theta \in \Theta_2$ choose} a sequence of elements $\{ \theta_i(\theta)\}_{i=1}^{\infty}$ converging to $\theta$ such that $$\alpha^{\top} \theta_i(\theta) > \delta\text{ and }\theta_i(\theta) \neq \theta,$$ and define
    $$
    \Theta_2' := \bigcup\limits_{\theta \in \Theta_2} \{ \theta_i(\theta)\}_{i=1}^{\infty} \,\,\,\,\,\,\,\, \text{and} \,\,\,\,\,\,\,\, \fR := {H}\times\left( \Theta_1 \cup \Theta_2' \right).
    $$
    By construction {${H}\times\Theta_1  \subseteq {\fR}$, hence \eqref{bigtheta1}  holds; and,  since any $\theta\in\Theta_2$ belongs to the closure of $\Theta_2'\smallsetminus \{\theta\}$, it follows that \eqref{bigtheta2} is satisfied as well.}
    %$\Theta_2 \subseteq \overline{\Theta_2'}$, one has ${H}\times\Theta_2  \subseteq \overline{\fR}$.
    %Thus conditions \eqref{bigtheta1} and \eqref{bigtheta2} hold. 
    {On the other hand,} since $\alpha^{\top} \theta \geq \delta$ for any $\theta \in \overline{\Theta_1 \cup \Theta_2' \smallsetminus \{ {\theta_0}\}}$, one has $0 \notin \conv \big( \overline{\Theta_1 \cup \Theta_2' \smallsetminus \{ {\theta_0}\}} \big)$. By Proposition \ref{ditect product dense} the collection $\mathcal{L}_{\fR \smallsetminus \left({H}\times \{ \theta_0\}  \right)}$ is not dense. Hence, by Lemma \ref{ch_not_tot_dense}, $\mathcal{L}_{\fR}$ does not contain a totally dense subcollection.

    In the case when $\Theta_1 = \varnothing$ the proof follows the same lines. We define $\alpha \in \mathbb{S}^{n-1}$ and $\delta > 0$ by the condition $\alpha^{\top} \theta > \delta$ for any $\theta \in \overline{\Theta_2}$, construct $\Theta_2'$ as above and define $\fR : = {H}\times\Theta_2 $. By Proposition \ref{ditect product dense} the collection $\mathcal{L}_{\fR}$ is not dense, and thus does not contain a totally dense subcollection.
\end{proof}

%\bigskip

\section{Logarithmically dense sets}\label{logdense_sets_section}

\subsection{Logarithmic density explained}\label{logdense_expl} In this section we collect several equivalent characterizations of logarithmically dense sets and discuss some interesting   examples. Recall that $P\subseteq \R^m$ was defined to be  logarithmically dense along $\phi \in \mathbb{S}^{m-1}$ if for any $\varepsilon > 0$  the set
%\eq{pepsilon}{
$
%Q_{\varepsilon}^{\theta}: = \left\{ \q \in Q: \,\,\, \left\| \frac{\q}{\| \q \|} - \theta \right\| < \varepsilon \right\} \,\,\,\,\,\,\,\,\,\,\,\, \text{and} \,\,\,\,\,\,\,\,\,\,\,\, 
P_{\varepsilon}^{\phi}  = 
\left\{    \p \in P :  \left\| \frac{\p}{\| \p \|} - \phi\right\| < \varepsilon \right\} 
$
contains an
unbounded sequence $\{ \p_k \}_{k=1}^{\infty}$ %of positive real numbers 
with \eq{lim=1}{\lim\limits_{k \rightarrow \infty} \frac{\|\p_{k+1}\|}{\|\p_k\|}=1.} 
%But this may not be the best way to describe  logarithmic density. \comm{Vasya -- this phrase raises a natural question: why don't we define it in a different way then? What we probably want to say instead is that sometimes alternative definitions/descriptions are more convenient to work with.} 
The next lemma offers several {alternative definitions}. 

\begin{lemma}\label{LDequiv}
For $P\subseteq \R^m$ and $\phi \in \mathbb{S}^{m-1}$, the following are equivalent.
\begin{itemize}
\item[\rm (i)] $P$ is  logarithmically dense along $\phi $.
\item[\rm (ii)] For any $\varepsilon,\delta > 0$ there exists $C' = C'(\varepsilon,\delta)$ such that $$C \ge C' \Longrightarrow 
\left\{  \p \in P: \ \left\| \frac{\p}{\| \p \|} - \phi\right\| < \varepsilon ,\ C < \|\p\| < (1+\delta)C\right\} \ne\varnothing.$$
\item[\rm (iii)] For any $\varepsilon  > 0$ there exists $C'' = C''(\varepsilon)$ such that $$C \ge C'' \Longrightarrow 
\left\{  \p \in P: \ \left\| \frac{\p}{\| \p \|} - \phi\right\| < \varepsilon ,\ C < \|\p\| < (1+\varepsilon)C\right\} \ne\varnothing.$$
\item[\rm (iv)] For any $\varepsilon  > 0$ there exists $C^\circ = C^\circ(\varepsilon)$ such that $$C \ge C^\circ \Longrightarrow 
\big\{  \p \in P: \ \left\|  \p - C\phi\right\| < C\varepsilon  \big\} \ne\varnothing.$$
\item[\rm (v)] $P+\vb$ is  logarithmically dense along $\phi $ for any $\vb\in\R^m$.
\end{itemize}
\end{lemma}

\begin{proof} Assuming (i), take an
unbounded sequence $\{ \p_k \}\subseteq P_{\varepsilon}^{\phi}$ %of positive real numbers 
with \equ{lim=1} and choose $k_0$ such that $\left\|\frac{\|\p_{k+1}\|}{\|\p_k\|} - 1\right\| < (1+\delta)$ for all $k\ge k_0$; then $C'(\varepsilon,\delta) = \|\p_{k_0}\|$ will satisfy (ii). The implication (ii)$\Longrightarrow$(iii) follows by setting $C''(\varepsilon) = C'(\varepsilon,\varepsilon)$, and to prove  (iii)$\Longrightarrow$(iv) it  suffices to take $C^\circ(\varepsilon) = C''\big(\min(\varepsilon/3,1)\big)$: indeed, if  $\p \in P$ is such that $\left\|\frac{\p}{\|\p\| } - \phi\right\| < \frac{\varepsilon}{3}$ and $ C < \|\p\| < \left(1 + \frac{\varepsilon}{3}\right)C,$
then $$
\begin{aligned}\|\p - C\phi\| &\leq \big\|\p - \|\p\| \phi\big\| + (\|\p\| - C)\|\phi\| = \|\p\| \left\|\frac{\p}{\|\p\| } - \phi\right\| + (\|\p\| - C)\\ &< \left(1+\frac{\varepsilon}{3}\right)C \cdot \frac{\varepsilon}{3} + \frac{\varepsilon}{3}C = \frac{\varepsilon}{3}C\left(2 + \frac{\varepsilon}{3}\right) \le \frac{\varepsilon}{3}C \cdot 3 = \varepsilon C,\end{aligned}$$
where in the last step we used $\varepsilon/3 \leq 1$.

 Assuming (iv), %take $C = C_0 (\varepsilon)$ and 
consider  balls centered at $(1+2\varepsilon)^{2k}C^\circ(\varepsilon)\phi$ with radius $(1+2\varepsilon)^{2k}C^\circ(\varepsilon)\varepsilon$ for $k = 1,\dots,k^*$, where $$k^* := \max\left\{k: (1+2\varepsilon)^{2k}C^\circ(\varepsilon) < C^\circ(\varepsilon/2)\right\}.$$ They are all disjoint as long as $(1+\varepsilon) <  (1-\varepsilon)(1+2\varepsilon)^2$, which is equivalent to $\varepsilon < 1/\sqrt{2}$. By (iv) each of them contains a point of $P$ which we can call $\p_k$, and it follows that 
$$\frac{\|\p_{k+1}\|}{\|\p_k\|}\le \frac{(1+2\varepsilon)^2(1+\varepsilon)}{(1-\varepsilon)}.$$ Repeating this process with $\varepsilon$ replaced by  $\varepsilon/2$,  $\varepsilon/4$ etc.\ and arranging all the points in an infinite sequence $\{\p_k\}_{k=1}^{\infty}$, one gets \equ{lim=1}, and (i) follows.

Finally, if $\{ \p_k \}\subseteq P_{\varepsilon/2}^{\phi}$ %of positive real numbers 
is an unbounded sequence satisfying \equ{lim=1}, then one has $\p_k +\vb \in P_{\varepsilon}^{\phi}$ if $k$ is {sufficiently large},
%large enough, 
and, since
$$
\frac{\|\p_{k+1}\|-\|\vb\|}{\|\p_k\|+\|\vb\|} \le \frac{\|\p_{k+1}+\vb\|}{\|\p_k+\vb\|}\le \frac{\|\p_{k+1}\|+\|\vb\|}{\|\p_k\|-\|\vb\|},
$$
it follows that $\lim\limits_{k \rightarrow \infty} \frac{\|\p_{k+1}+\vb\|}{\|\p_k+\vb\|}=1$, proving the equivalence of (i) and (v).
\end{proof}
It is also worth pointing out that the original definition of logarithmic density, as well as all the above equivalent conditions, are independent on the choice of the norm on $\R^m$. 

\smallskip

To construct  logarithmically dense sets one can start with the case $m=1$ and $\phi = 1$. Examples include the set $\mathbb P$ of prime numbers, as well as any nontrivial intersection of $\mathbb P$ with any residue class (this follows from the Prime Number Theorem along arithmetic progressions). Another example is given by the set of natural numbers with two given prime factors, or, more generally, numbers of the form $a^kb^l$ where $a$ and $b$ are multiplicatively independent. Finally let us observe that if $P\subseteq \R_{\geq 0}$ is logarithmically dense in the positive direction, then so is $\{f(x) : x\in P\}$ for any polynomial  $f$ with a positive leading coefficient.

\smallskip

In order to exhibit {higher-dimensional} examples, we will show that products of   logarithmically dense sets are logarithmically dense. More precisely, we have

\begin{proposition}\label{prod_of_log_dense}
    Let ${H}_i \subseteq \mathbb{R}^{m_i}, \,\, i = 1, 2$, be two cones. If $P_1$ is logarithmically dense in ${H}_1$ and $P_2$ is logarithmically dense in ${H}_2$, then $P_1 \times P_2$ is logarithmically dense in ${H}_1 \times {H}_2 \subseteq \R^{m_1+m_2}$.
\end{proposition}

\begin{proof} It is enough to show the result for ${H}_i = \R_{\geq 0}\xi_i$, where  $\xi_i$ are unit vectors in  $\mathbb{R}^{m_i}$. This reduces to   the following claim: if $P_i$ is logarithmically dense along  $\xi_i$, then $P_1 \times P_2$ is logarithmically dense along $\phi = (\alpha_1\xi_1,\alpha_2\xi_2)\in \mathbb{S}^{m_1+m_2-1}$ for any $\alpha_1,\alpha_2 \ge 0$ with $\alpha_1^2 + \alpha_2^2 = 1$.

Throughout the proof we are going to use %the Euclidean norm on all the vector spaces and
 the equivalence (i) $\Longleftrightarrow$ (iv) of Lemma \ref{LDequiv}. Fix $\varepsilon > 0$. We want to find $C^\circ$ such that  for any $C > C^\circ$ there exists $(\bx_1,\bx_2) \in P_1 \times P_2$ with  
 $$\|(\bx_1,\bx_2) - C\phi\| < C\varepsilon\quad\Longleftrightarrow\quad \|\bx_1 -  C\alpha_1\xi_1\|^2 + \|\bx_2 -  C\alpha_2\xi_2\|^2 < \varepsilon^2C^2.$$
In fact it suffices to have   $\|\bx_i -  C\alpha_i\xi_i\| <  \alpha_i{\varepsilon C}$ for $i=1,2$.
Assuming $\alpha_i > 0$, the   logarithmic density of $P_i$  along $\xi_i$ gives a function $C_i^\circ (\cdot)$ such that $$C \ge C_i^\circ (\varepsilon
%/\sqrt{2}
)/\alpha_i \Longrightarrow 
\big\{  \bx_i \in P_i: \ \left\|  \bx_i - (\alpha_i C)\xi_i\right\| < (\alpha_i C)\varepsilon  \big\} \ne\varnothing.$$
%Apply the  condition (iv) for $P_1$ along $\eta$ with parameter $\varepsilon/\sqrt{2}$: for $C > C''{\hat\phi}(\varepsilon/\sqrt{2})$, there exists $x \in P$ with $$\|x - \alpha C \hat\phi\| < \alpha C \cdot \frac{\varepsilon}{\sqrt{2}}\leq \frac{C\varepsilon}{\sqrt{2}}. $$ (We applied the condition at the value $\alpha C$ in place of $C$, which is valid for $\alpha C > C''{\hat\phi}(\varepsilon/\sqrt{2})$, i.e., $C > C''_{\hat\phi}(\varepsilon/\sqrt{2})/\alpha$.)
So as long as both $\alpha_1$ and $\alpha_2$ are nonzero, for all $C > C^\circ(\varepsilon) := \max_{i=1,2}\frac{C_i^\circ (\varepsilon
%/\sqrt{2}
)}{\alpha_i}$ we get $(\bx_1,\bx_2) \in P_1 \times P_2$ with $\|(\bx_1,\bx_2) - C\phi\| < C\varepsilon$. 
\smallskip

It remains to consider the case $\alpha_1 = 0$  (the case $\alpha_2=0$ is symmetric).
Then $\phi = (0,\xi)$, and we need to achieve  \eq{achieve}{\big\|(\bx_1,\bx_2)- C\phi\big\| = \big\|(\bx_1, \bx_2- C\xi)\big\| < C\varepsilon\quad\Longleftrightarrow\quad \|\bx_1\|^2 + \|\bx_2 - C\xi\|^2 < \varepsilon^2C^2.}
Using the   logarithmic density of $P_2$  along $\xi$, we get a function  $C_2^\circ (\cdot)$ such that $$C \ge C_2^\circ (\varepsilon
/\sqrt{2}
)  \Longrightarrow 
\big\{  \bx_2 \in P_2: \ \left\|  \bx_2 - C\xi\right\| < C\varepsilon/\sqrt{2}  \big\} \ne\varnothing.$$
Also let us fix any $\bx_1 \in P_1$; then we would have 
  $\|\bx_1\| < C\varepsilon/\sqrt{2}$ as long as $C > \|\bx_1\| \sqrt{2}/\varepsilon$. Thus the choice $C^\circ(\varepsilon) := \max\left(\|\bx_1\| \sqrt{2}/\varepsilon, 
  C_2^\circ (\varepsilon
/\sqrt{2}
) \right)$ yields \equ{achieve}, finishing the proof.
\end{proof}

\subsection{On the optimality of {logarithmic density  in Theorem \ref{not_on_line}}}\label{logdense_optimal_section}

In {this subsection} we will explain why logarithmic density is an essential condition in Theorem \ref{not_on_line}. Let us start with a special case $m = 1$ and $k = 0$ which highlights the {main} idea:

\begin{proposition}
{Let $n\ge 2$ and} suppose $P \subseteq \R$ is a set {that} is not logarithmically dense. Then there exists $Q \subseteq \R^n$   {satisfying} \equ{nonprop} such that the collection $\mathcal{L}_{P, Q}$ is not dense ({hence} not totally dense).
\end{proposition}

\begin{proof}
    Without loss of generality we can assume that $P$ is not logarithmically dense in $\R_{\geq 0}$. Then there exists $\delta > 0$ and two increasing sequences $\{ b_i \}, \{ c_i \}$ of {positive} real numbers {such that} 
    \begin{equation} \frac{c_i}{b_i} > \delta \,\,\,\,\,\,\,\, \text{and} \,\,\,\,\,\,\,\, P \cap (b_i, c_i) = \varnothing \,\, \text{for any} \,\, i.
    \end{equation}
    {Also, it is clearly enough to construct such an example for $n=2$   (when   $n>2$ we can take the product of all the lines in the collection with $\R^{n-2}$).  Let $\be_1 = (1,0)$, $\be_2 = (0,1)$ be the standard basis of $\R^2$, and take   $Q := \{ b_i \be_j: i \in \mathbb{N},\,j=1,2 \}$; 
    %Q = \{ b_i \cdot (1,0, \ldots, 0)^{\top}, \, b_i \cdot (0,1, 0,\ldots, 0)^{\top}, \,\, i \in \mathbb{N} \}$;
    clearly %such a $Q$ 
    {it} satisfies \equ{nonprop}. Let $A = (a_1, a_2)$; we will show that
\begin{equation}\label{m=1_not_dense}
        \text{if} \,\,\, 1 < a_1 < \delta \, \text{and} \, 1 < a_2 < \delta, \, \text{then} \, A \notin \bigcup\limits_{p \in P, \, \q \in Q} \bfL_{p, \q}.
    \end{equation}
    Suppose 
\begin{equation}\label{contradiction_to this}
    \text{$1 < a_1 < \delta, \, 1 < a_2 < \delta$ and $p \in P, \, \q \in Q$ are such that $A\q = p$.} 
    \end{equation} By reordering if necessary, we may assume that $\q = b_i\be_1$; then \eqref{contradiction_to this} implies that 
    %Let $x = \frac{p}{b_i}$; %the 
   % condition 
    $a_1 = A \be_1
    %\cdot (1,0, \ldots, 0)^{\top} 
    =  \frac{p}{b_i}$. There are two possible cases:
    \begin{itemize}
        \item If $p < b_i$, then $a_1   < 1$, which contradicts \eqref{contradiction_to this}.
        \item If $p > c_i > \delta b_i$, then $a_1   > \delta$, which again contradicts \eqref{contradiction_to this}.
    \end{itemize}}
    This contradiction shows that \eqref{m=1_not_dense} holds, and the collection $\mathcal{L}_{P,Q}$ is not dense.
\end{proof}

In {greater} generality, we will prove the following fact:

\begin{proposition}\label{logdense_necessary}
    Fix $0 \leq k \leq \min(n-1, m-1)$, and let ${H}$ be an $(m-k)$-dimensional linear subspace {of} $\mathbb{R}^m$. If $P \subseteq {H}$ is not logarithmically dense in ${H}$, there exists $Q \subseteq \R^n$ satisfying \equ{twohalfspheres} such that the collection $\mathcal{L}_{P, Q}$ is not dense. In particular, \equ{conclusion1} does not hold, and $\mathcal{L}_{P,Q}$ does not contain a totally dense subcollection.
\end{proposition}

%\comm{The condition that $P \subseteq {H}$ is somewhat restrictive, since we remove the cases when $P$ approximates some ${H}$ but is not inside it from consideration. I believe a stronger statement holds: we can assume that $P \subseteq {H}_{\varepsilon}$ and $|P|_{\varepsilon}^{ \phi}$ is not logatighmically dense for some $\varepsilon > 0$ and $\phi \in {H} \cap \mathbb{S}^{m-1}$. Here ${H}_{\varepsilon}$ denotes an $\varepsilon$-neighborhood (in terms of angle) of ${H}$). To do this, we may need to control how exceptional algebraic sets from Proposition \ref{density_any_k} continuously change under rotation of ${H}$ in directions, corresponding to taking ${H}_{\varepsilon}$. DK: OK, so what would you like to do with this remark? 
%-- At this point we probably can ignore it} 

\begin{proof}
    Since $P$ is not logarithmically dense in ${H}$, there exist
    %\begin{equation}
    $$
        \phi \in {H} \cap \mathbb{S}^{m-1}, \, \varepsilon > 0, \,  \delta > 1 \,\text{and two increasing sequences $\{ b_i \}, \{ c_i \}$ of real numbers}
    $$
    {such that} 
    \begin{equation} \frac{c_i}{b_i} > \delta \,\,\,\, \text{and} \, \,\,\,P_{\varepsilon}^{\phi} \cap \big( B_{c_i}(0) \smallsetminus B_{b_i}(0) \big) = \varnothing \,\,\,\, \text{for any} \, i.
    \end{equation}
    %\end{equation} 
    Consider two $(k+1)$-dimensional halfspheres
    {$$
    \Phi_1 = \big\{(\theta_1, \ldots, \theta_{k+1}, 0, \ldots, 0): \theta_1 > 0\big\} %, \, \max\limits_{i=1, \ldots, k+1} |\theta_i| = 1 
    %$$
    \text{ and }
    %$$
    \Phi_2 = \big\{(0, \theta_2, \ldots, \theta_{k+2}, 0, \ldots, 0): \theta_{k+2} > 0
    %, \, \max\limits_{i=2, \ldots, k+2} |\theta_i| = 1
    \big \}
    $$
   of  $ \mathbb{S}^{m-1}$}. Let $B = {\{ b_i \}_{i=1}^{\infty}}$, and let $Q := B \cdot \left( \Phi_1 \cup \Phi_2 \right)$. We want to construct an open set $U \subseteq M_{m,n}$ such that   \begin{equation}\label{noting_in_U}
        \bigcup\limits_{\bfL \in \mathcal{L}_{P,Q}} \bfL \cap U = \varnothing,
    \end{equation}
    proving that $\mathcal{L}_{P, Q}$ is not dense.

    Denote an open halfsphere $\{ (\theta_1, \theta_2, \ldots, \theta_{k+1}): \theta_1 > 0 \}$ of $\mathbb{S}^k$ by $\mathbb{S}_+^k$. Let ${\theta} \in \mathbb{S}_+^k, \,\, \bx \in \R^m$, and consider the subspaces 
    $$
    \bfL_{\bx, {\theta}}  = \{ {A} \in M_{m,k+1}:  {A} {\theta} = \bx \}.
    $$
    Let us also define the set
    $$
{H}_{\varepsilon}^{\phi}(\delta): = \left\{ \bx \in {H}: 1 < \|\bx\| <\delta,\, \left\| \frac{\bx}{\|\bx\|} - \phi \right\| < \varepsilon \right\}.
    $$
    \begin{observation}\label{nonempty_int}
        The set
        $$
        V: = \bigcup\limits_{{\theta} \in \mathbb{S}^{k}, \, \bx \in {H}_{\varepsilon}^{\phi}(\delta)} \bfL_{\bx, {\theta}}
        $$
        has %a 
        nonempty interior.
    \end{observation}

{Indeed: since the set ${H}_{\varepsilon}^{\phi}(\delta)$ is nonempty and open in $H$, and $\mathbb{S}^k$ is open in itself, {$V$} has nonempty interior by Lemma \ref{opensetlemma}.}
    %\begin{proof}
     %   Define %the set
     %   $$
     %   {W} = \{{A} \in M_{m,k+1}: \,\, \text{there exists a unique up to sign pair} \, (\bx, {\theta}) \in {H} \times \Phi \,\, \text{such that} \,\, {A} {\theta} = \bx \}.$$  By Proposition \ref{density_any_k} and Remark \ref{continuous_on_halfsphere}, this set is open and dense in $M_{m,k+1}$, and, moreover, the functions        $$ {\theta}: \,\, {W} \rightarrow \mathbb{S}_+^k \,\,\,\, \text{and} \,\,\,\, \bx: \,\,{W} \rightarrow {H}: \,\,\,\, A' {\theta}({A}) = \bx({A})       $$   are continuous. Since the set ${H}_{\varepsilon}^{\phi}(\delta)$ is open, the set $$\{ {A} \in {W}: \bx({A}) \in {H}_{\varepsilon}^{\phi}(\delta)  \} \subseteq V$$ is also open.
    %\end{proof}

Now let us choose an open set $V^0 \times V' \subseteq V$, where $V^0$ is %an 
open %set 
in $\R^m$ and $V'$ is %an 
open %set 
in $M_{m,k}$; here we identify elements of $\R^m$ with the first column of $m \times (k+1)$-matrices and elements of $M_{m,k}$ with submatrices  consisting of columns $2$ to $k+1$. Since $\Phi_1 = \mathbb{S}_+^k \times \{ 0 \}^{m-k-1}$, Observation \ref{nonempty_int} implies:
\begin{itemize}
    \item for any $A \in V^0 \times V' \times M_{m, n-k-1}$ there exists a unique $\big(\theta^1(A), \bx^1(A)\big) \in \Phi_1 \times {H}$ such that $A \theta^1(A) = \bx^1(A)$;
    \item $\bx^1(A) \in {H}_{\varepsilon}^{\phi}(\delta)$.
\end{itemize}

We note that $\Phi_2$ can be obtained from $\Phi_1$ by interchanging first and $k+2$-nd coordinates in $m$-dimensional vectors $\theta$. Thus, a similar statement holds for $\Phi_2$:

\begin{itemize}
    \item for any $A \in \R^m \times V' \times V^0 \times M_{m, n-k-2}$ there exists a unique $\big(\theta^2(A), \bx^2(A)\big) \in \Phi_2 \times {H}$ such that $A \theta^2(A) = \bx^2(A)$;
    \item $\bx^2(A) \in {H}_{\varepsilon}^{\phi}(\delta)$.
\end{itemize}

Combining the two statements above, we conclude:

\begin{observation}\label{double_uniqueness}
    Let 
    $$
    U: = V^0 \times V' \times V^0 \times M_{m, n-k-2} \subseteq \R^m \times M_{m,k} \times \R^m \times M_{m,n-k-2} = M_{m,n}.
    $$
    For any $A \in U$ there exists a unique pair
    $$
    (\theta^1, \bx^1) \in \Phi_1 \times {H} \,\,\,\, \text{such that} \,\,\,\, A \theta^1 = \bx^1
    $$
    and a unique pair
    $$
    (\theta^2, \bx^2) \in \Phi_2 \times {H} \,\,\,\, \text{such that} \,\,\,\, A \theta^2 = \bx^2,
    $$
    and in addition $\bx^1, \bx^2 \in {H}_{\varepsilon}^{\phi}(\delta)$.
    In particular: if $A \theta = \bx$ for $\theta \in \Phi_1 \cup \Phi_2$ and $\bx \in {H}$, then $\bx \in {H}_{\varepsilon}^{\phi}(\delta)$.
\end{observation}

%\medskip

It remains to show how Observation \ref{double_uniqueness} implies \eqref{noting_in_U}. Suppose there exist   $(\p,\q) \in P\times Q$ and $A \in U$  such that $A \in \bfL_{\p, \q}$. %There exists $i$ such that 
{We know that $\q = b_i \theta$ %where
for some $i\in\N$ and} $\theta \in \Phi_1 \cup \Phi_2$. Let $\bx  : = \frac{\p}{b_i} \in {H}$; then $A \theta^1 = \bx$. By Observation \ref{double_uniqueness} it implies that $\bx \in {H}_{\varepsilon}^{\phi}(\delta)$. %However,
{Now} there are two possibilities:

\begin{itemize}
    \item $\|\p\| < b_i$. In this case $\|\bx\| < 1$, and thus $\bx \notin {H}_{\varepsilon}^{\phi}(\delta)$;
    \item $\|\p\| > c_i > \delta b_i$. In this case $\|\bx\| > \delta$, and again $\bx \notin {H}_{\varepsilon}^{\phi}(\delta)$.
\end{itemize}
This contradiction shows that \eqref{noting_in_U} holds and completes the proof.
\end{proof}

\section{{Totally dense collections and subcollections}}\label{subcollections_section}
%Proof of Theorem \ref{8.1}}
%\label{proof8.1}

\begin{comment}
    A statement very similar to Theorem \ref{8.1} directly follows from the main results of our paper: namely, that these collections $\mathcal{L}_{P, Q}$ contain totally dense subcollections.  

To see how this modified version follows, one needs to notice that:
\begin{enumerate}
    \item We show that the conditions of Theorem \ref{8.1} (a) imply the condition \ref{anykA} of Theorem \ref{any_k_theorem} with $k = \min(n-2, m)$. Let $F$ be the span of $P$ over $\mathbb{R}$; it is a $>m-n+1$-dimensional linear subspace of $\mathbb{R}^m$, thus it contains a $\max(m-n+2, 0) = m-k$-dimensional subspace ${H}$. 

    If we define $\fR = \pr(P \times Q)$, one can notice that $\mathbb{S}^{n-1} \times \Pi \subseteq \overline{\fR}$, and since the inclusion $\left(\Phi_1 \cup \Phi_2 \right) \cap \mathbb{S}^{n-1} \subseteq \mathbb{S}^{n-1}$ trivially holds for any two distinct $k+1 = n-1$-dimensional linear subspaces of $\mathbb{R}^n$, condition \ref{anykA} of Theorem \ref{any_k_theorem} holds.
    \item In turn, the conditions of Theorem \ref{8.1} (b) are a special case of conditions of Theorem \ref{not_on_line}: such a direct product set will always have two non-proportional arbitrary long vectors, and any set with bounded Hausdorff distance from $\mathbb{R}^m$ is exponentially dense.
\end{enumerate} 

\vskip+0.3cm

\end{comment}
\subsection{Proof of Theorem \ref{8.1}}\label{pf8.1}
{As was explained in \S\ref{pf1.6}, our main results imply a modified (and weaker) version of Theorem~\ref{8.1}: namely, that under its assumptions   the  collection $\mathcal{L}_{P, Q}$ contains a totally dense subcollection. Such a conclusion has the same (known to us) applications as the total density itself. Here, for completeness, we prove Theorem \ref{8.1}} in its original form %. We will
{by showing how to deduce it} from Lemma \ref{usefullemma}. 

\begin{proof}[Proof of Theorem \ref{8.1}]
We will use the notation $\fR = \pr(P \times Q)$. Fix $(\q_0, \p_0) \in P \times Q$ and let $\br_0 = \left( \frac{\q_0}{\|\q_0\|}, \frac{\p_0}{\|\q_0\|}  \right) \in \fR$.

Recall that in part (a) we {let
   $P$ be a subgroup of $\Z^m$ of rank $>m-n+1$, and took $Q = \Z^n\nz$.}   Let ${H} = \Span_{\mathbb{R}}(P)$; it  is a  {subspace of $\mathbb{R}^m$ of dimension greater than $m-n+1$}. %It is easy to show 
 {We claim} that 
        \begin{equation}\label{81A_closure}
            {\mathbb S}^{n-1} \times {H} \subseteq \overline{\fR}.
        \end{equation} Indeed, for any $\theta \in {\mathbb S}^{n-1}$ and any $\varepsilon > 0$ we can find %an element 
        $\q \in %Q = 
        \mathbb{Z}^n \smallsetminus \{ 0 \}$ such that $\left\|\theta - \frac{\q}{\|\q\|} \right\| < \varepsilon$. Next for any $\bx \in {H}$ %we can 
       {find %an integer number 
        $k\in\Z$} and $\p \in P$ such that $\left\| \bx - \frac{\p}{k \q} \right\| < \varepsilon$, and thus
        $$
        \dist \left( \pr\big({(\p,k \q)\big), (\bx,\theta)} \right) < \varepsilon.
        $$
        Let $\Phi$ be {an %linear 
        $(n-1)$-subsphere}
        %dimensional subspace 
        of {$\mathbb{S}^{n-1}$ such that %which does not contain %$\q_0$
       % {$ \frac{\q_0}{\|\q_0\|}$}, and let 
    \begin{equation}\label{dist_to_f0}
       \dist\left( \frac{\q_0}{\|\q_0\|}, \Phi \right) \ge 2 \varepsilon.
        \end{equation}}
        Combining \eqref{81A_closure} and \eqref{dist_to_f0}, we conclude that 
        $$
        \fR_0: =  %\left( {\mathbb S}^{n-1} \cap 
        {\Phi %\right)
         \times {H} }\subseteq  \overline{\fR} \smallsetminus \big( B_{\br_0}(\varepsilon) \cup B_{-\br_0}(\varepsilon)\big) \subseteq \overline{\fR \smallsetminus \big( B_{\br_0}(\varepsilon) \cup B_{-\br_0}(\varepsilon)\big)}.
        $$

        By Proposition \ref{density_any_k} (applied with $k = n-2$) and {Lemma} \ref{closure} {the collections $\mathcal{L}_{\fR_0}$ and} 
        \linebreak
        %the collection 
        $\mathcal{L}_{\fR \smallsetminus \left( B_{\br_0}(\varepsilon) \cup B_{-\br_0}(\varepsilon)\right)}$ are dense. {Hence, by Lemma \ref{usefullemma}}, {condition \eqref{local_dens_condition} holds for any open $U \subseteq M_{m,n}$ intersecting $L = \bfL_{\p_0, \q_0}$}.
        %there exists a  {nonempty} %convex  
    %open set $\Omega \subseteq M_{m,n}$ such that condition \eqref{local_dens_condition} holds for $\bfL = \bfL_{\p_0, \q_0}$. 
    Since the density of %the collection 
        $\mathcal{L}_{\fR}$ has %also
    {already} been established, we conclude that the collection $\mathcal{L}_{\fR} = \mathcal{L}_{P, Q}$ is totally dense. 
\smallskip

       % \item 
      %  item[\rm(b)]  
{Now consider part (b), where      $Q = Q_1\times\cdots\times Q_n\subset \Z\times\cdots\times\Z$ 
such that
 at least two of the sets $Q_i$ are infinite, 
and $P$ is of bounded Hausdorff distance from $\R^m$.
 Assume without loss of generality that $Q_1$ is unbounded  and %without loss of generality 
 $\q_0$  is not proportional to % \neq 
 $(1,0, \ldots, 0)$ (otherwise we can interchange $Q_1$ with $Q_i$ such that $i>1$ and $Q_i$ is unbounded).}
 %is equivalent to $\q_0 = (0, \ldots, 0, q)$ by interchanging $Q_1$ and $Q_n$ \comm{sorry, I don't understand this remark}). 
 Then
        \begin{equation}\label{1st_coord}
        \dist\left( \frac{\q_0}{\|\q_0\|}, (1, 0, \ldots, 0) \right) =: 2 \varepsilon > 0.
        \end{equation}
        Since $Q_1$ is unbounded, we can approximate any element of $\mathbb{R}^m$ with an arbitrarily small error by an element of the form $\frac{\p}{q_1}$, where $\p \in P$ and $q_1 \in Q_1$. Therefore, $\fR_0: = (1, 0, \ldots, 0) \times \mathbb{R}^m \subseteq  \overline{\fR}$. Combining this with \eqref{1st_coord}, we can %claim
        {conclude} that
        $$
        \fR_0 \subseteq  \overline{\fR} \smallsetminus \big( B_{\br_0}(\varepsilon) \cup B_{-\br_0}(\varepsilon)\big) \subseteq \overline{\fR \smallsetminus \big( B_{\br_0}(\varepsilon) \cup B_{-\br_0}(\varepsilon)\big)}.
        $$
        By Corollary \ref{sections_density} and {Lemma} \ref{closure} the {collections $\mathcal{L}_{\fR_0}$ and %the collection 
        $\mathcal{L}_{\fR \smallsetminus \left( B_{\br_0}(\varepsilon) \cup B_{-\br_0}(\varepsilon)\right)}$} are dense. The rest of the proof follows as in the previous case. %\comm{I don't understand: did you use the assumption that $Q_n$ is unbounded?}
%   \end{enumerate}
\end{proof}

\subsection{A totally dense subcollection of a non-totally dense collection}\label{subcoll_td_section}
In {the remaining part of the} section we %will show why our ability to consider collections which are not totally dense, but contain a totally dense subcollections, provides us with new nontrivial examples. We 
construct  {an example of} sets $P$ and $Q$ {that} satisfy {the assumptions of} Theorem \ref{not_on_line} 
%such that 
{yet} {for which} the collection $\mathcal{L}_{P, Q}$ is not totally dense ({and}, according to Theorem \ref{not_on_line}, contains a totally dense subcollection).

We will work in the simplest case $n = 2$, $m=1$. Let $\q_k: = (1, k), \, \q' : = (0, 1)$, and define
$$
Q = \{ \q_k \}_{k=1}^{\infty} \cup \{ \q' \}, \,\,\,\,\,\,\,\, P = \mathbb{Z}.
$$

\begin{proposition}\label{subcoll}
 The collection $\mathcal{L}_{P, Q}$ is not totally dense.
\end{proposition}

\begin{proof}
    Let $\bfL_0: = \{ (x, y): \,\,y = 0 \} \in \mathcal{L}_{P,Q}$. We will show that for any two real numbers $a < b$ the set 
\begin{equation}\label{closure_in_counterex}
        \overline{\bigcup\limits_{(\p, \q) \in P \times Q: \,\,\, \bfL_{\p, \q} \cap \bfL_0 \cap (a, b) \neq \varnothing} \bfL_{\p, \q} }
    \end{equation}
    has empty interior.
    {Assume} that the interior of this set is nonempty, and in addition that this set contains an open subset in the upper half-plane (the lower half-plane case is analogous). The key observation here is that in this case the set \eqref{closure_in_counterex} contains a set $\Delta$ of the form
    $$
    \Delta = \{ (x, y): \,\,\, x_1 < x < x_2, \, \delta < y < 2 \delta \}
    $$
    %as a subset 
    (here $x_1 < x_2$ are two real numbers and {$\delta>0$}).
    % is a positive real parameter). 
    We %want to show 
    {claim} that all but finitely many of the lines $\bfL_{\p, \q}$ intersecting the interval % $(a, b)$ 
    {$\{ (x, y): a<x<b, \ y = 0 \}$} of the line $\bfL_0$ pass strictly under $\Delta$.  

    Indeed: if $k > \frac{x_1 - a}{\delta}$, then any line of the form $\bfL_{\p, \q_k}$ is strictly under the line $\{(x, y): \,\,\, x - a = ky \}$, which is, in turn, strictly under $\Delta$. Lines of the form $\bfL_{\p, \q'}$ are parallel to $\bfL_0$ and thus are %out of consideration
    {need not be considered}. There are only finitely many values $k \leq \frac{x_1 - a}{\delta}$ and for each of them there are only finitely many values $\p \in \mathbb{Z}$ such that $\bfL_{\p, \q_k} \cap (a, b) \neq \varnothing$. Thus the closure of the union of these lines has empty interior, and in particular does not contain $\Delta$, which contradicts our assumption and completes the proof. %\comm{Even though we can use Theorem \ref{not_on_line} to guarantee a totally dense subcollection, why not also  explicitly pick a line to throw away from the collection to get totally density?}
\end{proof}

{\begin{remark} \rm
    In this particular case, we {can not only}  say that $\mathcal{L}_{P,Q}$ contains a totally dense subcollection, but {also} explicitly describe one. Namely, one {can} remove %the point 
    $\q'$ from %the set 
    $Q$,
    or, equivalently, remove all the horizontal lines from consideration. It is not hard to verify that the resulting collection %will be
    is totally dense.
\end{remark}
}

\section{Proof of Lemma \ref{opensetlemma}.}\label{technicalsection}

Recall that we are given  $0 \leq k \leq \min(n-1, m)$ and an  open subset $U$ of ${H} \times \Phi$, where $H$ is an $(m-k)$-dimensional linear subspace of $ \mathbb{R}^m$ and $\Phi$ a $k$-dimensional subsphere    of $\mathbb{S}^{n-1}$; our goal is to prove that the set 
    $\bigcup\limits_{
%(\p, \theta)
\br\in U} \bfL_{
%\p, \theta
\br}$
    has a nonempty interior in $M_{m,n}$.

We start the proof with a few reductions. It is enough to take $U$ to be a product of a small neighborhood of some $\theta_0\in\Phi$   and a small neighborhood of some $\p_0\in H$. Also, it suffices to prove the lemma with   $\Phi$, $\theta_0$ and $H$ being fixed, and the general case will follow from the transitivity of rotations of $\R^m$ and $\R^n$. Finally, instead of a   neighborhood of   $\theta_0\in\Phi$ it will be easier to work with  a  neighborhood of   $\theta_0$ in the tangent space to $\Phi$ at $\theta_0$; this is easily justified by the fact that one is a homeomorphic image of the other.

%\smallskip

    Now %we can work
    let us choose $\theta_0 = (1, 0, \ldots, 0)$ and $\Phi =   \mathbb{S}^{n-1}\cap \{\by\in\R^n:  y_{k+2} = \cdots =y_n = 0\}$. Then the  $\varepsilon$-neighborhood of   $\theta_0$ in the tangent space to $\Phi$ at $\theta_0$ can be parametrized as $\{ R_{\theta} \theta_0: \, \|\theta\| < \varepsilon \}$, where  
    %$R_{\theta}$ is a unipotent matrix given by 
    $$
    \theta = \begin{pmatrix}\theta_2 \\ \vdots \\ \theta_{k+1}\end{pmatrix}\quad\text{ and }\quad R_{\theta} = \begin{pmatrix}
        1 & 0 & 0  \\
        \theta & I_k & 0 \\
        0 & 0 &  I_{n-k-1}
    \end{pmatrix}. 
    $$
    Also let 
    $H = \{(\bx,0, \ldots, 0): \bx\in\R^{m-k}\}$, and fix 
    $\p_0 = (p_1, \ldots, p_{m-k}, 0, \ldots, 0) \in H$. In the computation below we will identify $H$ with  $\R^{m-k}$. 
    Then the  $\varepsilon$-neighborhood of   $\p_0$ in  $H$ can be expressed as $$\p_0 + \big\{ \bx \in H: \|\bx\| < \varepsilon \big\},$$ and, applying \eqref{scaling}, one can write\\
\begin{equation}\label{combinedmatrix}
\begin{aligned}
&\bigcup\limits_{(\p, \theta) \in B_{\varepsilon}(\p_0) \times B_{\varepsilon}(\theta_0)} \bfL_{\p, \theta} = \bigcup\limits_{\|\bx\| < \varepsilon, |\theta| < \varepsilon} \bfL_{\p_0 + \bx, R_{\theta}\theta_0} = \bigcup\limits_{\|\bx\| < \varepsilon, \|\theta\| < \varepsilon} \bfL_{\p_0 + \bx, \theta_0}R_{-\theta}\\
   & = \left\{ \left.\begin{pmatrix}
        p_1 + x_1 - \sum\limits_{j=2}^{k+1} a_{1,j} \theta_j & a_{1,2} & \ldots & a_{1,n} \\
        \ldots & \ldots & \ldots & \ldots \\
        p_{m-k} + x_{m-k} - \sum\limits_{j=2}^{k+1} a_{m-k,j} \theta_j & a_{m-k,2} & \ldots & a_{m-k,n} \\
        - \sum\limits_{j=2}^{k+1} a_{m-k+1,j} \theta_j & a_{m-k+1,2} & \ldots & a_{m-k+1,n} \\
          \ldots & \ldots & \ldots & \ldots \\
           - \sum\limits_{j=2}^{k+1} a_{m,j} \theta_j & a_{m,2} & \ldots & a_{m,n}
    \end{pmatrix}\ \right|\ 
    \begin{aligned}
    \|\bx\|&< \varepsilon, %\,%\, i = 1, \ldots, m-k;
    \\
    \|\theta\| &< \varepsilon%, \,\, j = 2, \ldots, k+1;
    \\
    a_{i,j} &\in \R\\ (i = \ &1, \ldots, m, \\ j = \ &2, \ldots, n)
    \end{aligned}
    \right\}.
\end{aligned}
    \end{equation}
    To prove that this set has nonempty interior for any $\varepsilon > 0$, we will show that the Jacobian of the real analytic map $\Pi: \R^{mn} \rightarrow \R^{mn}$ defined by the matrix from \eqref{combinedmatrix} %has a Jacobian determinant 
    is not identically equal to zero (and therefore the set of regular points %where the Jacobian   is nonzero
    of $\Pi$ is dense in $\R^{mn}$). For convenience, we will use a change of coordinates: let
    $$
    a_{i,1}: = x_i, \,\,\,\, i = 1, \ldots, m-k, \,\,\,\, \text{and} \,\,\,\, a_{i,1}: = \theta_{i-m+k+1},\,\,\,\, i = m-k+1, \ldots, m.
    $$
    Let $\Pi_{s, l}$ be the coordinate function in the $s$-th row and the $l$-th column of the matrix \eqref{combinedmatrix}. With the notation $\delta_{\alpha,\beta}: = \begin{cases}
                1, \,\,\,\,\alpha = \beta \\
                0, \,\,\,\, \alpha \neq \beta
\end{cases}$ it is straightforward to check that 
$$
\frac{\partial \Pi_{s, l}}{\partial a_{i,j}} =
\begin{cases}
    \frac{\partial \Pi_{s, 1}}{\partial x_i} = \delta_{s,i} \quad&\text{if }l = 1,\ j=1,\ i = 1, \ldots, m-k;\\\frac{\partial \Pi_{s, 1}}{\partial \theta_{i-m+k+1}} = a_{s, i-m+k+1}\quad&\text{if }l = 1,\ j=1,\ i = m-k+1, \ldots, m
    ;\\
    - \theta_j \delta_{s,i}\quad&\text{if }l = 1,\ j = 2, \ldots, k+1%,\ i = 1, \ldots, m
    ;\\
    0 \quad&\text{if }l = 1,\ j = k+2, \ldots, n%,\ i = 1, \ldots, m
    ;\\
    0 \quad&\text{if }l = 2, \ldots, n,\ j = 1%,\ i = 1, \ldots, m
    ;\\ \delta_{(s,l),(i,j)}&\text{if }l = 2, \ldots, n,\ j = 2, \ldots, n.
\end{cases}
$$
Then, with the notation
$$A_1 = \begin{pmatrix}
        a_{1,2} & \ldots & a_{1,k+1} \\
        \ldots & \ldots & \ldots \\
        a_{m-k,2} & \ldots & a_{mm-k,k+1}
\end{pmatrix}\quad\text{and}\quad A_2 = \begin{pmatrix}
        a_{m-k+1,2} & \ldots & a_{m-k+1,k+1} \\
        \ldots & \ldots & \ldots \\
        a_{m,2} & \ldots & a_{m,k+1}
    \end{pmatrix},$$
one can see that the Jacobian matrix of $\Pi$ has the form
$$
%J_\Pi = 
\begin{pmatrix}I_{m-k} & 0 & 0\\ -A_1^{\top} & -A_2^{\top}  & 0\\ * & * & I_{m(n-1)}  \end{pmatrix},
$$
and the proof is completed by noting that   the set $\{A\in\mr: A_2$ is invertible$\}$ is open and dense. \qed

%\printbibliography
%\end{document}

\end{document}